\documentclass[12pt]{article}
\usepackage{ct}

\usepackage{tikz}
\usetikzlibrary{arrows.meta}
\usetikzlibrary{calc}
\usepackage{placeins}
\usepackage[labelformat=empty]{subfig}

\specs{6}{4 (1)}{2024}{}

\dateline{Mar 20, 2023}{Dec 29, 2023}{Jun 30, 2024}



\title{An aperiodic monotile}

\author[1]{David Smith}
\author[2]{Joseph Samuel Myers\thanks{Development of software used in
    this work was supported in part by a Senior Rouse Ball Studentship
    for 2002--3 from Trinity College, Cambridge.}}
\author[3]{Craig S. Kaplan}
\author[4]{Chaim Goodman-Strauss} 

\affil[1]{%
Yorkshire, U.K.

\email{ds.orangery@gmail.com}%
}

\affil[2]{%
Cambridge, U.K.

\email{jsm@polyomino.org.uk}%
}

\affil[3]{%
School of Computer Science, University of Waterloo, Waterloo, Ontario, Canada

\email{csk@uwaterloo.ca}%
}

\affil[4]{%
National Museum of Mathematics, New York, New York, U.S.A.

\email{chaimgoodmanstrauss@gmail.com}%
}

\newcommand{\secref}[1]{Section~\ref{#1}}
\newcommand{\fignum     } [1] {\ref{#1}}
\newcommand{\fig        } [1] {Figure~\fignum{#1}}
\graphicspath{{figures/}}
\newcommand{\vcoords}[2]{($sqrt(0.75)*(#1,0)+0.5*(0,#1)+(0,#2)$)}
\newcommand{\hcoords}[2]{($sqrt(12)*(#1,0)+sqrt(3)*(#2,0)+3*(0,#2)+sqrt(3)*(1,0)+(0,1)$)}
\newcommand{\colA}{red}
\newcommand{\colB}{blue}
\newcommand{\colX}{green!80!black}
\newcommand{\colF}{pink}
\newcommand{\colL}{gray}
\newcommand{\colP}{Goldenrod}
\newcommand{\colG}{violet}
\newcommand{\Ttile}[4]{%
  \draw[shift={\vcoords{#2}{#3}},rotate=#1,\colA] \vcoords{0}{0} -- \vcoords{3}{0};
  \fill[shift={\vcoords{#2}{#3}},rotate=#1,\colA] \vcoords{1.25}{0} -- \vcoords{1.25}{0.5} -- \vcoords{1.75}{0} -- cycle;
  \draw[shift={\vcoords{#2}{#3}},rotate=#1,\colA] \vcoords{3}{0} -- \vcoords{0}{3};
  \fill[shift={\vcoords{#2}{#3}},rotate=#1,\colA] \vcoords{1.25}{1.25} -- \vcoords{1.25}{1.75} -- \vcoords{1.75}{1.25} -- cycle;
  \draw[shift={\vcoords{#2}{#3}},rotate=#1,\colB] \vcoords{0}{3} -- \vcoords{0}{0};
  \fill[shift={\vcoords{#2}{#3}},rotate=#1,\colB] \vcoords{0}{1.75} arc[start angle=90,delta angle=180,radius=0.25] -- cycle;
  \draw[shift={\vcoords{#2}{#3}},rotate=#1] \vcoords{1}{1} node {#4}
}
\newcommand{\Htile}[4]{%
  \draw[shift={\vcoords{#2}{#3}},rotate=#1,\colX] \vcoords{0}{0} -- \vcoords{1}{-1};
  \draw[shift={\vcoords{#2}{#3}},rotate=#1,\colX] \vcoords{0.5}{-0.5} -- \vcoords{0.25}{-0.75};
  \draw[shift={\vcoords{#2}{#3}},rotate=#1,\colB] \vcoords{1}{-1} -- \vcoords{4}{-1};
  \fill[shift={\vcoords{#2}{#3}},rotate=#1,\colB] \vcoords{2.75}{-1} arc[start angle=30,delta angle=180,radius=0.25] -- cycle;
  \draw[shift={\vcoords{#2}{#3}},rotate=#1,\colX] \vcoords{4}{-1} -- \vcoords{5}{-1};
  \draw[shift={\vcoords{#2}{#3}},rotate=#1,\colX] \vcoords{4.5}{-1} -- \vcoords{4.25}{-0.5};
  \draw[shift={\vcoords{#2}{#3}},rotate=#1,\colX] \vcoords{5}{-1} -- \vcoords{5}{0};
  \draw[shift={\vcoords{#2}{#3}},rotate=#1,\colX] \vcoords{5}{-0.5} -- \vcoords{5.5}{-0.75};
  \draw[shift={\vcoords{#2}{#3}},rotate=#1,\colB] \vcoords{5}{0} -- \vcoords{2}{3};
  \fill[shift={\vcoords{#2}{#3}},rotate=#1,\colB] \vcoords{3.25}{1.75} arc[start angle=150,delta angle=180,radius=0.25] -- cycle;
  \draw[shift={\vcoords{#2}{#3}},rotate=#1,\colX] \vcoords{2}{3} -- \vcoords{1}{4};
  \draw[shift={\vcoords{#2}{#3}},rotate=#1,\colX] \vcoords{1.5}{3.5} -- \vcoords{1.25}{3.25};
  \draw[shift={\vcoords{#2}{#3}},rotate=#1,\colX] \vcoords{1}{4} -- \vcoords{0}{4};
  \draw[shift={\vcoords{#2}{#3}},rotate=#1,\colX] \vcoords{0.5}{4} -- \vcoords{0.25}{4.5};
  \draw[shift={\vcoords{#2}{#3}},rotate=#1,\colA] \vcoords{0}{4} -- \vcoords{0}{1};
  \fill[shift={\vcoords{#2}{#3}},rotate=#1,\colA] \vcoords{0}{2.75} -- \vcoords{-0.5}{2.75} -- \vcoords{0}{2.25} -- cycle;
  \draw[shift={\vcoords{#2}{#3}},rotate=#1,\colX] \vcoords{0}{1} -- \vcoords{0}{0};
  \draw[shift={\vcoords{#2}{#3}},rotate=#1,\colX] \vcoords{0}{0.5} -- \vcoords{0.5}{0.25};
  \draw[shift={\vcoords{#2}{#3}},rotate=#1] \vcoords{2}{1} node {#4}
}
\newcommand{\Ptile}[4]{%
  \draw[shift={\vcoords{#2}{#3}},rotate=#1,\colL] \vcoords{0}{0} -- \vcoords{1}{-1};
  \draw[shift={\vcoords{#2}{#3}},rotate=#1,\colX] \vcoords{1}{-1} -- \vcoords{2}{-2};
  \draw[shift={\vcoords{#2}{#3}},rotate=#1,\colX] \vcoords{1.5}{-1.5} -- \vcoords{1.75}{-1.25};
  \draw[shift={\vcoords{#2}{#3}},rotate=#1,\colX] \vcoords{2}{-2} -- \vcoords{3}{-2};
  \draw[shift={\vcoords{#2}{#3}},rotate=#1,\colX] \vcoords{2.5}{-2} -- \vcoords{2.75}{-2.5};
  \draw[shift={\vcoords{#2}{#3}},rotate=#1,\colA] \vcoords{3}{-2} -- \vcoords{6}{-2};
  \fill[shift={\vcoords{#2}{#3}},rotate=#1,\colA] \vcoords{4.25}{-2} -- \vcoords{4.25}{-1.5} -- \vcoords{4.75}{-2} -- cycle;
  \draw[shift={\vcoords{#2}{#3}},rotate=#1,\colL] \vcoords{6}{-2} -- \vcoords{5}{-1};
  \draw[shift={\vcoords{#2}{#3}},rotate=#1,\colX] \vcoords{5}{-1} -- \vcoords{4}{0};
  \draw[shift={\vcoords{#2}{#3}},rotate=#1,\colX] \vcoords{4.5}{-0.5} -- \vcoords{4.25}{-0.75};
  \draw[shift={\vcoords{#2}{#3}},rotate=#1,\colX] \vcoords{4}{0} -- \vcoords{3}{0};
  \draw[shift={\vcoords{#2}{#3}},rotate=#1,\colX] \vcoords{3.5}{0} -- \vcoords{3.25}{0.5};
  \draw[shift={\vcoords{#2}{#3}},rotate=#1,\colB] \vcoords{3}{0} -- \vcoords{0}{0};
  \fill[shift={\vcoords{#2}{#3}},rotate=#1,\colB] \vcoords{1.75}{0} arc[start angle=30,delta angle=180,radius=0.25] -- cycle;
  \draw[shift={\vcoords{#2}{#3}},rotate=#1] \vcoords{3}{-1} node {#4}
}
\newcommand{\Ftile}[4]{%
  \draw[shift={\vcoords{#2}{#3}},rotate=#1,\colL] \vcoords{0}{0} -- \vcoords{1}{-1};
  \draw[shift={\vcoords{#2}{#3}},rotate=#1,\colX] \vcoords{1}{-1} -- \vcoords{2}{-2};
  \draw[shift={\vcoords{#2}{#3}},rotate=#1,\colX] \vcoords{1.5}{-1.5} -- \vcoords{1.75}{-1.25};
  \draw[shift={\vcoords{#2}{#3}},rotate=#1,\colX] \vcoords{2}{-2} -- \vcoords{3}{-2};
  \draw[shift={\vcoords{#2}{#3}},rotate=#1,\colX] \vcoords{2.5}{-2} -- \vcoords{2.75}{-2.5};
  \draw[shift={\vcoords{#2}{#3}},rotate=#1,\colL] \vcoords{3}{-2} -- \vcoords{4}{-2};
  \draw[shift={\vcoords{#2}{#3}},rotate=#1,\colX] \vcoords{4}{-2} -- \vcoords{5}{-2};
  \draw[shift={\vcoords{#2}{#3}},rotate=#1,\colX] \vcoords{4.5}{-2} -- \vcoords{4.25}{-1.5};
  \draw[shift={\vcoords{#2}{#3}},rotate=#1,\colF] \vcoords{5}{-2} -- \vcoords{5}{-1};
  \draw[shift={\vcoords{#2}{#3}},rotate=#1,\colF] \vcoords{5}{-1.5} -- \vcoords{5.25}{-1.5};
  \draw[shift={\vcoords{#2}{#3}},rotate=#1,\colF] \vcoords{5}{-1} -- \vcoords{4}{0};
  \draw[shift={\vcoords{#2}{#3}},rotate=#1,\colF] \vcoords{4.5}{-0.5} -- \vcoords{4.5}{-0.75};
  \draw[shift={\vcoords{#2}{#3}},rotate=#1,\colX] \vcoords{4}{0} -- \vcoords{3}{0};
  \draw[shift={\vcoords{#2}{#3}},rotate=#1,\colX] \vcoords{3.5}{0} -- \vcoords{3.25}{0.5};
  \draw[shift={\vcoords{#2}{#3}},rotate=#1,\colB] \vcoords{3}{0} -- \vcoords{0}{0};
  \fill[shift={\vcoords{#2}{#3}},rotate=#1,\colB] \vcoords{1.75}{0} arc[start angle=30,delta angle=180,radius=0.25] -- cycle;
  \draw[shift={\vcoords{#2}{#3}},rotate=#1] \vcoords{3}{-1} node {#4}
}
\newcommand{\Pplus}[4]{%
  \draw[shift={\vcoords{#2}{#3}},rotate=#1,\colL] \vcoords{0}{0} -- \vcoords{1}{-1};
  \draw[shift={\vcoords{#2}{#3}},rotate=#1,\colP] \vcoords{1}{-1} -- \vcoords{5}{-1};
  \fill[shift={\vcoords{#2}{#3}},rotate=#1,\colP] \vcoords{2}{-1} -- \vcoords{2.125}{-1.25} -- \vcoords{4.125}{-1.25} -- \vcoords{4}{-1} -- cycle;
  \draw[shift={\vcoords{#2}{#3}},rotate=#1,\colX] \vcoords{5}{-1} -- \vcoords{4}{0};
  \draw[shift={\vcoords{#2}{#3}},rotate=#1,\colX] \vcoords{4.5}{-0.5} -- \vcoords{4.25}{-0.75};
  \draw[shift={\vcoords{#2}{#3}},rotate=#1,\colX] \vcoords{4}{0} -- \vcoords{3}{0};
  \draw[shift={\vcoords{#2}{#3}},rotate=#1,\colX] \vcoords{3.5}{0} -- \vcoords{3.25}{0.5};
  \draw[shift={\vcoords{#2}{#3}},rotate=#1,\colB] \vcoords{3}{0} -- \vcoords{0}{0};
  \fill[shift={\vcoords{#2}{#3}},rotate=#1,\colB] \vcoords{1.75}{0} arc[start angle=30,delta angle=180,radius=0.25] -- cycle;
  \draw[shift={\vcoords{#2}{#3}},rotate=#1] \vcoords{2.5}{-0.5} node {#4}
}
\newcommand{\Pminus}[4]{%
  \draw[shift={\vcoords{#2}{#3}},rotate=#1,\colX] \vcoords{0}{0} -- \vcoords{1}{-1};
  \draw[shift={\vcoords{#2}{#3}},rotate=#1,\colX] \vcoords{0.5}{-0.5} -- \vcoords{0.75}{-0.25};
  \draw[shift={\vcoords{#2}{#3}},rotate=#1,\colX] \vcoords{1}{-1} -- \vcoords{2}{-1};
  \draw[shift={\vcoords{#2}{#3}},rotate=#1,\colX] \vcoords{1.5}{-1} -- \vcoords{1.75}{-1.5};
  \draw[shift={\vcoords{#2}{#3}},rotate=#1,\colA] \vcoords{2}{-1} -- \vcoords{5}{-1};
  \fill[shift={\vcoords{#2}{#3}},rotate=#1,\colA] \vcoords{3.25}{-1} -- \vcoords{3.25}{-0.5} -- \vcoords{3.75}{-1} -- cycle;
  \draw[shift={\vcoords{#2}{#3}},rotate=#1,\colL] \vcoords{5}{-1} -- \vcoords{4}{0};
  \draw[shift={\vcoords{#2}{#3}},rotate=#1,\colP] \vcoords{4}{0} -- \vcoords{0}{0};
  \fill[shift={\vcoords{#2}{#3}},rotate=#1,\colP] \vcoords{3}{0} -- \vcoords{3.125}{-0.25} -- \vcoords{1.125}{-0.25} -- \vcoords{1}{0} -- cycle;
  \draw[shift={\vcoords{#2}{#3}},rotate=#1] \vcoords{2.5}{-0.5} node {#4}
}
\newcommand{\Fplus}[4]{%
  \draw[shift={\vcoords{#2}{#3}},rotate=#1,\colL] \vcoords{0}{0} -- \vcoords{1}{-1};
  \draw[shift={\vcoords{#2}{#3}},rotate=#1,\colG] \vcoords{1}{-1} -- \vcoords{5}{-1};
  \fill[shift={\vcoords{#2}{#3}},rotate=#1,\colG] \vcoords{2}{-1} -- \vcoords{3.125}{-1.25} -- \vcoords{4}{-1} -- cycle;
  \draw[shift={\vcoords{#2}{#3}},rotate=#1,\colF] \vcoords{5}{-1} -- \vcoords{4}{0};
  \draw[shift={\vcoords{#2}{#3}},rotate=#1,\colF] \vcoords{4.5}{-0.5} -- \vcoords{4.5}{-0.75};
  \draw[shift={\vcoords{#2}{#3}},rotate=#1,\colX] \vcoords{4}{0} -- \vcoords{3}{0};
  \draw[shift={\vcoords{#2}{#3}},rotate=#1,\colX] \vcoords{3.5}{0} -- \vcoords{3.25}{0.5};
  \draw[shift={\vcoords{#2}{#3}},rotate=#1,\colB] \vcoords{3}{0} -- \vcoords{0}{0};
  \fill[shift={\vcoords{#2}{#3}},rotate=#1,\colB] \vcoords{1.75}{0} arc[start angle=30,delta angle=180,radius=0.25] -- cycle;
  \draw[shift={\vcoords{#2}{#3}},rotate=#1] \vcoords{2.5}{-0.5} node {#4}
}
\newcommand{\Fminus}[4]{%
  \draw[shift={\vcoords{#2}{#3}},rotate=#1,\colX] \vcoords{0}{0} -- \vcoords{1}{-1};
  \draw[shift={\vcoords{#2}{#3}},rotate=#1,\colX] \vcoords{0.5}{-0.5} -- \vcoords{0.75}{-0.25};
  \draw[shift={\vcoords{#2}{#3}},rotate=#1,\colX] \vcoords{1}{-1} -- \vcoords{2}{-1};
  \draw[shift={\vcoords{#2}{#3}},rotate=#1,\colX] \vcoords{1.5}{-1} -- \vcoords{1.75}{-1.5};
  \draw[shift={\vcoords{#2}{#3}},rotate=#1,\colL] \vcoords{2}{-1} -- \vcoords{3}{-1};
  \draw[shift={\vcoords{#2}{#3}},rotate=#1,\colX] \vcoords{3}{-1} -- \vcoords{4}{-1};
  \draw[shift={\vcoords{#2}{#3}},rotate=#1,\colX] \vcoords{3.5}{-1} -- \vcoords{3.25}{-0.5};
  \draw[shift={\vcoords{#2}{#3}},rotate=#1,\colF] \vcoords{4}{-1} -- \vcoords{4}{0};
  \draw[shift={\vcoords{#2}{#3}},rotate=#1,\colF] \vcoords{4}{-0.5} -- \vcoords{4.25}{-0.5};
  \draw[shift={\vcoords{#2}{#3}},rotate=#1,\colG] \vcoords{4}{0} -- \vcoords{0}{0};
  \fill[shift={\vcoords{#2}{#3}},rotate=#1,\colG] \vcoords{3}{0} -- \vcoords{2.125}{-0.25} -- \vcoords{1}{0} -- cycle;
  \draw[shift={\vcoords{#2}{#3}},rotate=#1] \vcoords{2.5}{-0.5} node {#4}
}
\newcommand{\markpt}[2]{%
  \filldraw \vcoords{#1}{#2} circle [radius=3pt]
}
\newcommand{\vctxt}[3]{%
  \draw \vcoords{#1}{#2} node {#3}
}
\newcommand{\colfill}{lightgray}
\newcommand{\colcluster}{lime!85!black}
\newcommand{\rawtileA}[6]{%
  #1[#2shift={\hcoords{#4}{#5}},rotate=#3]
    \vcoords{0}{0} -- \vcoords{-2}{1} -- \vcoords{-2}{0} --
    \vcoords{0}{-2} -- \vcoords{1}{-2} -- \vcoords{2}{-4} --
    \vcoords{4}{-5} -- \vcoords{4}{-4} -- \vcoords{3}{-3} --
    \vcoords{4}{-2} -- \vcoords{3}{0} -- \vcoords{2}{0} --
    \vcoords{1}{1} -- cycle;
  \draw[shift={\hcoords{#4}{#5}},rotate=#3] \vcoords{2}{-1} node {#6}
}
\newcommand{\rawtileAr}[6]{%
  #1[#2shift={\hcoords{#4}{#5}},rotate=#3]
    \vcoords{0}{0} -- \vcoords{1}{-2} -- \vcoords{2}{-2} --
    \vcoords{3}{-3} -- \vcoords{4}{-2} -- \vcoords{3}{0} --
    \vcoords{4}{0} -- \vcoords{4}{1} -- \vcoords{2}{2} --
    \vcoords{1}{1} -- \vcoords{0}{2} -- \vcoords{-2}{2} --
    \vcoords{-2}{1} -- cycle;
  \draw[shift={\hcoords{#4}{#5}},rotate=#3] \vcoords{2}{-1} node {#6}
}
\newcommand{\tileA}[4]{\rawtileA{\draw}{}{#1}{#2}{#3}{#4}}
\newcommand{\tileAr}[4]{\rawtileAr{\draw}{}{#1}{#2}{#3}{#4}}
\newcommand{\ftileA}[4]{\rawtileA{\filldraw}{fill=\colfill,}{#1}{#2}{#3}{#4}}

\newcommand{\rawtileB}[6]{%
  #1[#2shift={\hcoords{#4}{#5}},rotate=#3]
    \vcoords{0}{0} -- \vcoords{-2}{1} -- \vcoords{-2}{0} --
    \vcoords{-3}{0} -- \vcoords{-2}{-2} -- \vcoords{2}{-4} --
    \vcoords{3}{-3} -- \vcoords{4}{-4} -- \vcoords{5}{-4} --
    \vcoords{4}{-2} -- \vcoords{2}{-1} -- \vcoords{2}{0} --
    \vcoords{1}{1} -- cycle;
  \draw[shift={\hcoords{#4}{#5}},rotate=#3] \vcoords{0}{-2} node {#6}
}

\newcommand{\tileB}[4]{\rawtileB{\draw}{}{#1}{#2}{#3}{#4}}

\newcommand{\colfillkite}{lime}
\newcommand{\colkite}{lightgray}
\newcommand{\rawkite}[5]{%
  #1[#2shift={\hcoords{#4}{#5}},rotate=#3]
    \vcoords{0}{0} -- \vcoords{-2}{1} -- \vcoords{-2}{0} --
    \vcoords{-1}{-1} -- cycle
}

\newcommand{\ffkite}[3]{\rawkite{\filldraw}{fill=\colfillkite,draw=\colkite,}{#1}{#2}{#3}}
\newcommand{\threekite}[4]{%
  \draw[shift={\hcoords{#2}{#3}},rotate=#1]
    \vcoords{0}{0} -- \vcoords{-6}{3} -- \vcoords{-6}{0} --
    \vcoords{-3}{-3} -- cycle;
  \draw[shift={\hcoords{#2}{#3}},rotate=#1] \vcoords{-3}{0} node {#4}
}
\newcommand{\fftileA}[4]{\rawtileA{\filldraw}{fill=\colfillkite,}{#1}{#2}{#3}{#4}}
\newcommand{\fftileAr}[4]{\rawtileAr{\filldraw}{fill=\colfillkite,}{#1}{#2}{#3}{#4}}
\newcommand{\vcoordsx}[5]{($sqrt(0.75)*(#1,0)+0.5*(0,#1)+(0,#2)+#5*0.5*(#3,0)+#5*0.5*sqrt(1/3)*(0,#3)+#5*sqrt(1/3)*(0,#4)$)}
\newcommand{\tiler}[5]{%
  \draw[shift={\vcoordsx{#2}{#3}{#4}{#5}{#1}}]
  \vcoordsx{0}{0}{0}{0}{#1} -- \vcoordsx{0}{0}{-2}{1}{#1} --
  \vcoordsx{0}{-1}{-2}{1}{#1} -- \vcoordsx{2}{-3}{-2}{1}{#1} --
  \vcoordsx{3}{-3}{-2}{1}{#1} -- \vcoordsx{3}{-3}{-1}{-1}{#1} --
  \vcoordsx{3}{-3}{1}{-2}{#1} -- \vcoordsx{3}{-2}{1}{-2}{#1} --
  \vcoordsx{2}{-1}{1}{-2}{#1} -- \vcoordsx{2}{-1}{2}{-1}{#1} --
  \vcoordsx{2}{-1}{1}{1}{#1} -- \vcoordsx{1}{-1}{1}{1}{#1} --
  \vcoordsx{0}{0}{1}{1}{#1} -- cycle
}
\newcommand{\tilerr}[5]{%
  \draw[shift={\vcoordsx{#2}{#3}{#4}{#5}{#1}},rotate=-30]
  \vcoordsx{0}{0}{0}{0}{#1} -- \vcoordsx{0}{0}{1}{-2}{#1} --
  \vcoordsx{1}{0}{1}{-2}{#1} -- \vcoordsx{2}{-1}{1}{-2}{#1} --
  \vcoordsx{2}{-1}{2}{-1}{#1} -- \vcoordsx{2}{-1}{1}{1}{#1} --
  \vcoordsx{3}{-1}{1}{1}{#1} -- \vcoordsx{3}{0}{1}{1}{#1} --
  \vcoordsx{3}{0}{-1}{2}{#1} -- \vcoordsx{3}{0}{-2}{1}{#1} --
  \vcoordsx{2}{1}{-2}{1}{#1} -- \vcoordsx{0}{1}{-2}{1}{#1} --
  \vcoordsx{0}{0}{-2}{1}{#1} -- cycle
}
\newcommand{\drawsquare}[2]{%
  \draw[shift={(#1,#2)}] (0,0) -- (1,0) -- (1,1) -- (0,1) -- cycle
}
\newcommand{\drawtriangle}[3]{%
  \draw[shift={\vcoords{#2}{#3}},rotate=#1]
  \vcoords{0}{0} -- \vcoords{1}{0} -- \vcoords{0}{1} -- cycle
}
\newcommand{\hexkites}[4]{%
  \draw[shift={\vcoordsx{#1}{#2}{#3}{#4}{3}}]
  \vcoords{0}{0} -- \vcoords{2}{-2} -- \vcoords{4}{-2} --
  \vcoords{4}{0} -- \vcoords{2}{2} -- \vcoords{0}{2} -- cycle;
  \draw[shift={\vcoordsx{#1}{#2}{#3}{#4}{3}}]
  \vcoords{0}{1} -- \vcoords{4}{-1};
  \draw[shift={\vcoordsx{#1}{#2}{#3}{#4}{3}}]
  \vcoords{1}{-1} -- \vcoords{3}{1};
  \draw[shift={\vcoordsx{#1}{#2}{#3}{#4}{3}}]
  \vcoords{3}{-2} -- \vcoords{1}{2}
}
\newcommand{\trikites}[5]{%
  \draw[shift={\vcoordsx{#2}{#3}{#4}{#5}{3}},rotate=#1]
  \vcoordsx{0}{0}{0}{0}{3} -- \vcoordsx{0}{0}{2}{0}{3} --
  \vcoordsx{0}{0}{0}{2}{3} -- cycle;
  \draw[shift={\vcoordsx{#2}{#3}{#4}{#5}{3}},rotate=#1]
  \vcoordsx{0}{0}{1}{0}{3} -- \vcoordsx{0}{0}{2}{2}{1};
  \draw[shift={\vcoordsx{#2}{#3}{#4}{#5}{3}},rotate=#1]
  \vcoordsx{0}{0}{1}{1}{3} -- \vcoordsx{0}{0}{2}{2}{1};
  \draw[shift={\vcoordsx{#2}{#3}{#4}{#5}{3}},rotate=#1]
  \vcoordsx{0}{0}{0}{1}{3} -- \vcoordsx{0}{0}{2}{2}{1}
}
\newcommand{\colextside}{lime!85!black}
\definecolor{colroot}{rgb}{0.86, 0.38, 0.03}

\extrafloats{100}
\begin{document}

\maketitle


\begin{abstract}
  A longstanding open problem asks for an aperiodic monotile, also known
  as an ``einstein'': a shape 
  that admits tilings of the plane, but never periodic tilings.
  We answer this problem for topological disk tiles by exhibiting a continuum of 
  combinatorially equivalent aperiodic polygons.
  We first show that a representative example, the ``hat''
  polykite, can form clusters called ``metatiles'', for which substitution
  rules can be defined.  Because the metatiles admit tilings of
  the plane, so too does the hat.  We then prove that generic members
  of our continuum of polygons are aperiodic, through a new kind of
  geometric incommensurability argument.  Separately, we give a combinatorial,
  computer-assisted proof that the hat must form hierarchical---and hence
  aperiodic---tilings.  
\end{abstract}



\section{Introduction}
\label{sec:intro}

\begin{figure}[ht]
\begin{center}
	\includegraphics[width=.8\textwidth]{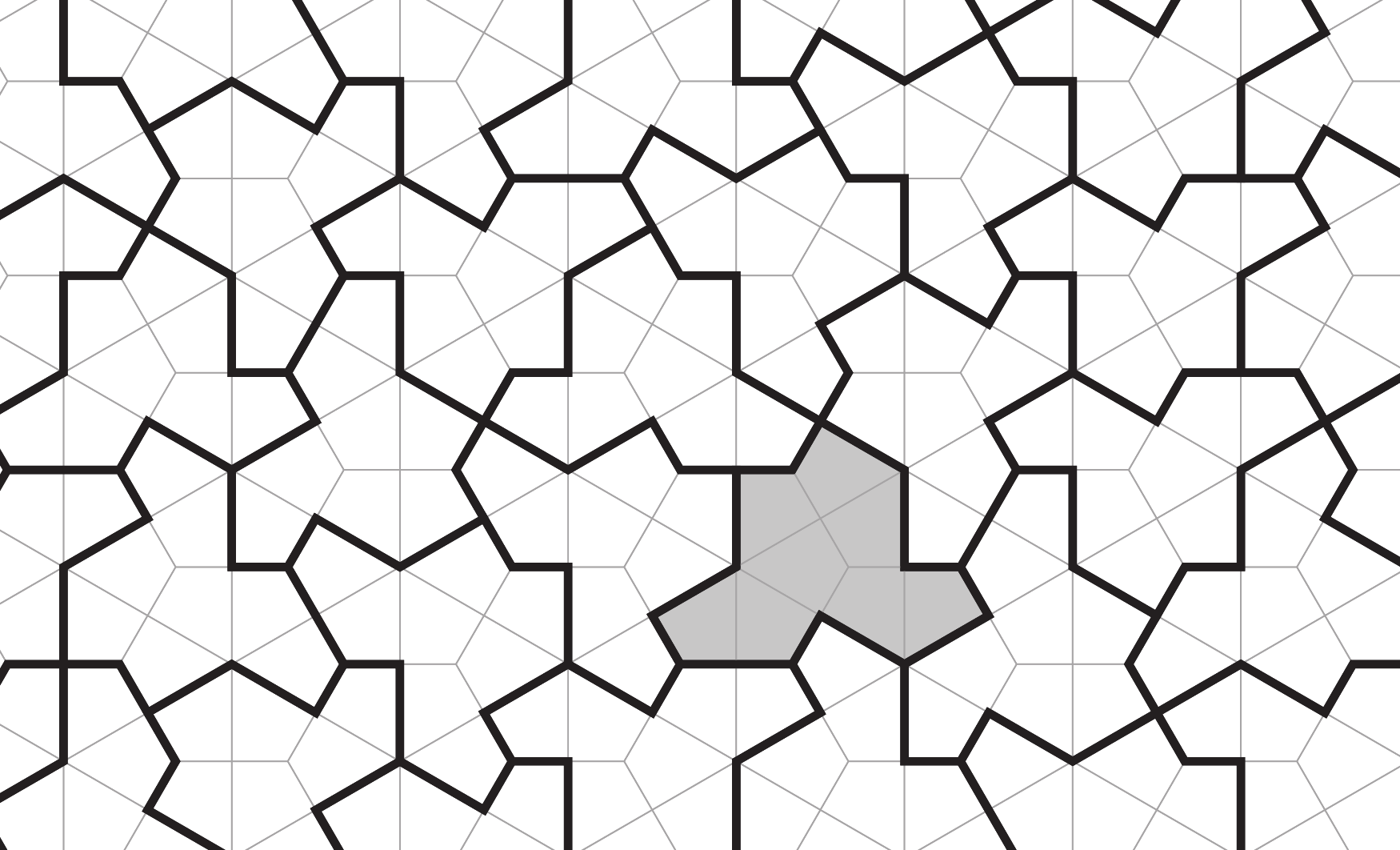}
\end{center}
\caption{\label{fig:polykite} The grey ``hat'' polykite tile
is an aperiodic monotile, also known as an ``einstein''.
Copies of this tile may be assembled into tilings of the plane
(the tile ``admits'' tilings), but none of those tilings can have 
translational symmetry.  In fact, the hat admits uncountably many tilings. In Sections~\ref{sec:substitution},~\ref{sec:clusters}, and~\ref{sec:subst} we describe how these tilings all arise from substitution rules, and thus all have the same local structure.}
\end{figure}

Given a set of two-dimensional tiles,  the nature of the planar
tilings that they admit arises from a deep interaction between the
local and the global.  Constraints on the ways that two neighbouring
tiles interlock can reverberate through the global structure of a
tiling at every scale.  Local constraints
encoded in a set of tiles determine the larger space of tilings they admit in
subtle ways.

{\em Aperiodic} sets of tiles walk a fine line between order and disorder, admitting tilings, but only  those without the simple repetition of
translational symmetry. 
Their study dates to Wang's work on the then remaining open cases of Hilbert's {\it Entscheidungsproblem}~\cite{Wang}. Wang encoded logical fragments by what are now known as \emph{Wang tiles}---congruent squares with coloured edges---to
be tiled by translation only with colours matching on adjoining edges.
He conjectured that every set of Wang tiles that admits a tiling (possibly
using only a subset of the tiles) must also
admit a periodic tiling, and showed that this would imply the decidability
of the \textit{tiling problem} (or \textit{domino problem}): the question
of whether a given set of Wang tiles admits any tilings at all.
The algorithm would consist of enumerating, for each positive integer $n$,
the finite set of all legal~$n\times n$ blocks of tiles.
If there is no tiling by the tiles, there must be some~$n$
for which no such block exists (by the
Extension Theorem~\cite[Theorem~3.8.1]{GS}, which ultimately depends on the
compactness of spaces of patches), and we will eventually encounter the
smallest such~$n$.
On the other hand, if there is a fundamental
domain for a periodic tiling, we will eventually discover it in  
a block.  If Wang's conjecture held and aperiodic sets of tiles did not 
exist, this algorithm would always terminate.

Berger~\cite{Berger} then showed that it was undecidable whether a set
of Wang tiles admits a tiling of the plane. He 
constructed the first aperiodic set of  $20426$~Wang tiles, which he used 
as a kind of scaffolding for encoding finite but unbounded runs of arbitrary computation.

Subsequent decades have spawned a rich literature on aperiodic tiling, touching many  different mathematical and scientific settings; we do not attempt a broad survey here. Yet there remain remarkably few really distinct methods of proving aperiodicity in the plane, despite or due to the underlying undecidability of the tiling problem. 

Berger's initial set comprised thousands of tiles, naturally prompting
the question of how small a set of tiles could be while still forcing
aperiodicity.   
Professional and amateur mathematicians produced successively smaller aperiodic sets, culminating in discoveries by Penrose~\cite{Penrose}
and others of several consisting of just two tiles.  Surveys of these 
sets appear in Chapters~10 and~11 of Gr\"unbaum and Shephard~\cite{GS} and
in an account of the Trilobite and Cross tiles~\cite{ChaimTC}.
A recent table appears in the work of Greenfeld and Tao~\cite{GT1}, counting tiles by translation classes (tiles in different orientations are counted as distinct).

The obvious conclusion of this reduction in size would be to arrive
at an \textit{aperiodic monotile}, a single shape that can form tilings (is a monotile) but can only form non-periodic ones (is aperiodic).  Such a shape is also sometimes referred to as an
``einstein'' (a pun from the German ``ein stein'', roughly ``one
shape'', popularized by Danzer).  In the present article we reserve
these terms for two-dimensional closed topological disks that tile
aperiodically purely by virtue of their geometry, without the need
for any kind of non-geometric matching rules that further constrain
tile adjacencies.  It has long been an open question whether such
a tile exists.  Can one tile embody enough complexity
to forcibly disrupt  periodic order at all scales?

\subsection{The search for an einstein}
\label{sec:search}

Several candidate tiles have been proposed as einsteins, but they all challenge in some way the concepts
of ``tile'', ``tiling'', or ``aperiodic''.

Gummelt~\cite{Gummelt} and Jeong and Steinhardt~\cite{SteinhardtJeong,JeongSteinhardt}
describe a single regular decagon that can cover the plane with copies that are allowed to overlap by prescribed rules, but only non-periodically, in a manner tightly coupled to the Penrose tiling. 
Senechal~\cite{Senechalpersonalcommunication} similarly describes simple rules that allow copies of the Penrose dart to overlap and cover the plane, but never periodically. The result is an ingenious route to aperiodicity, but not a 
tiling in the usual sense.

Tiles are often endowed with \textit{matching rules} that constrain
their placement.  Matching rules have taken a variety of different forms
in the literature.  They sometimes act as a symbolic proxy for neighbour
relationships that could easily be encoded geometrically, but they can 
also determine more complex relationships between tiles.
The Taylor--Socolar tile~\cite{ST1} is a regular hexagon with matching
rules in the form of markings in the interiors of tiles.  The matching rules
force aperiodicity, but they require non-adjacent tiles to exchange
information.  As a result, it is impossible to reduce the behaviour of
the tile to the shape of a closed two-dimensional topological
disk.  The matching rules can be expressed
purely geometrically, but doing so requires either a disconnected tile, a
 tile with cutpoints, or a three-dimensional shape that aperiodically tiles a thickened plane ${\mathbb R}^2\times [0,1]$~\cite{ST2}. 

The structure of the Taylor--Socolar
tiling is closely related to Penrose's $1+\epsilon+\epsilon^2$
tiling \cite{penrose_epsilon, baakegahlergrimm2012,taylornotes}. Like the Trilobite and Crab tiles~\cite{ChaimTC}, these can be adjusted so that an arbitrarily high fraction of the area lies in copies of just one kind of tile. 
But no matter how thin or small they become, the other tiles remain necessary.

Loosely speaking, it is often possible to shift the complexity in
a construction from the tiles to the matching rules or vice versa.
For example, if we use a finite atlas of finite configurations as
our allowed matching rules, even the lowly $2\times1$ rectangle is
an aperiodic monotile!\footnote{Beginning with an aperiodic set of
tiles with, say, geometric matching rules, pixelate pictures of the
tiles and how they fit together, in some black and white bit-map.
Take an atlas of these pictures, splitting black pixels vertically
and white ones horizontally into identical rectangles. The rectangle
is an aperiodic monotile with this atlas of matching rules.} Walton
and Whittaker recently described a hexagonal tile that, like the
Taylor--Socolar tile, achieves aperiodicity via a system of
markings~\cite{WW21}.  These ``orientational'' rules are edge-to-edge,
in that they only constrain a tile's relationships to its immediate
neighbours.  However, this tile's behaviour also cannot be expressed
as pure geometry.

Moving to higher dimensional space permits richer forms of aperiodicity
to arise.  The Schmitt--Conway--Danzer tile~\cite[Section 7.2]{Senechal}
tiles $\mathbb{R}^3$, with tilings that never have translations as
symmetries; none of its tilings have compact fundamental domains.
However, the tile does admit a tiling whose symmetry group contains a screw
motion, and hence an infinite cyclic subgroup of screw motions.  We
refer to such a tile as \emph{weakly aperiodic}.
The ``weak'' label is appropriate, as such tiles 
appear readily in the
hyperbolic plane and other non-amenable spaces.  As early as 1974,
B\"or\"oczky exhibited a weakly aperiodic monotile in the hyperbolic
plane~\cite{Boroczsky}, the elegantly simple basis of the ``binary
tilings''~\cite{BlockWeinberger,regprod,MargulisMozes,Mozes97}.

Following Mozes~\cite{Mozes97}, we say a set of tiles is \emph{strongly
aperiodic} if it admits tilings but none with any infinite cyclic
symmetry.  In the Euclidean plane, a set of ``normal'' tiles is weakly 
aperiodic if and only if it is strongly aperiodic~\cite[Theorem~3.7.1]{GS},
leaving us with a single notion of aperiodicity there.

 Recently, Greenfeld and Tao~\cite{GT2} showed that for a  sufficiently 
high number $n$ of dimensions, a single tile, tiling \emph{only by translation}, can be
aperiodic in~$\mathbb{Z}^n$ (and thus in~$\mathbb{R}^n$); Greenfeld and
Kolountzakis~\cite{greenfeld2023tiling} strengthened this result by showing
that the tile can be connected.
Greenfeld and Tao also showed that it is undecidable whether a single
tile, again tiling by translation, admits a tiling of a periodic
subset of $\mathbb{Z}^2 \times G$ for some nonabelian
group~$G$~\cite{GT1}, and subsequently proved this for tiling a
periodic subset of $\mathbb{Z}^n$ (where $n$~is one of the inputs to
the decision problem and not fixed)~\cite{greenfeld2023undecidability}.
Translational aperiodicity is known to be impossible in~$\mathbb{R}^2$.
Kenyon~\cite{Kenyon,Kenyonerratum,Kenyon2}, building on the work of
Girault-Beauquier and Nivat~\cite{GiraultBeauquierNivat}, showed that 
any topological disk that admits a tiling by translation also admits a
periodic tiling.  Bhattacharya~\cite{Bhattacharya} showed the
same for any finite set in~$\mathbb{Z}^2$.

Little is known about limits on what sorts of shapes could potentially be
aperiodic monotiles.  Rao~\cite{Rao} showed through a computer
search that the list of 15 known families of convex pentagons that tile the
plane is complete, thereby eliminating any remaining possibility
that a convex polygon could be an einstein.
Jeandel and Rao~\cite{JeandelRao} showed that the smallest aperiodic set
of Wang tiles is of size~$11$.

Even when a single tile admits periodic tilings, that periodicity
may be more or less abstruse, in a way that offers tantalizing hints
about aperiodicity.  The \emph{isohedral number} of a tile is the
minimum number of transitivity classes in any tiling it admits; a
tile is \emph{anisohedral} if its isohedral number is greater than
one.  The second part of Hilbert's 18th problem~\cite{Hilbert} asked
whether there exist anisohedral polyhedra in $\mathbb{R}^3$.
Gr\"unbaum and Shephard suggest \cite[Section~9.6]{GS} that this
question was asked in $\mathbb{R}^3$ because Hilbert assumed that
no such tiles exist in the plane.  But Reinhardt~\cite{Reinhardt}
found an example of such a polyhedron, and Heesch~\cite{Heesch}
then gave an example of such a tile in the plane.  Many anisohedral
prototiles are known today. The computer enumeration by Myers~\cite{Myers}
furnished numerous anisohedral polyominoes, polyhexes, and polyiamonds,
including a record-holding $16$-hex that tiles with a minimum of
ten transitivity classes.  It is unknown whether there is an upper
bound on isohedral numbers of monotiles.\footnote{The problem of
determining whether or not a given set of tiles admits a periodic
tiling is also undecidable, at least for larger sets of
tiles~\cite{Gurevich}. If we enumerate sets of tiles, and define
$I(n)$ to be the isohedral number of the $n$th set if it admits a
periodic tiling, and $-1$ otherwise, then $I(n)$ cannot
be bounded by any computable function. This defies our imagination.}

Related insights can be gleaned from the study of shapes that do not tile
the plane.  A tile's \emph{Heesch number} is the largest possible 
combinatorial radius of any patch formed by copies of the shape (or
equivalently, the maximum number of complete concentric rings that can 
be constructed around it).  A shape that tiles the plane is said to have a
Heesch number of $\infty$.  Heesch first exhibited a shape with Heesch
number~1, and a few isolated examples with Heesch numbers up to~3 were
discovered thereafter~\cite{Mann2004}.  Mann and Thomas discovered marked
polyforms with Heesch numbers up to~5 through a brute-force computer 
search~\cite{MT2016}.  Kaplan conducted a search on unmarked 
polyforms~\cite{Kaplan}, yielding examples with Heesch numbers up to~4.
Ba{\v{s}}i{\'c} discovered the current record holder, a shape with Heesch
number~6~\cite{Basic2021}.  \emph{Heesch's problem} asks which 
positive integers can be Heesch numbers; beyond specific examples with
Heesch numbers up to~6, nothing is known about the solution.
An upper bound on finite Heesch numbers would imply the decidability of 
the tiling problem for a single shape. The algorithm would simply consist of
generating all possible concentric rings around a central tile; eventually
one will either fail (in which case the shape does not tile the plane) or
exceed the upper bound on Heesch numbers (in which case it must tile the plane).

\subsection{Main result}

\begin{figure}[ht!]
\begin{center}
\begin{tikzpicture}[x=1cm,y=1cm]
  \node[draw,text width=4cm] at (0,10.3) {Polykites with periodic tilings
    have aligned periodic tilings (Lemma~\ref{lemma:polykitealign})};
  \node[draw,text width=4cm] at (0,7.3) {Polyforms with aligned weakly
    periodic tilings have aligned strongly periodic tilings (similar
    to \cite[Theorem~3.7.1]{GS})};
  \node[draw,text width=4cm] at (0,4.3) {The hat polykite does not have
    aligned strongly periodic tilings (\secref{sec:coupling})};
  \node[draw,text width=4cm] at (0,2.3) {Clusters of hat polykites can form
    metatiles (\secref{sec:discussion})};
  \node[draw,text width=4cm] at (0,0) {Metatiles have a substitution
    system forming combinatorially equivalent supertiles
    (Sections \ref{sec:discussion} and~\ref{sec:subst})};
  \node[draw,text width=4cm] at (5.5,0) {Metatiles tile the plane};
  \node[draw,text width=4cm] at (5.5,2) {Hat polykites tile the plane};
  \node[draw,text width=4cm,ultra thick] at (5.5,4) {The hat polykite is strongly
    aperiodic};
  \node[draw,text width=4cm,ultra thick] at (5.5,6.5) {All $\mathrm{Tile}(a, b)$ for
    positive $a \ne b$ are strongly aperiodic};
  \node[draw,text width=4cm] at (5.5,10) {Tilings by $\mathrm{Tile}(a,b)$ are
    combinatorially equivalent to those by the hat polykite for
    positive $a \ne b$ (\secref{sec:family})};
  \draw[arrows={->[length=2mm]}] (2.2,0) -- (3.3,0);
  \draw[arrows={->[length=2mm]}] (2.2,2) -- (3.3,2);
  \draw[arrows={->[length=2mm]}] (2.2,4) -- (3.3,4);
  \draw[arrows={->[length=2mm]}] (2.2,7.3) -- (3.3,4.6);
  \draw[arrows={->[length=2mm]}] (2.2,10.3) -- (3.5,4.6);
  \draw[arrows={->[length=2mm]}] (5.3,0.4) -- (5.3,1.4);
  \draw[arrows={->[length=2mm]}] (5.3,2.6) -- (5.3,3.4);
  \draw[arrows={->[length=2mm]}] (5.3,4.6) -- (5.3,5.6);
  \draw[arrows={->[length=2mm]}] (5.3,8.3) -- (5.3,7.4);
\end{tikzpicture}
\end{center}
\caption{The high-level structure of the first proof of aperiodicity in
  this paper}
\label{fig:proofstructure}
\end{figure}
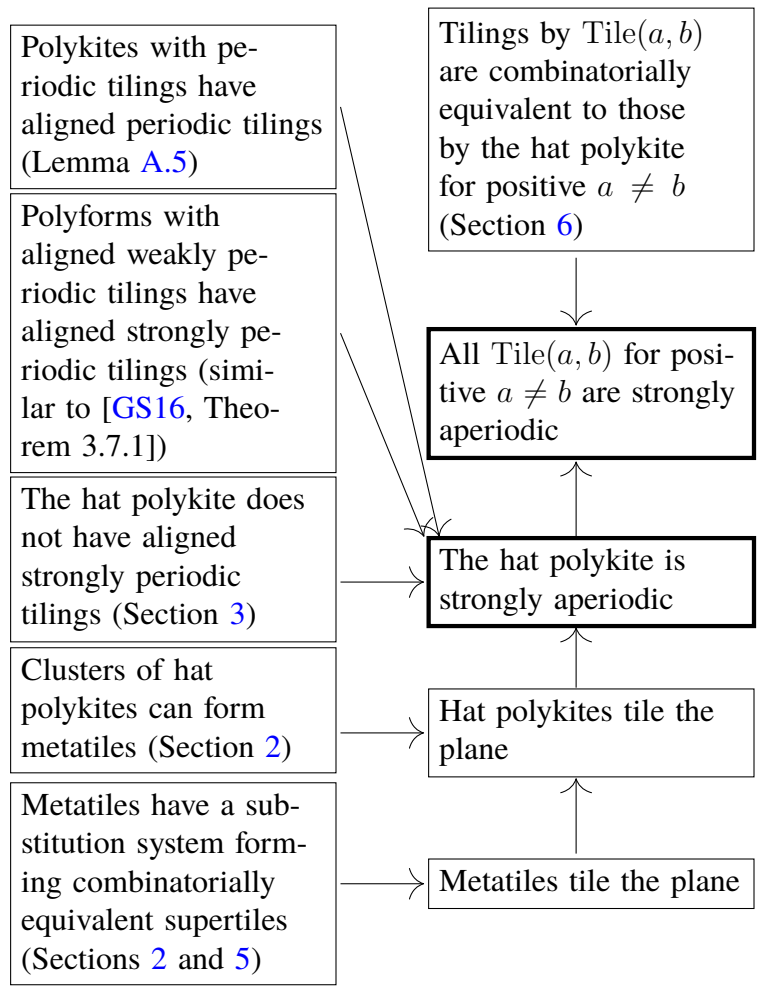

\begin{figure}[ht!]
\begin{center}
\begin{tikzpicture}[x=1cm,y=1cm]
  \node[draw,text width=4cm] at (0,6.4) {Polykites with periodic tilings
    have aligned periodic tilings (Lemma~\ref{lemma:polykitealign})};
  \node[draw,text width=4cm] at (0,3.3) {Clusters of hat polykites
    must form metatiles, adjoining in accordance with matching rules
    (\secref{sec:clusters} and Appendix~\ref{sec:patches})};
  \node[draw,text width=4cm] at (0,-0.3) {Metatiles must follow a
    substitution system forming combinatorially equivalent supertiles
    (Sections \ref{sec:discussion} and~\ref{sec:subst})};
  \node[draw,text width=4cm] at (5.5,0) {The metatiles are strongly aperiodic};
  \node[draw,text width=4cm,ultra thick] at (5.5,2) {The hat polykite is strongly
    aperiodic};
  \node[draw,text width=4cm,ultra thick] at (5.5,4.5) {All $\mathrm{Tile}(a, b)$ for
    positive $a \ne b$ are strongly aperiodic};
  \node[draw,text width=4cm] at (5.5,8) {Tilings by $\mathrm{Tile}(a,b)$ are
    combinatorially equivalent to those by the hat polykite for
    positive $a \ne b$ (\secref{sec:family})};
  \draw[arrows={->[length=2mm]}] (2.2,0) -- (3.3,0);
  \draw[arrows={->[length=2mm]}] (2.2,2) -- (3.3,2);
  \draw[arrows={->[length=2mm]}] (2.2,6.3) -- (3.3,2.6);
  \draw[arrows={->[length=2mm]}] (5.3,0.6) -- (5.3,1.4);
  \draw[arrows={->[length=2mm]}] (5.3,2.6) -- (5.3,3.6);
  \draw[arrows={->[length=2mm]}] (5.3,6.3) -- (5.3,5.4);
\end{tikzpicture}
\end{center}
\caption{The high-level structure of the second proof of aperiodicity in
  this paper}
\label{fig:proofstructure2}
\end{figure}
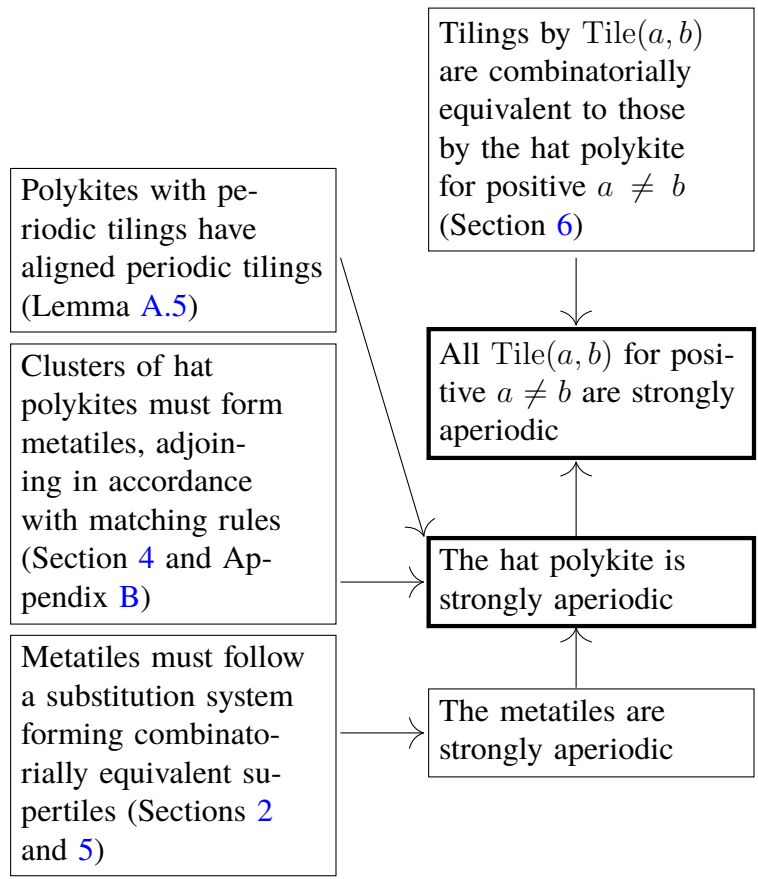

In this paper, we prove the following:  

\begin{theorem}
\label{thm:main}
The shape shown shaded in \fig{fig:polykite}, a polykite that we call
the ``hat'', is an aperiodic monotile.
\end{theorem}

The shape is almost mundane in its simplicity. It is a \textit{polykite}:
the union of eight kites in the Laves tiling $[3.4.6.4]$ (drawn in 
thin lines in \fig{fig:polykite}), the dual
to the $(3.4.6.4)$ Archimedean tiling.  No special qualifications
or additional matching rules are required: as shown, this shape
tiles the plane, but never with any translational symmetries.

We provide two different proofs of aperiodicity, both with novel
aspects.  The first proof follows the structure shown in
\fig{fig:proofstructure}, centred on a 
new approach in \secref{sec:coupling} for proving aperiodicity in the plane.
We observe that
any tiling by the hat corresponds to tilings by two different
polyiamonds, one with two thirds the area of the other.  If there were a
strongly periodic tiling by the hat, the other two tilings would also
be strongly periodic. We prove that if so, the lattices of
translations in the polyiamond tilings would necessarily be related by a
similarity; but
no similarity between lattices of translations on the regular
triangular tiling can have the scale factor~$\sqrt{2}$ required by the
ratio of the areas.  This argument does not show that a tiling exists, and
must be combined with an explicit construction of a tiling (outlined
in \secref{sec:discussion} and given in detail in Sections
\ref{sec:clusters} and~\ref{sec:subst}) to complete the proof of
aperiodicity.

The second proof presented (but the first one found) follows the structure
shown in \fig{fig:proofstructure2}.  Here we generally adhere to
Berger's approach, but we must begin with a novel step
to get to the point where such a proof is possible.  
We first show that in any tiling by
the hat polykite, every tile belongs uniquely to one of four distinct
clusters called \textit{metatiles} (\secref{sec:clusters}), which inherit
matching rules from the geometry of the hats that make them up.
The metatiles abstract away the details of individual hats, and support
a standard style of hierarchical construction.
We then proceed with a Berger-style inductive proof of non-periodicity
in \secref{sec:subst}.  We show that any
tile in any tiling by these four metatiles lies in a unique hierarchy
of \textit{supertiles}---effectively combinatorial copies
of the metatiles---at larger and larger scales. The proof is
constructive.  We show that every metatile belongs uniquely to a 
level-$1$ supertile, and that these supertiles 
have the same combinatorial structure as the metatiles.  The 
level-$1$ supertiles must therefore lie uniquely within 
level-$2$ supertiles with the same combinatorics, and so on
for subsequent levels.
This construction proves that a tiling by copies of the
metatiles must be non-periodic, because if it contained a
translational symmetry, then these hierarchies of supertiles could not be
unique.\footnote{In particular, if a tiling had a translational
symmetry, then for sufficiently large $k$ there would exist a level-$k$
supertile that overlaps its image under this translation.
Any metatile in the intersection of these two supertiles would then 
lie within both of their infinite hierarchies, contradicting the
supposed uniqueness of those hierarchies~\cite[Theorem 10.1.1]{GS}.}
It also shows that the metatiles (and hence the hats) admit 
tilings of the plane, because we
construct clusters of arbitrary size~\cite[Theorem 3.8.1]{GS}.
We are not aware of past work that uses a metatile-like construction
as an intermediate stage towards a proof of aperiodicity.

Because of the combinatorial complexity of the hat polykite, 
a significant fraction of our second proof relies on exhaustive enumeration
of cases, which we carried out and cross-checked with two 
independent software implementations developed by two of the authors
in isolation.  These calculations are necessarily ad hoc, and are essentially
unenlightening.  This case
analysis is only needed to show that all tilings follow the
substitution structure; it is not needed for showing that a tiling
exists, and thus is not needed to show that the tile is aperiodic,
given the proof in \secref{sec:coupling} that no periodic tiling exists.

In Section~\ref{sec:clusters} we learn that every tiling by hats
necessarily contains a mixture of reflected and unreflected tiles.
Thus the hat's status as a monotile depends on
whether one considers a shape and its reflection to be congruent.
By longstanding tradition in the tiling literature (indeed, going
back to Euclid’s \textit{Elements}), shapes are considered congruent
if they are equivalent under any Euclidean isometry, including those
that reverse orientation.  The hat is therefore rightly considered a monotile.
 Still, this potential caveat emphasizes the importance of considering the
setting in which a tiling problem is defined: the geometric space
in which we are working, conditions on the tiles and their matching
rules, and the specific families of isometries that we are allowed
to use.  The diversity of ideas discussed in \secref{sec:search}
illustrates how context can colour the problem of aperiodicity.
We revisit the question of tiling aperiodically without
reflections in \secref{sec:conclusion}.

We close this introduction with definitions of the essential terminology
we will need for the rest of the article.
In \secref{sec:discussion},
we then present a compendium of provisional observations about this polykite,
including an explicit construction of a tiling and aspects of its structure
that deserve further study.  Our two proofs of aperiodicity follow:
we show that there are no periodic tilings (\secref{sec:coupling}),
then that tiles must group into clusters that define metatiles
equipped with matching rules (\secref{sec:clusters}),
and finally that metatiles must compose into
supertiles with combinatorially equivalent matching rules
(\secref{sec:subst}).
In \secref{sec:family}, we offer additional remarks about the continuum
of tiles that contains the hat polykite.  As noted there,
computer search shows that the hat is the smallest aperiodic polykite.

\subsection{Terminology}
\label{sec:terminology}

Terminology used for tilings generally follows that of
Gr\"unbaum and Shephard~\cite{GS}.

A \emph{tile} in a metric space is a closed set of points from that
space. A \emph{tiling} by a set of tiles is a collection of images of
tiles from that set under isometries, the interiors of which are
pairwise disjoint and the union of which is the whole space; we
say a set of tiles \emph{admits} the tiling, or in the case of a single tile
that it admits the tiling.  For most purposes, it is convenient for
tiles to be nonempty compact sets that are the closures of their
interiors; the tiles considered here are polygons, or more generally
closed topological disks.  A
tiling is \emph{monohedral} if all its tiles are congruent (where
congruences can incorporate mirror reflections).  All
tilings considered here are also \emph{locally finite}: every circular
disk meets only finitely many tiles. Every
monohedral plane tiling by closed topological disks is locally
finite.

In any locally finite tiling of the plane by closed topological disks,
the connected components of the intersection of two or more tiles are
isolated points, which are called \emph{vertices} of the tiling, and
Jordan arcs, which are called \emph{edges} of the tiling, and the
boundary of any tile is divided into finitely many edges, alternating
with vertices. Each edge lies on the boundary of exactly two tiles,
which we refer to as lying on opposite sides of the edge.  Two
distinct tiles are \emph{neighbours} if they share any point of their
boundaries, and \emph{adjacents} if they share an edge.

When a (closed topological disk) tile has a polygonal boundary, we
refer to it as having \emph{sides} (maximal straight line segments
lying on that boundary) and \emph{corners} (between two sides), to
distinguish these features from the edges and vertices of a tiling.  We rely on
context to distinguish the meanings of ``side'' as referring to sides
of a polygon or the two sides of an edge of a tiling.  A tiling by
polygons is \emph{edge-to-edge} if the corners and sides of the
polygons coincide with the vertices and edges of the tiling.

A \textit{patch} of tiles is a collection of non-overlapping tiles whose
union is a topological disk.  More specifically, a \textit{$0$-patch}
is a patch containing a single tile, and an \textit{$(n+1)$-patch} is 
a patch formed from the union of an $n$-patch $P$ and a set $S$ of
additional tiles, so that $P$ lies in the interior of the patch and no
proper subset of~$S$ yields a patch with $P$ in its interior.  (In
other words, an $n$-patch is
a tile surrounded by $n$ concentric rings of tiles.)  Every tile in a 
fixed tiling generates an $n$-patch for all finite $n$, by recursively
constructing an $(n-1)$-patch and adjoining all its neighbours in the
tiling, along with any other tiles required to fill in holes left by
adding neighbours.

Given a tiling~$\mathcal{T}$\!, a \emph{poly-$\mathcal{T}$-tile} is a
closed topological disk that is the union of finitely many tiles
from~$\mathcal{T}$; in other words, it is the union of the tiles in
a patch within~$\mathcal{T}$. Poly-$\mathcal{T}$-tiles are also referred to
generically as \emph{polyforms}.  Poly-$\mathcal{T}$-tiles may also
be defined so that they are permitted to have holes.
Because we are mainly concerned with tiles
that admit monohedral tilings, it is not generally significant for the
purposes of this paper whether shapes with holes are allowed or not.

The \emph{symmetry group} of a tiling is the group of those isometries
that act as a permutation on the tiles of the tiling.  A tiling is
\emph{weakly periodic} if its symmetry group has an element of
infinite order; in the plane, this means it includes a nonzero
translation.\footnote{Some authors such as Greenfeld and
Tao~\cite{GT0} have used the term ``weakly periodic'' to refer to a
tiling that is a finite union of sets of tiles, each of which is
weakly periodic in the sense used here.}  A tiling is \emph{strongly
periodic} if the symmetry
group has a discrete subgroup with cocompact action on the space
tiled. In Euclidean space, all strongly periodic
tilings are also weakly periodic.  A set of tiles (or a single tile)
is \emph{weakly aperiodic} if it admits a tiling but does not admit a
strongly periodic tiling, and \emph{strongly aperiodic} if it admits a
tiling but does not admit a weakly periodic tiling.

Any finite set of polygons in the plane that admits a weakly periodic
edge-to-edge tiling also admits a strongly periodic
tiling~\cite[Theorem~3.7.1]{GS}. A similar but simpler argument
shows the same to be the case for a finite set of
poly-$\mathcal{T}$-tiles where $\mathcal{T}$ is itself a strongly
periodic tiling and the weakly periodic tiling consists of copies 
of the
tiles all aligned to the same underlying copy of~$\mathcal{T}$,
instead of being edge-to-edge.  Thus in such contexts it is not
necessary to distinguish weak and strong aperiodicity and we refer to
tiles and sets of tiles simply as \emph{aperiodic}.

A \emph{uniform tiling}~\cite[Section~2.1]{GS} is an edge-to-edge
tiling by regular polygons with a vertex-transitive symmetry group.
In the Euclidean plane, a uniform tiling can be described by listing 
the sequence of regular polygons around each vertex, yielding notation
such as $(3.4.6.4)$.  A \emph{Laves
tiling}~\cite[Section~2.7]{GS} is an edge-to-edge monohedral tiling by
convex polygons with regular vertices (all angles between consecutive
edges at a vertex equal) and a tile-transitive symmetry group.  Analogous
notation such as $[3.4.6.4]$ is used for Laves tilings, listing the
sequence of vertex degrees round each tile, and in an appropriate
sense Laves tilings are dual to uniform tilings.


\section{The hat polykite and its tilings}
\label{sec:discussion}

\label{sec:substitution}

Before proceeding to the full proof of aperiodicity, we 
first offer a less formal presentation of the hat, including an explicit
construction of a tiling.  This section 
fulfills three goals.  First, it offers an abundance of visual 
intuition, which provides context for the technical machinery that will
follow.  Second, it gives some sense of our process of discovery and
analysis, though it should not be interpreted as an ordered timeline.
Third, it includes a few observations that will not be 
considered further in this article, but which might provide opportunities
for future work by others.

The first author (Smith) began investigating the hat polykite as
part of his open-ended visual exploration of shapes and their tiling
properties.  Working largely by hand, with the assistance of
Scherphuis's PolyForm Puzzle Solver 
software (\href{https://www.jaapsch.net/puzzles/polysolver.htm}{\nolinkurl{www.jaapsch.net/puzzles/polysolver.htm}}),
he could find no obvious barriers to the construction of large patches, 
and yet no clear cluster of tiles that filled the plane periodically.

Because the hat is a polyform, it was natural at this point to
obtain an initial diagnosis of its tiling properties computationally.
We modified Kaplan's SAT-based Heesch number software~\cite{Kaplan}
to determine that if the hat does not tile the plane, then its
Heesch number must be at least 16.  Similarly, we modified Myers'
polyform tiling software~\cite{Myers} to determine that if the hat
admits periodic tilings, then its isohedral number must be at least
64.  These two computations already establish that the hat is of
extreme interest---if it had turned out not to be an einstein, then 
it would have shattered either
the record for Heesch numbers or the record for isohedral numbers,
in both cases by a wide margin.

\begin{figure}[htp!]
\begin{center}
\includegraphics[width=\textwidth]{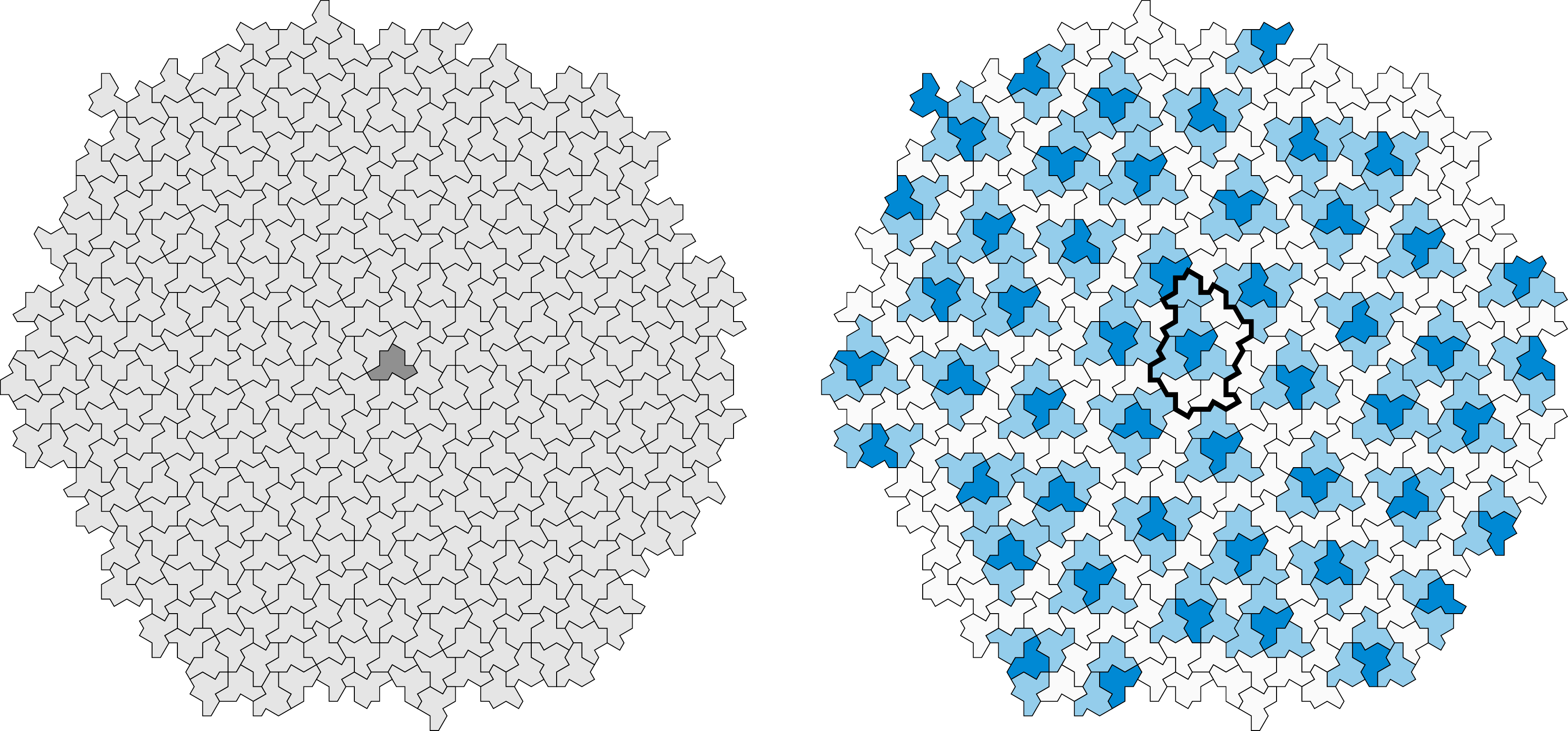}
\end{center}
\caption{\label{fig:patch10}A computer-generated $10$-patch of 391 hats
	(left), arranged in ten concentric rings around a central shaded hat.
	The tiles can be coloured (right), showing that the reflected 
	hats (dark blue) are sparsely distributed and each is surrounded by
	a congruent ``shell'' of three unreflected hats (light blue).
	A thickened outline shows the boundary of the maximal cluster
	of tiles that appears congruently around every reflected tile.}
\end{figure}

\fig{fig:patch10} (left) shows a computer-generated $10$-patch 
(i.e., ten concentric rings of tiles around a shaded central tile,
where each tile in a ring touches the ring it encloses in at least one 
point).  It was constructed by allowing Kaplan's software to work outward to
that radius, and then stopping it manually.
At first glance, it can be difficult to discern any
structure at all in this patch.  However, by colouring the tiles in different
ways, clear ``features'' begin to emerge.  Of course, we cannot infer any
conclusive properties of infinite tilings from a finite computed patch.
We must be particularly wary of tiles near the periphery of the patch,
where features may break down under the extra freedom afforded by
the proximity to empty space.  However, for a sufficiently large patch,
we might hope that tiles near the centre will be representative of 
configurations that arise in generic tilings.

\begin{figure}[htp!]
\begin{center}
\includegraphics[width=\textwidth]{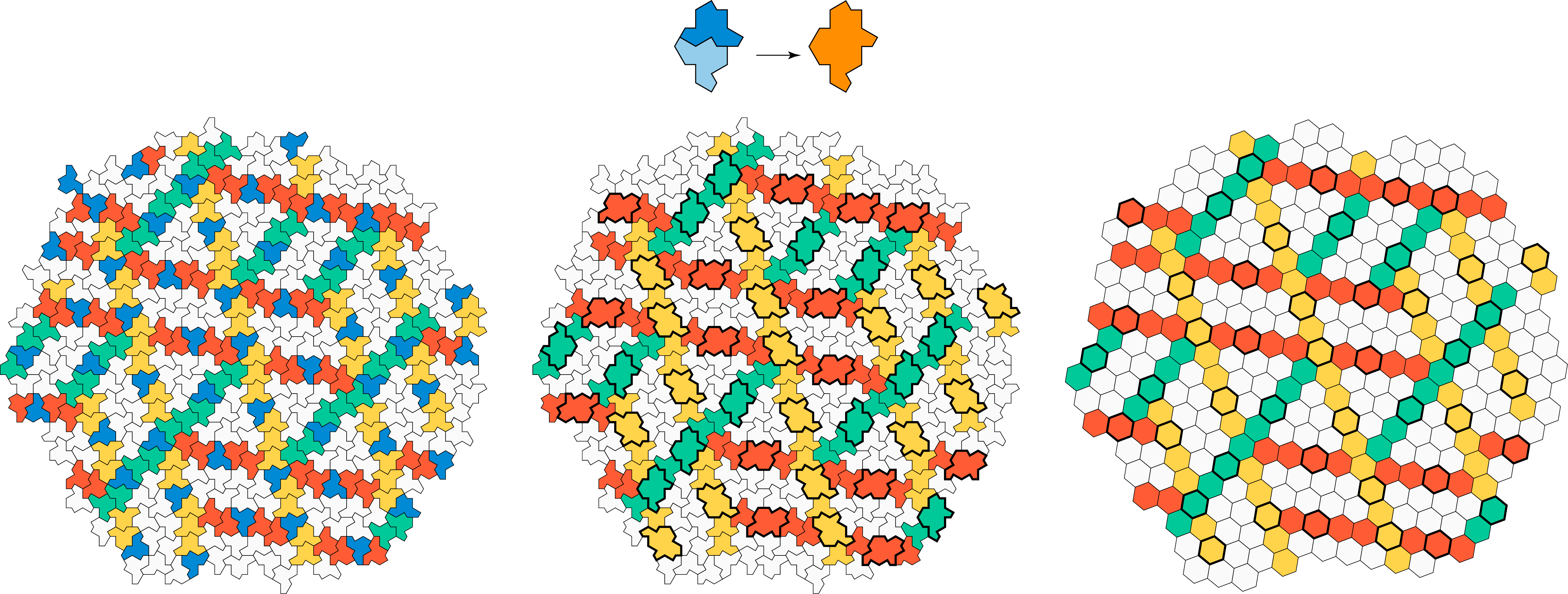}
\end{center}
\caption{\label{fig:patch10col}
	Long chains of similarly oriented tiles pass through reflected tiles
	in six directions (left).  We can merge each reflected tile with one
	of its neighbours in its chain (centre), yielding a structure
	that can be placed into one-to-one correspondence with a patch of
	regular hexagons (right).}
\end{figure}

The most important colouring for the purposes of this article is the one shown
on the right in \fig{fig:patch10}.  A single hat is asymmetric, 
and so in any patch we can distinguish between ``unreflected'' and 
``reflected'' orientations of tiles.  In the patches we computed,
reflected tiles (shown in dark blue) are always
distributed sparsely and evenly within a field of unreflected tiles.
Furthermore, every reflected tile is contained within a congruent cluster
of nine tiles, where the other eight tiles in the cluster are unreflected.
One such cluster is outlined in bold in the illustration.  The interior of
the patch can be covered completely by overlapping copies of that cluster.
Within the cluster, we are particularly interested in the ``shell'' of 
three light blue tiles adjacent to each reflected tile.  Every reflected tile
resides in a congruent, non-overlapping copy of this shell.

We have also observed that unreflected tiles tend to form long
``chains'' of like orientation, occasionally interrupted by reflected
tiles.  The chains contained in the example patch are shown coloured
on the left in \fig{fig:patch10col}.  Because the hats are aligned
with the underlying kite grid, unreflected tiles come in six
orientations, all of which also appear as chain directions.  Chains
may end at reflected tiles or pass through them, but each reflected
tile is a hub for at least two, and at most five spokes.  
Long segments of these chains have boundaries with halfturn symmetry.
It is tempting to seek parallels between these chains and linear
features in other aperiodic tilings, such as Ammann bars~\cite[Section
10.6]{GS} and Conway worms~\cite[Section 10.5]{GS}.
Finally, we have noticed that these chains seem to impart a rough
hexagonal arrangement to the hats, which is particularly clear in the
triangular and parallelogram-shaped structures that are surrounded
by chains.  We have found that if we merge each reflected tile with its
immediate neighbour as shown in \fig{fig:patch10col} (centre), then the
tiles in any patch can be put into one-to-one correspondence with a 
patch of hexagons, as in \fig{fig:patch10col} (right).  The hexagonal grid
may provide a convenient domain in which to perform computations on the
combinatorial structure of tilings by hats.

\begin{figure}[htp!]
\begin{center}
\includegraphics[width=\textwidth]{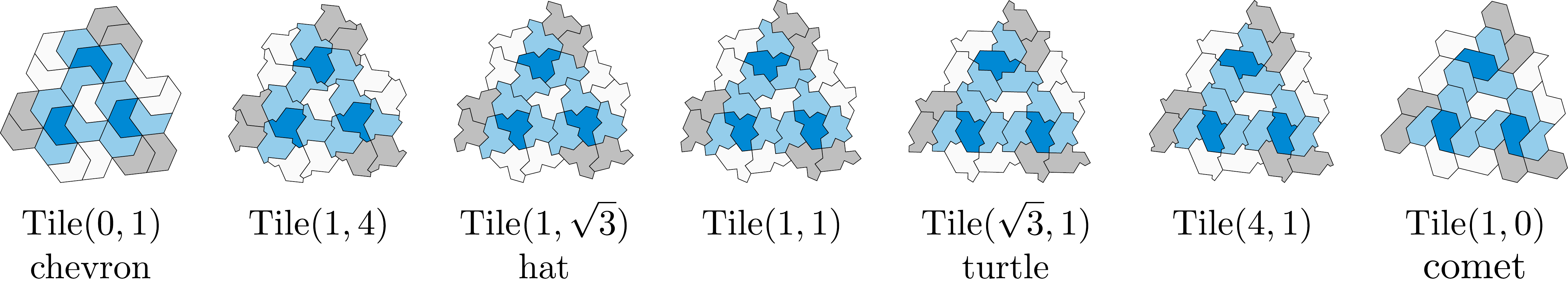}
\end{center}
\caption{\label{fig:tile_ab}The two edge lengths in the hat polykite
	can be manipulated independently, producing a continuum of shapes.
	A selection of those shapes is shown here (with each patch rescaled
	for legibility).  
	$\mathrm{Tile}(0,1)$ (the ``chevron''), $\mathrm{Tile}(1,1)$, 
	and $\mathrm{Tile}(1,0)$ (the ``comet'')
	admit periodic tilings; all others are aperiodic.}
\end{figure}

In the course of his explorations, the first author discovered a
\emph{second} polykite that did not seem to have a finite isohedral
number or a finite Heesch number, this one a union of ten kites that we
call the ``turtle''.
The idea of identifying two einsteins back-to-back seemed too good to be 
true!  It was both a relief and a revelation when we determined that not
only were the hat and the turtle related, they were in fact two
points from a continuum of shapes that all tile the plane the same
way.  The hat is derived from the $[3,4,6,4]$ grid, and therefore its edges
come in two lengths, which we can take to be $1$ and $\sqrt{3}$
(where we regard an edge of length~$2$ as two consecutive edges of
length~$1$).  Furthermore,
these edges come in parallel pairs, allowing us to set the two lengths
independently to any non-negative values.  We use the notation 
$\mathrm{Tile}(a,b)$ with $a$ and $b$ not both zero to refer to the shape
produced when edge lengths~$a$ and~$b$ are used in place of~$1$ and
$\sqrt{3}$, respectively.
Note that $\mathrm{Tile}(a,b)$ is similar to
$\mathrm{Tile}(ka,kb)$ for any $k\neq 0$.  

\fig{fig:tile_ab} shows
a selection of shapes from the $\mathrm{Tile}(a,b)$ continuum.
We have also created an animation showing a continuous evolution
of $\mathrm{Tile}(a,1-a)$ as $a$ moves from $0$ to $1$ and back---see
\href{https://youtu.be/W-ECvtIA-5A}{\nolinkurl{youtu.be/W-ECvtIA-5A}}.
Within this continuum, we see that the hat is $\mathrm{Tile}(1,\sqrt{3})$ and the 
turtle is $\mathrm{Tile}(\sqrt{3},1)$.  When one of~$a$ or~$b$ is zero,
we obtain two additional shapes of interest.  We refer to 
$\mathrm{Tile}(0,1)$ as the ``chevron''; it is a tetriamond, a union
of four equilateral triangles.  Similarly $\mathrm{Tile}(1,0)$ is an 
octiamond that we call the ``comet''.
The chevron, the comet,
and the equilateral $\mathrm{Tile}(1,1)$ each admit simple periodic tilings;
in \secref{sec:family}, we will show that all other shapes in this continuum
are aperiodic monotiles with combinatorially equivalent tilings.
In \secref{sec:coupling}, the chevron and comet 
will play a crucial role in establishing that the hat is aperiodic.
Inspired by cut-and-project methods~\cite{deBruijn1,deBruijn2}, we are also
left wondering
whether it would be productive to construct a closed path in four or
six dimensions, which projects down to this family of tiles from a suitable
set of directions.

\begin{figure}[htp!]
\begin{center}
\includegraphics[height=2.75in]{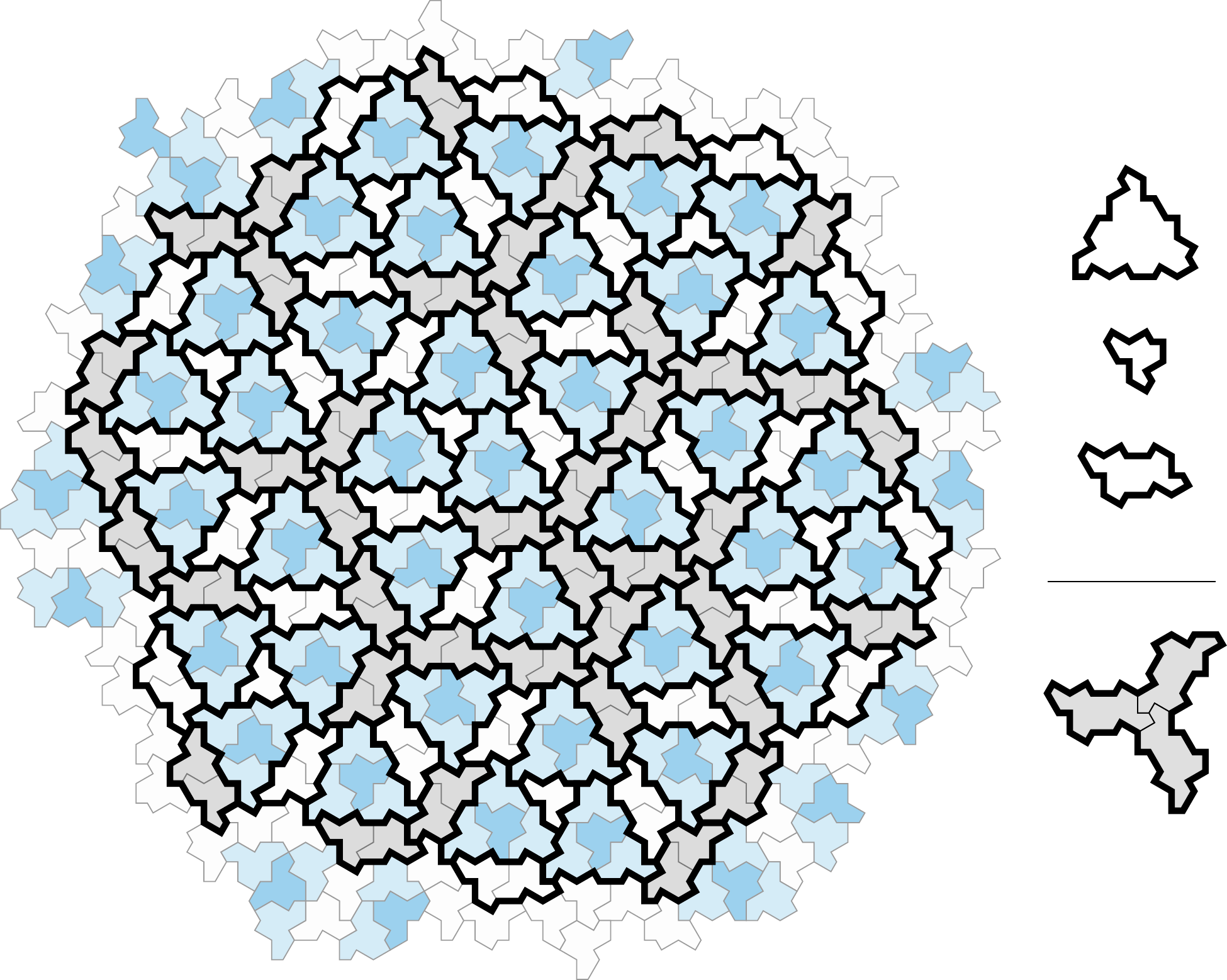}
\end{center}
\caption{\label{fig:patch10clu}A grouping of tiles into clusters in the
	example patch.  In addition to four-tile clusters consisting of a 
	reflected hat and its three-hat shell, we identify clusters consisting
	of a single tile, and parallelogram-shaped clusters consisting of
	pairs of tiles. The parallelograms come in two varieties: one 
	separates two nearby shells, and the other joins up with two rotated
	copies to make a three-armed propeller shape called a \textit{triskelion}.
	An isolated triskelion is shown shaded in grey in the lower right.}\end{figure}

Given the colouring in \fig{fig:patch10col} showing non-overlapping
clusters of reflected tiles and their shells,
it is natural to wonder whether the remaining unaffiliated tiles in the
patch reliably 
form clusters of other kinds.  \fig{fig:patch10clu} illustrates that we can
account for all remaining tiles using two additional cluster types (shown
separately on the right).  First,
where three shells meet they enclose a single isolated tile, which must be
included as a cluster of size one.  Then the remaining tiles group into
copies of a parallelogram-shaped cluster of size two.
These appear in two varieties, depending on the local arrangement of
clusters around them.  In the first case, coloured white in the drawing,
the parallelogram is adjacent to
two shells along its long edges.  In the second case, coloured grey,
one end of the 
parallelogram is plugged into a local centre of threefold rotation, joining
six hats into a three-armed propeller shape called a \textit{triskelion}.
A triskelion is shown in isolation on the bottom right of \fig{fig:patch10clu}.

These clusters are the starting point for the definition of a
substitution system, one that can be iterated to produce patches
of hats of arbitrary size.  The substitution rules
do not apply to the hats directly.  Instead, we derive new \textit{metatiles}
from the clusters, and build a substitution system based on the
metatiles. The underlying hats are simply brought along for the
ride.

\begin{figure}[htp!]
\begin{center}
\includegraphics[width=0.75\textwidth]{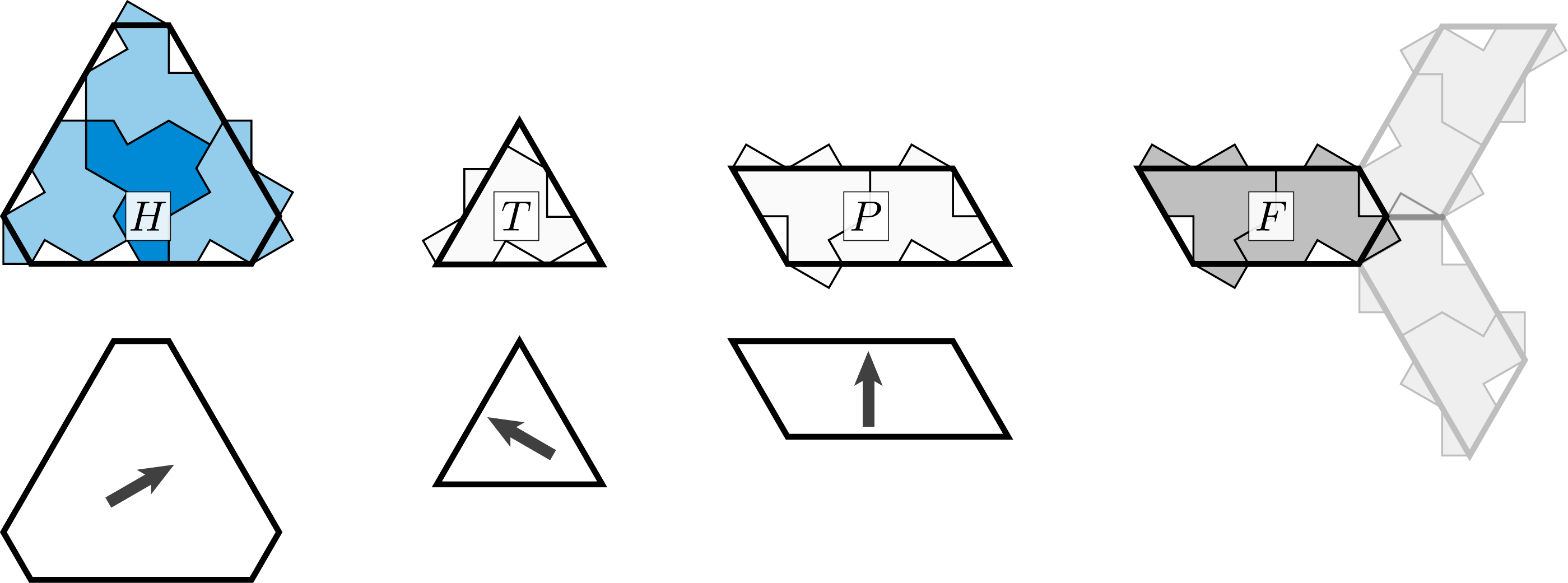}
\end{center}
\caption{\label{fig:simp_outlines}The $H$, $T$, $P$, and $F$ metatiles (top),
	constructed by simplifying the boundaries of clusters of hats.
	We mark the $H$, $T$, and $P$ metatiles with arrows when needed (bottom),
	to distinguish between otherwise symmetric orientations.}
\end{figure}

\fig{fig:simp_outlines} shows the shapes of the metatiles.  Each
one is constructed by simplifying the boundary of one of the clusters
of hats in \fig{fig:patch10clu}.  In order to ensure that the metatiles
do not overlap, we must distinguish between the two varieties of two-hat
parallelograms discussed above.  Specifically, we remove a triangular
notch from the parallelogram associated with each leg of a triskelion.  Thus
the three clusters yield four metatiles: an irregular hexagon ($H$), 
an equilateral triangle ($T$), a parallelogram ($P$), and a pentagonal
triskelion leg ($F$).
The original clusters can now
be seen as endowing the metatiles with matching rules
along their edges; these rules will be formalized in
\secref{sec:clusters}.

The $H$, $T$, and $P$ metatiles have rotational symmetries.
In the bottom row of \fig{fig:simp_outlines}, we mark tiles
with arrows showing their intended orientations.  In each case, the
arrow points to the (unique) side of the metatile from which two adjacent
kites protrude.  The arrows suffice to distinguish symmetric rotations
and our construction will not use reflections. (We will not need these
arrows in later sections, as metatile orientations will be implied by 
labels on their edges.)

\begin{figure}[htp!]
\begin{center}
\includegraphics[width=0.95\textwidth]{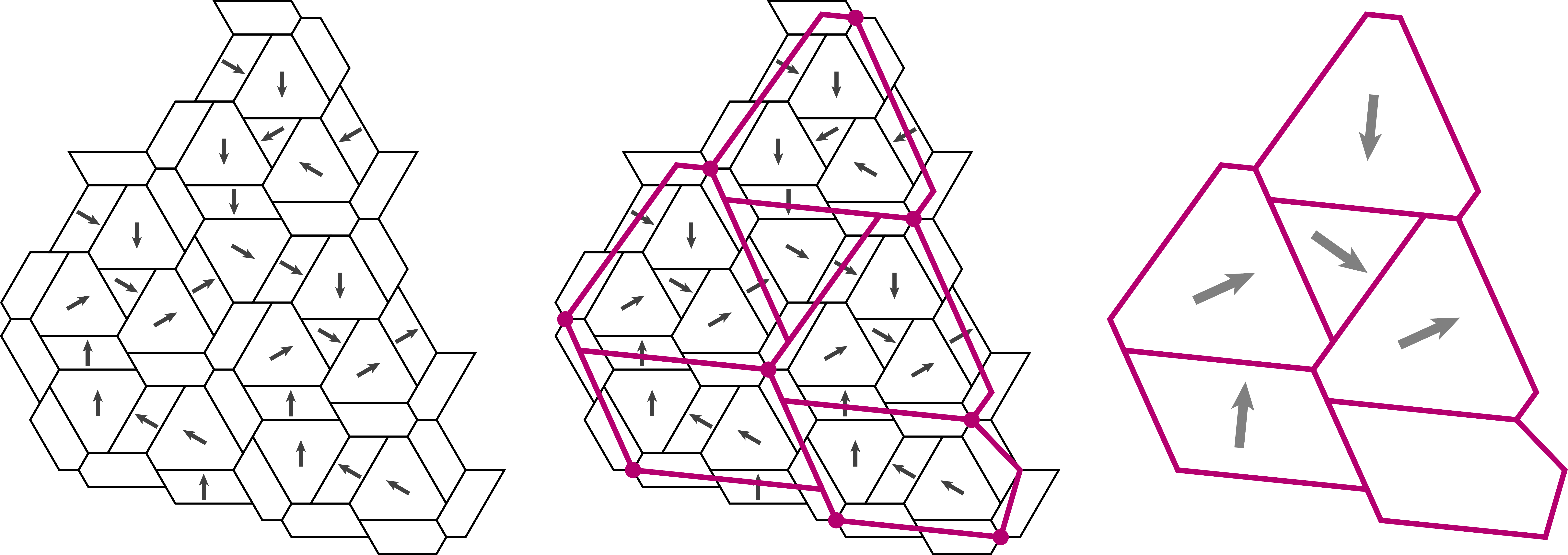}
\end{center}
\caption{\label{fig:supertiles}The construction of a family of 
	supertiles from a patch of metatiles.  The patch of metatiles on
	the left can be used to locate key vertices of the supertiles,
	marked with red dots in the central diagram.  Those dots, 
	together with constraints on angles, fully determine the shapes
	of the supertiles, which are not merely scaled-up copies of their
	progenitors.  On the right, the supertiles are marked with
	arrows indicating their orientations.}
\end{figure}

We can now define a family of supertiles that are analogous to the
metatiles, following the procedure illustrated in 
\fig{fig:supertiles}.  We first assemble the patch of oriented metatiles
shown on the left.  It can easily be checked that the elided
hats borne by these tiles fit together with no gaps and no
overlaps.  This patch is large enough to pick out one
or more copies of 
each supertile, drawn in red in the central diagram.
The supertile shapes are fully determined by two constraints:
the red dots coincide with the centres of triskelions, and all interior angles
of the hexagonal outlines are~$120^\circ$.  The diagram on the right 
shows the supercluster outlines in isolation, with their inherited
orientation markings.  Here, each arrow points to the unique supertile
edge that passes through an outward-pointing $P$ tile from the previous
generation.

\begin{figure}[ht!]
\begin{center}
\includegraphics[width=\textwidth]{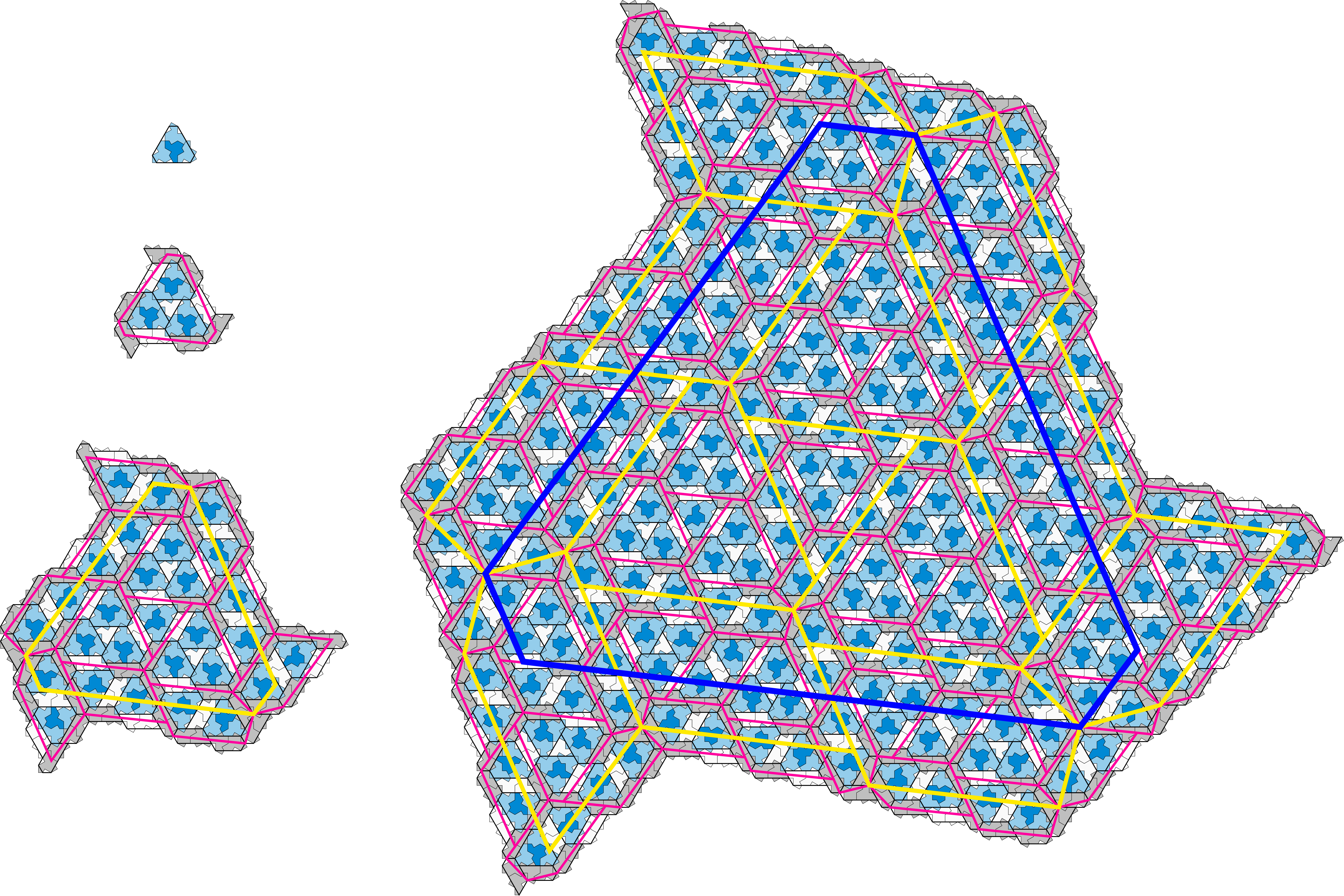}
\end{center}
\caption{\label{fig:hexclusters}The first four iterations of the 
	$H$ metatile and its supertiles.  At each level, tiles partially
	overlap the boundary of their supertile.  Overlaps are acceptable
	here, because the supertile will be met by neighbouring supertiles
	with the same configuration of smaller tiles on its boundary.}
\end{figure}

\begin{figure}[ht!]
\begin{center}
\includegraphics[width=\textwidth]{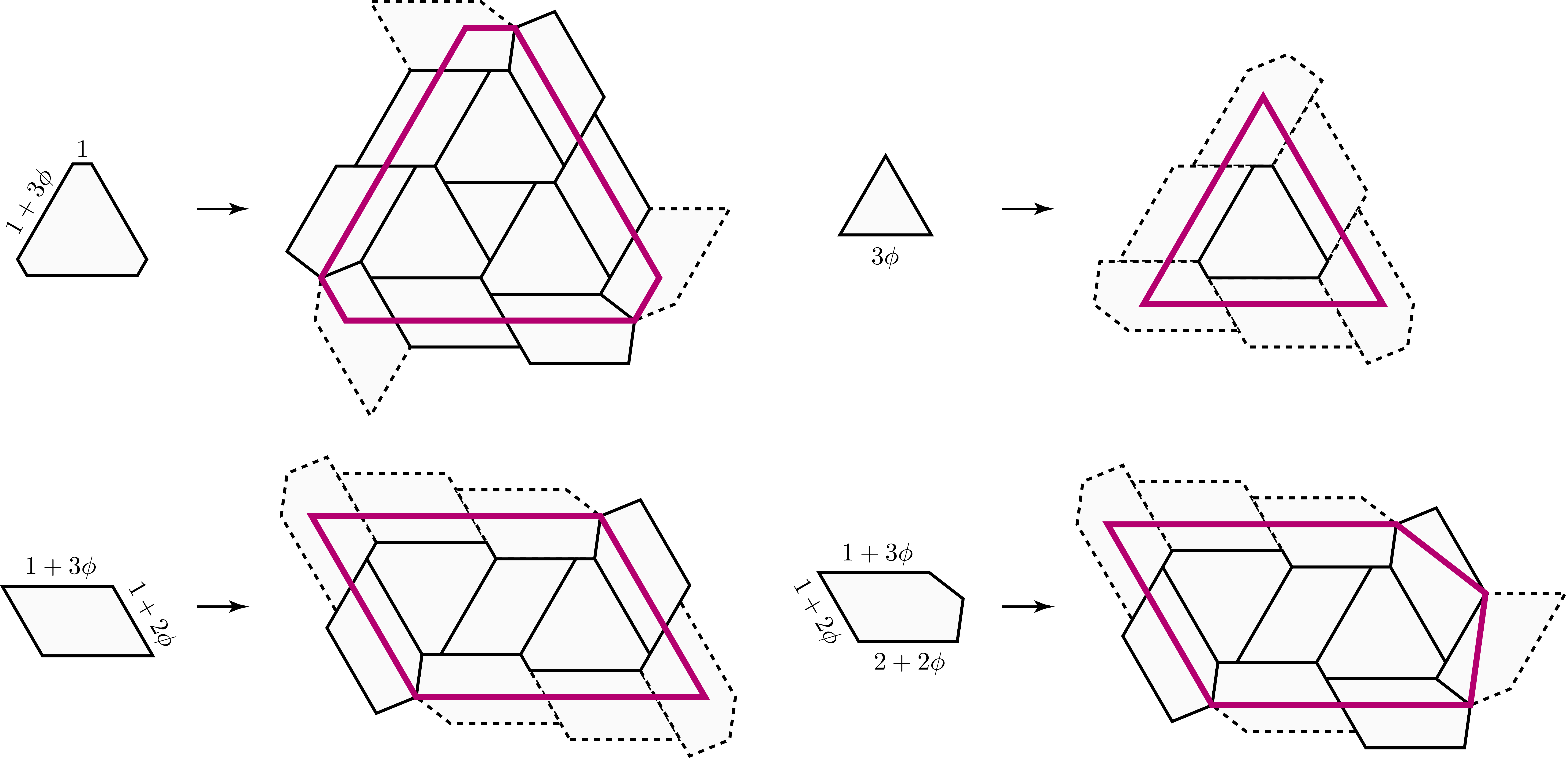}
\end{center}
\caption{\label{fig:perfect}A geometric substitution system based on
	converged
	tile shapes.  The tiles are scaled so that the short edges of the $H$
	tile have unit length.  All tile edges except the two 
	unmarked $F$ edges
	have lengths in $\mathbb{Z}[\phi]$, where $\phi$ is the golden ratio.
	The lengths of the unmarked edges in $H$, $T$, and $P$ are given by
	symmetry.
	In each substitution rule, tiles shown with dashed boundaries can be 
	omitted, leading to patches in which there are no duplicate tiles
	contributed by supertiles sharing an edge.}
\end{figure}

At first glance, these supertiles appear to be scaled-up copies
of the metatiles.  If that were so, we could perhaps proceed to
define a typical substitution tiling, where each scaled-up supertile is
associated with a set of rigidly transformed tiles.  However, with
the obvious exception of the $T$, none of the 
supertiles is similar to its corresponding metatile.
Despite that discrepancy, the supertiles are fully
compatible with the construction in \fig{fig:supertiles}---they
can be arranged in the same configuration shown on the left, and used as
a scaffolding for deriving outlines of level-$2$ supertiles (implicitly
yielding a much larger patch of hats along the way).  Indeed, the construction
can be iterated any number of times, with slightly different outlines in
every generation.  
Substitution systems like this one, where successive generations
are combinatorially but not geometrically compatible, are uncommon
in the world of aperiodic tilings.  Here we are forced to work with the
properties of a shape discovered in the wild, instead of
engineering a tile set to conform to our wishes.
To see this construction in action, please try our interactive
browser-based visualization tool at
\href{https://cs.uwaterloo.ca/~csk/hat/}{\nolinkurl{cs.uwaterloo.ca/~csk/hat/}}.

We know from \fig{fig:simp_outlines} that each of the four metatiles
can be associated with a cluster of hats.  The construction in
\fig{fig:supertiles} can then be iterated any number of times to form
ever-larger patches of metatiles, and hence of hats.  We can, for 
example, consider the $H$ supertiles formed through this process of
iteration, and the patch of hats each one contains.  The first few 
generations of
$H$ supertiles are illustrated in \fig{fig:hexclusters}.  These patches form
a sequence that grows in radius without bound, each patch a subset of
its successor.  The Extension Theorem~\cite[Section 3.8]{GS} allows
us to continue this iteration process ``to infinity'', yielding the
following result.

\begin{theorem}
\label{thm:subst_tiling}
	The hat polykite admits tilings of the plane.
\end{theorem}

Of course, this theorem is not sufficient to establish aperiodicity
on its own---we must also show that the hat does not also admit
periodic tilings.  In \secref{sec:subst} we revisit this
substitution process, tracking matching rules on supertile edges
after every step.  There we show that \textit{all} tilings by the 
hat necessarily obey the substitution rules given here
(Theorem~\ref{thm:subst}).  The construction in that section also
furnishes a more detailed proof that the hat tiles the plane.

The shapes of each generation of supertiles are different from those
of the generation before it.  However, by normalizing the tiles for
size, we have computed that they converge on a fixed point, a 
set of tiles that truly do yield scaled copies of themselves under the
construction in \fig{fig:supertiles}.  These converged tile
shapes are particularly interesting because they can be used to define a 
geometric substitution system that operates via inflation and replacement.
The converged tiles, together with
their substitution rules, are shown in \fig{fig:perfect}.
By virtue of its connection to the original metatiles in
\fig{fig:simp_outlines}, we know that this substitution tiling is aperiodic
when the tiles are endowed with suitable matching rules on their
edges.  We can also use this system as an alternative means of constructing 
patches of hats.  We cannot simply associate a cluster of hats rigidly with
each converged tile, but a patch of converged tiles is combinatorially
equivalent to a corresponding patch of metatiles, which are
equipped with hats.

If we rescale the converged tiles so that the short $H$ edges
have unit length, then all tile edges except the two $F$ edges adjacent to
a triskelion centre
will have lengths in $\mathbb{Z}[\phi]$, where $\phi$ is the
golden ratio.  Furthermore, this substitution system has an inflation factor
of $\phi^2$.  The factor of~$\phi^2$ can also be derived algebraically,
using an eigenvalue computation on the substitution matrix corresponding to
the system presented in
this section.

At first blush, it may be surprising to see $\phi$ arise
in a tiling closely associated with the Laves tiling $[3.4.6.4]$; it
appears more naturally in contexts such as Penrose tilings, which feature
angles derived from the regular pentagon.
The presence of $\phi$ appears to be closely related to the
appearance of~$\sqrt{2}$ in the argument of \secref{sec:coupling}.
That number is also not expected on the regular triangular tiling or
related contexts (distances are the square roots of integers
that can be expressed by the quadratic form $x^2+xy+y^2$, so expected
square roots are of $3$ and primes of the form $6k+1$).  However, $1 +
\phi^{-1} + \phi^{-2} = 2$, from which it follows that a triangle with a
$120^\circ$~angle between sides of lengths $1$ and~$\phi^{-1}$ has
a third side of length~$\sqrt{2}$ (\fig{fig:arctan}, left). 
Now observe that in any tiling by $\mathrm{Tile}(a,b)$ we can construct
progressively larger equilateral triangles whose corners coincide with the
centres of triskelions (these triangles correspond to alternating corners
of the $H$-metatile and its supertiles).  
We may then superimpose a tiling by $\mathrm{Tile}(1,0)$ with a rotated
and translated tiling by $\mathrm{Tile}(0,\sqrt{2})$ 
so that two such equilateral
triangles are aligned.  We find that the equilateral triangle
grids supporting these two tilings are offset by an angle that approximates
$\tan^{-1}\sqrt{3/5}$, which is one of the angles of the triangle with
sides $1$, $\phi^{-1}$, and~$\sqrt{2}$, and that the approximation appears
to converge as we align larger triangles~(\fig{fig:arctan}, right).

\begin{figure}[ht!]
\begin{center}
\includegraphics[width=0.65\textwidth]{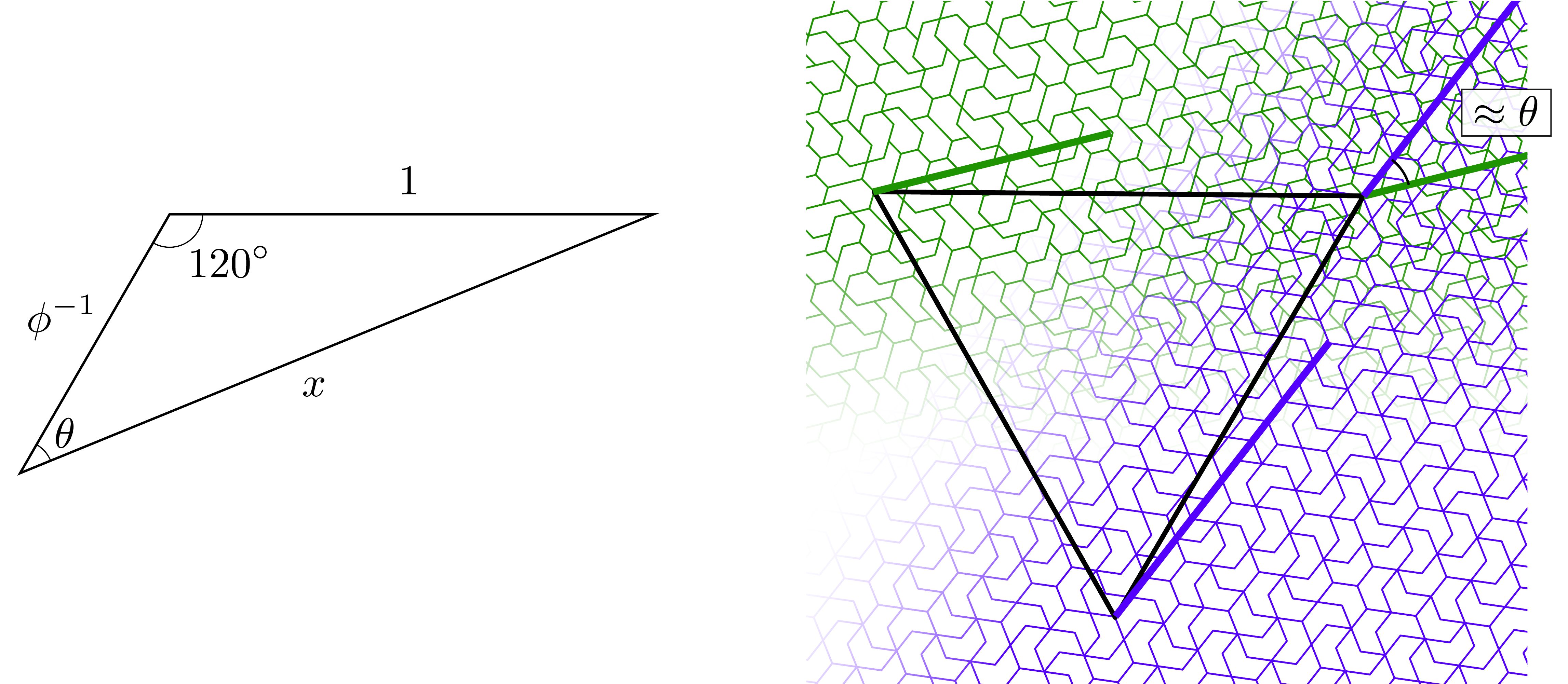}
\end{center}
\caption{\label{fig:arctan}
	A demonstration of how the golden ratio $\phi$ might arise in
	the context of tilings by hats.  The triangle on the left has
	an angle of $120^\circ$ between sides of lengths $1$ and $\phi^{-1}$;
	from trigonometric identities we can compute that $x=\sqrt{2}$ and
	$\theta=\tan^{-1}\sqrt{3/5}$.  On the right we show portions
	of tilings by $\mathrm{Tile}(1,0)$ (green) and $\mathrm{Tile}(0,\sqrt{2})$
	(blue), registered to the same centres of local threefold rotation.
	The angle between the edges of the triangle tilings underlying these
	two polyiamond tilings is approximately $\theta$, and we believe this
	approximation converges as we register larger patches of the two
	tilings.}
\end{figure}

\begin{figure}[htp!]
\begin{center}
\includegraphics[width=3in]{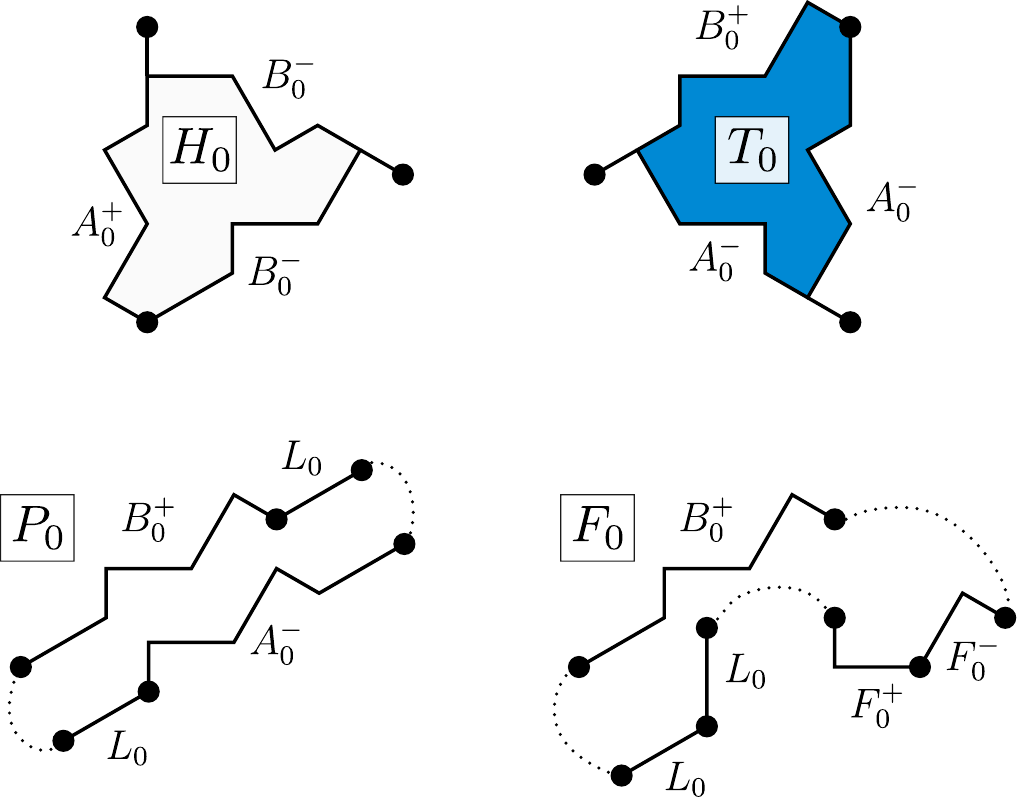}
\end{center}
\caption{\label{fig:subclusters}
	Four subclusters that may be thought of as preceding the clusters making
	up the metatiles in \fig{fig:simp_outlines}.  Edges are marked with
	the labels that will be used in \secref{sec:clusters}.  The~$P_0$ and $F_0$ subclusters have zero area; dotted lines indicate 
	vertices that should be regarded as coincident.}
\end{figure}

The arrangement of three $H$ tiles and a $T$ tile inside of an $H$ 
supertile mimics the arrangement of a reflected hat and its three
unreflected neighbours in a single $H$ metatile.  We are naturally
led to wonder whether the clusters of hats that make up the metatiles
are primordial, or whether they are preceded by a set of ``subclusters''
that launch the substitution process one step earlier.
A possible form for such subclusters is shown in Figure~\ref{fig:subclusters}.
The labels on the edges denote matching rules that will be explained
in detail in \secref{sec:clusters}.
Note that subclusters $P_0$ and~$F_0$
have zero area; their boundaries are shown split into multiple
parts to clarify the sequence of edges (in the case of~$F_0$, some of
those edges intersect others).  Also, the $X^+$ and $X^-$
edges from Figure~\ref{fig:tileaclusters} have length zero in the
subclusters and are not shown in the diagram.  Defining what exactly
it means to partition a tiling into subclusters following matching
rules, when some subclusters have area zero and some edges have length
zero but must still adjoin in the correct orientations, seems 
more awkward than the corresponding argument
based on metatiles, so we do not pursue the subclusters further.

\begin{figure}[t]
\begin{center}
\includegraphics[width=0.9\textwidth]{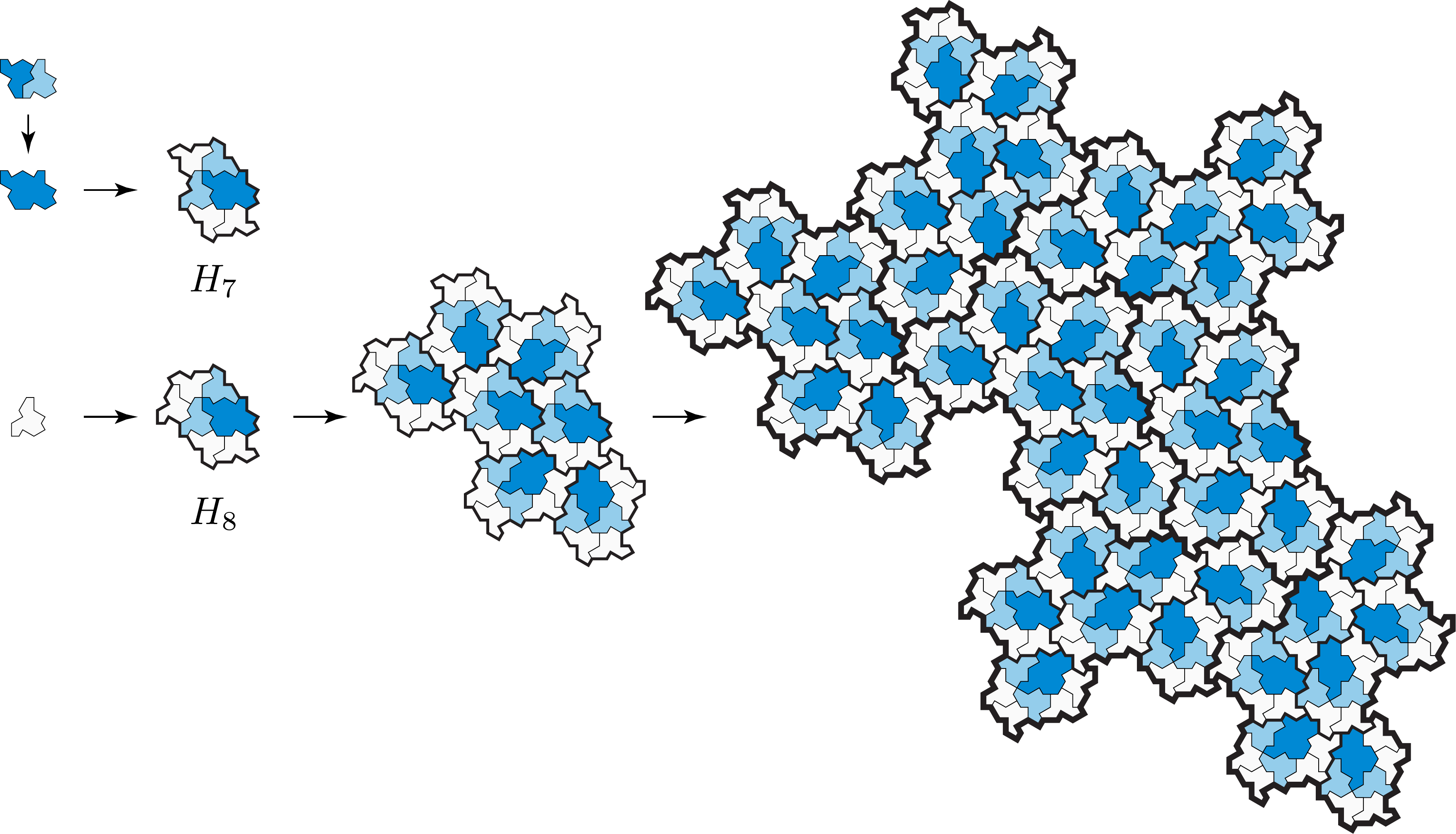}
\end{center}
\caption{\label{fig:alt_subst}An alternative substitution system that
	yields the same tiling by hats as the system presented earlier.
	We merge a reflected hat with one of its neighbours (top left) to
	produce a two-hat compound.  We can then define substitution
	rules that replace a compound by a cluster labelled $H_7$ and a
	single hat by a cluster labelled $H_8$.  Two additional iterations of
	the $H_8$ rule are shown.}
\end{figure}

Finally, in \fig{fig:alt_subst} we exhibit an alternative substitution
system based on a different set of clusters.  Here each reflected
hat is merged with a specific neighbour in its shell, 
in a manner similar to \fig{fig:patch10col}, to form the two-hat 
compound shown in the upper left.  We can then
define just two combinatorial substitution rules: one replaces a
two-hat compound by a cluster of a compound and five hats (labelled $H_7$ 
in the figure), and the other replaces a single hat by a cluster of a 
compound and six hats (labelled $H_8$).
As with the original metatiles, this 
process can be iterated to produce a patch of any size, after which
each compound can be split back into a pair of hats.
This substitution system is attractive for its minimality, though we 
believe it would be more cumbersome for proving aperiodicity.
Although the tilings themselves are MLD (mutually locally
derivable)~\cite{BaakeMLD}
with those by the $H$, $T$, $P$, and $F$ metatiles presented earlier,
deriving those metatiles from the clusters shown here requires
considering a radius larger than a single cluster.
$H_7$ is always equivalent to the union of an
$H$ metatile, a $T$ metatile, and an $F$ metatile, but $H_8$
corresponds
to three different combinations of $H$, $P$, and $F$.  Alternatively,
to establish an MLD
system, we could define four congruent but combinatorially 
inequivalent copies of the substitution rule for a single hat.
It is impossible to define the hat tiling through a single subtitution 
rule. The implied substitution matrix would necessarily yield a rational
asymptotic increase in area after substitution, whereas we already know
that the hat tiling inflates areas by a factor of $\phi^4$.

\begin{figure}[ht!]
\begin{center}
\includegraphics[width=\textwidth]{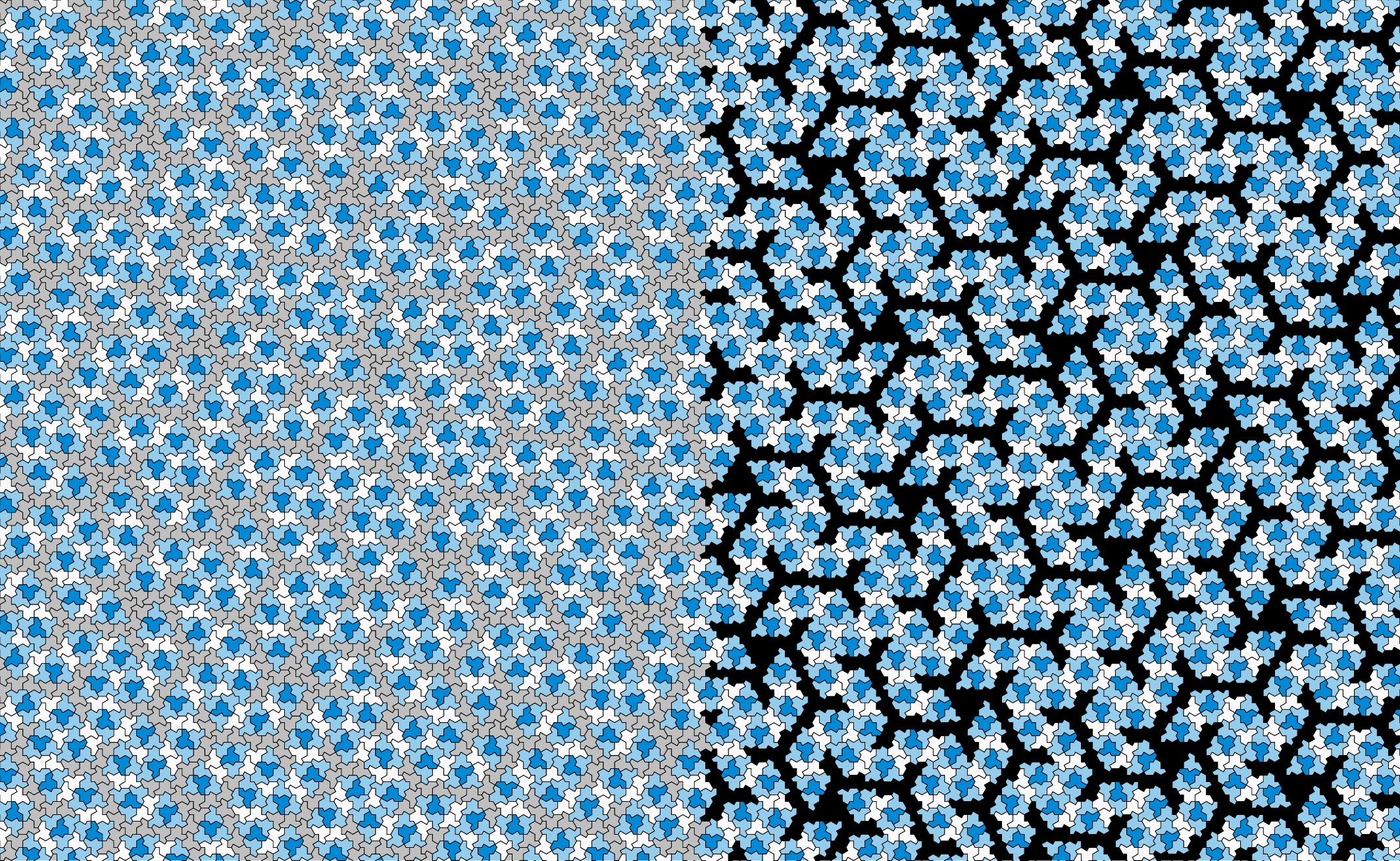}
\end{center}
\caption{\label{fig:gargantuan}An excerpt from a very large patch generated
	using the substitution system presented in this section.  In the right
	half of the drawing, hats belonging to $F$ metatiles are coloured black. It is an open question whether the $F$ metatiles must form a tree with tree complement, as suggested by the figure.}
\end{figure}

The ideas presented in this section are sufficient to show that the
hat does in fact tile the plane.  \fig{fig:gargantuan} offers a
final large patch of tiles as a demonstration.  On the right side
of the illustration we observe that the tiles belonging to triskelions
appear to form a connected tree structure that interlocks with a tree formed
from the remaining tiles.  This structure is reminiscent of those found
in other aperiodic tilings, such as the Taylor--Socolar tiling and 
the $1+\epsilon+\epsilon^2$ tiling.

However, exhibiting a tiling is usually the easy part of 
a proof of aperiodicity; it is also necessary to prove that none of the 
tilings admitted by the hat can be periodic.  In the next section
we present a novel geometric proof of aperiodicity.  Then, in 
Sections~\ref{sec:clusters} and~\ref{sec:subst} we turn to a more standard
combinatorial argument that the matching rules implied by the 
substitution system shown earlier are forced in tilings by the hat.

\FloatBarrier


\section{Aperiodicity via coupling of polyiamond tilings}
\label{sec:coupling}

In Theorem~\ref{thm:subst_tiling},  we proved that the hat polykite 
is a monotile: it admits tilings of the plane.  Our proof used the 
metatile substitution system of \secref{sec:discussion}, described  
in detail in Sections~\ref{sec:clusters} and~\ref{sec:subst}. 

In those sections we also use a computer-assisted case
analysis to show  that every tiling by the hat polykite arises from
the substitution rules. For that reason the hat polykite does not
admit periodic tilings, completing a proof of Theorem~\ref{thm:main}
using a standard approach going back to Berger~\cite{Berger}.

In this section we give a more direct proof of aperiodicity
that exploits the hat's
membership in the $\mathrm{Tile}(a,b)$ continuum introduced in 
\secref{sec:discussion}.
As noted in \secref{sec:terminology}, a planar tile that does not admit
strongly periodic tilings also cannot admit weakly periodic tilings.  Building
on that fact and Theorem~\ref{thm:subst_tiling}, the following establishes 
Theorem~\ref{thm:main}.

\begin{theorem}
\label{thm:coupling:aperiodic}
Let $\mathcal{T}$ be a tiling by the hat polykite.  Then $\mathcal{T}$
is not strongly periodic.
\end{theorem}

\begin{figure}
\begin{center}
\includegraphics[width=\textwidth]{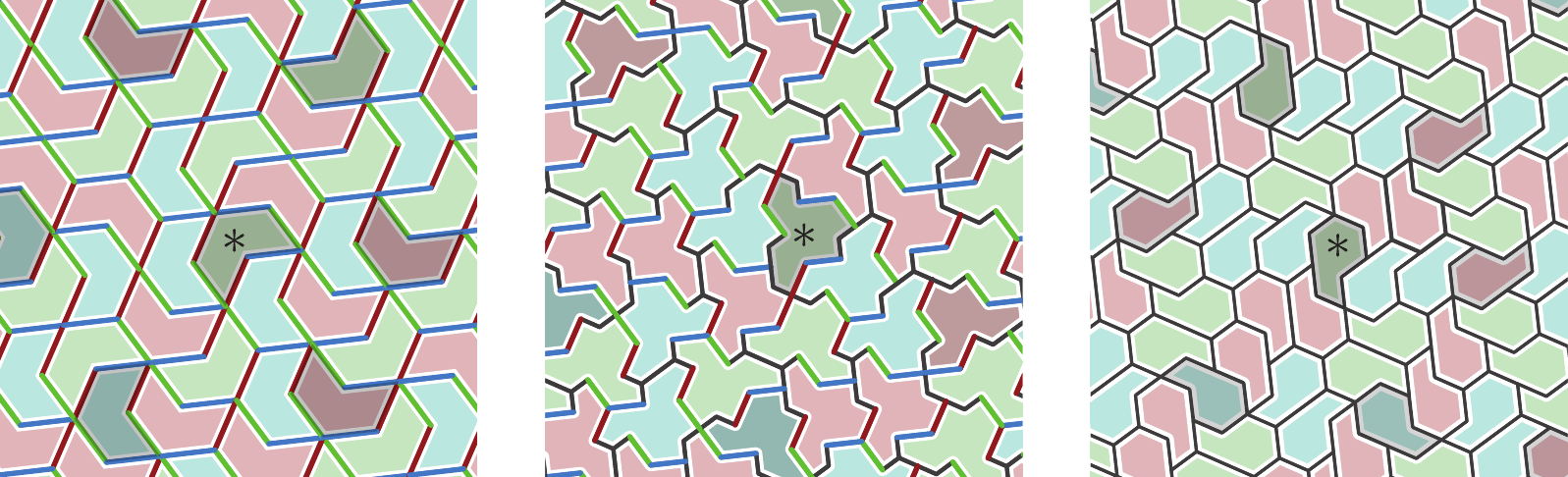}
\end{center}
\caption{\label{fig:tt4t8}A patch from a hat tiling $\mathcal{T}$ (centre),
	together with corresponding patches from a chevron tiling
	$\mathcal{T}_4$ (left) and a comet tiling $\mathcal{T}_8$ (right).
	Corresponding reference tiles are marked in each patch.
	Edges of length~$1$ are shown in black;
	tiles and edges of length~$\sqrt{3}$ are given distinct colours 
	according to their orientations, with mirrored tiles shaded darker. 
	$\mathcal{T}_4$ and $\mathcal{T}_8$ are obtained by contracting the
	black and coloured edges of $\mathcal{T}$\!, respectively.
	$\mathcal{T}$ is shown at half scale to fit more of it into the figure.
	We assume that these combinatorially equivalent tilings are periodic
	in order to derive a contradiction.
}
\end{figure}

We suppose throughout this section that there is a strongly periodic
tiling~$\mathcal{T}$ by the hat polykite, described as $\mathrm{Tile}(1,\sqrt{3})$ in \secref{sec:discussion}, and derive a contradiction. 

In the tiling $\mathcal{T}$\!, tiles are necessarily aligned to 
an underlying $[3.4.6.4]$ Laves tiling.  This claim is justified
by Lemma~\ref{lemma:tileaalign},
which shows that all tilings by the hat polykite are aligned in that way.

Contracting the sides of length 
$1$ or~$2$ to length~$0$ in $\mathcal{T}$ produces a strongly periodic
tiling~$\mathcal{T}_4$ by chevrons, tiles of the form $\mathrm{Tile}(0,\sqrt{3})$.
Each chevron is the union of four equilateral triangles
of side length~$\sqrt{3}$, and must therefore have area~$3\sqrt{3}$.
Similarly, contracting the sides of length~$\sqrt{3}$ to length~$0$
produces a strongly periodic tiling~$\mathcal{T}_8$ by comets of
the form $\mathrm{Tile}(1,0)$, which have area~$2\sqrt{3}$.
(Because this contraction process is well-defined around any tile,
edge or vertex, it yields a combinatorial tiling of the plane, and a
combinatorial tiling corresponds to a geometrical tiling of the entire
plane~\cite[Lemma~1.1]{regprod}.)  
\fig{fig:tt4t8} (centre) shows a patch from an example tiling
$\mathcal{T}$, together
with corresponding patches from $\mathcal{T}_4$ (left) and $\mathcal{T}_8$
(right).  This mapping to tiles of
different side lengths is discussed in more detail in
\secref{sec:family}.

Because both~$\mathcal{T}_4$ and~$\mathcal{T}_8$ are strongly
periodic, there must exist an affine map~$g$ that acts
as a bijection between the translation symmetries of~$\mathcal{T}_4$
and~$\mathcal{T}_8$.
Recall that a \textit{similarity} is an affine map that 
preserves shape but not necessarily size (that is, it
scales uniformly in every direction). We will first show that~$g$
is not a similarity. We will then prove that it must be,
obtaining a contradiction.

The tilings $\mathcal{T}$\!, $\mathcal{T}_4$, and $\mathcal{T}_8$ 
are coupled, in the sense that there are bijections
between their tiles, with corresponding tiles in corresponding
orientations and translation symmetries of any one mapping directly to
translation symmetries of the others.  They also have close combinatorial
relationships: any neighbours in the original
polykite tiling correspond to neighbours in both polyiamond tilings. 
The affine map~$g$ defined above transforms every translation symmetry
of~$\mathcal{T}_4$ into a corresponding translation symmetry
of~$\mathcal{T}_8$ 
(one between corresponding pairs of tiles).  Given the areas of the
chevrons and comets,~$g$ must scale areas by $2/3$.

If~$g$ were a similarity, then we could deduce immediately that it
must scale lengths in every direction by~$\sqrt{2/3}$.  However, 
a similarity with this scale factor cannot also map translations 
in~$\mathcal{T}_4$ to translations in~$\mathcal{T}_8$.
Consider the regular tiling by
equilateral triangles, positioned to include a unit edge from $(0,0)$ to
$(1,0)$.  Every vertex of this tiling is given by $m(1,0)+n(1/2,\sqrt{3}/2)$
for integers~$m$ and~$n$, meaning that vectors joining vertices 
(including all possible vectors defining translation symmetries) must
have this form as well.  It follows that 
any distance~$d$ between two vertices must have $d^2$ of the form 
$m^2+mn+n^2$, in which case $d^2$ has an even number
of factors of~$2$.  Therefore a scale factor
of~$\sqrt{2}$ is not possible between translations on two triangular
tilings with the same edge length, and a scale factor of~$\sqrt{2/3}$
is not possible between translations on two triangular tilings with
edge lengths in a ratio of $\sqrt{3}$ to~$1$.

Using the fact that the six translation classes of kites must appear with equal frequency in
any aligned tiling by polykites, we now proceed to show that $g$  must 
in fact be a similarity, which gives the required contradiction.

A polykite is a union of kites from a $[3.4.6.4]$ Laves tiling, and
so its constituent kites are constrained to six possible orientations.
It happens that the hat uses four of those six kite orientations once
each, and the other two orientations (which are related by a halfturn)
twice each.  In any aligned
hat tiling, there are twelve possible tile orientations.  Tiles can
therefore be partitioned into three ``classes'' of four orientations
each, based on the orientations of their repeated kites.
\fig{fig:kitetwocol} (left)
shows the four hat orientations that make up one such class;  within
each hat, kites in orientations that appear more than once are
shaded darker.  Because of these repeated kite orientations, hats 
in each class claim the Laves tiling's
kites in an unbalanced way, favouring two kite orientations over
the other four.  In an infinite hat tiling all kite orientations must
be used in equal proportion, and so to restore balance the tiling
must use tiles from the three classes in equal proportion as well
(meaning that in any patch with perimeter~$x$, the imbalance between
the numbers of polykites with orientations from any two of the sets
is~$O(x)$).  In \fig{fig:tt4t8} (centre), copies of the hat 
with the same orientation class and handedness are coloured the same way.

\begin{figure}[htp!]
\begin{center}
\includegraphics[width=.85\hsize]{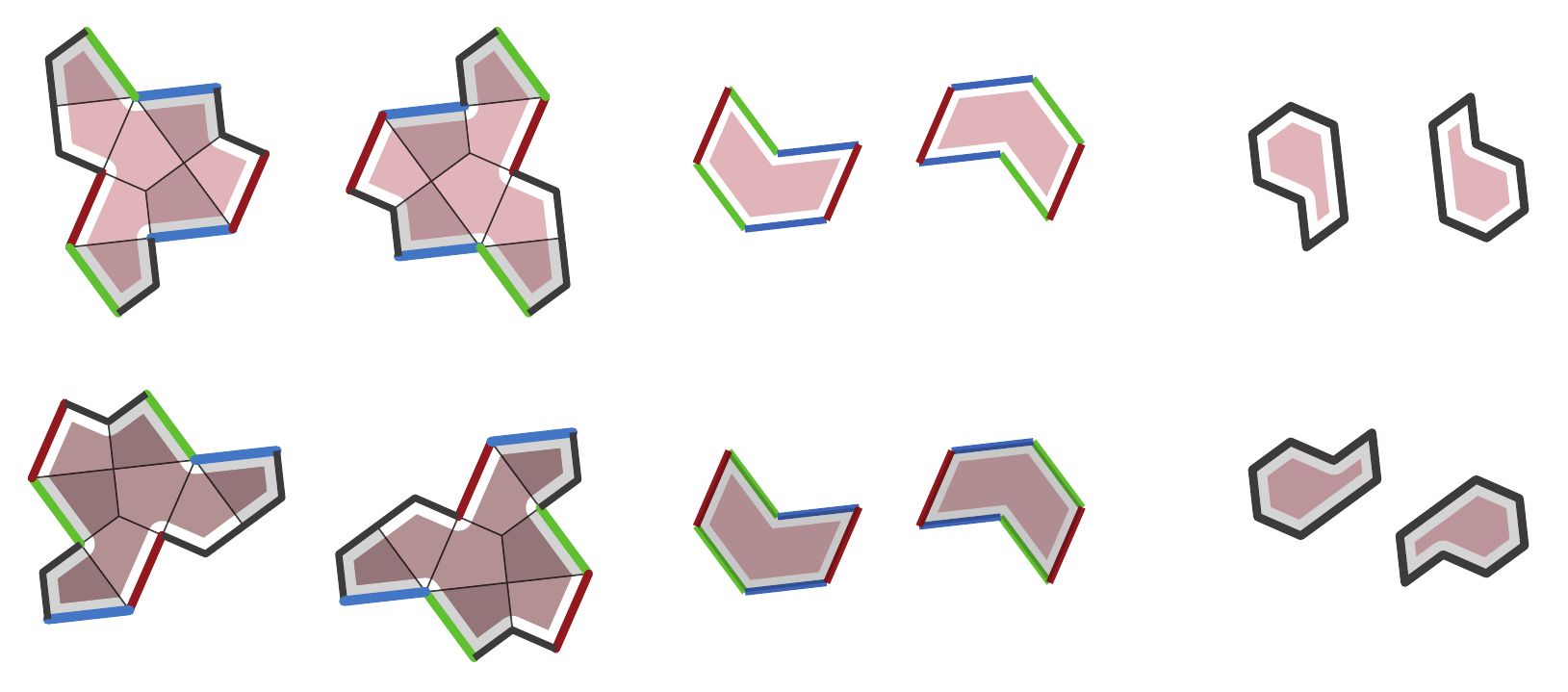}
\end{center}
\caption{Four orientations of the hat polykite that make up one
	orientation class, based on the repeated kite orientations they
	contain (left).  After contracting edges, these hats give rise to
	corresponding sets of chevrons (centre) and comets (right).}
\label{fig:kitetwocol}
\end{figure}

In the centre of  \fig{fig:kitetwocol}, we contract the sides of
the hat of length $1$ and~$2$, shown in black, to form chevrons in the
same orientations.
(This chevron is symmetric, so two orientations of the polykite in one of those sets
can give rise to identical-looking chevrons.  Those should still be
considered as different orientations of the chevron, as if it were
given an asymmetric marking.)
At right we contract the
sides of length~$\sqrt{3}$, shown coloured, to produce comets.
In Figures~\ref{fig:tt4t8} and~\ref{fig:istrips}
these tiles are coloured by orientation, matching the hats from
which they originated.

Note that given any two chevron sides in~$\mathcal{T}_4$, the
corresponding vector in~$\mathcal{T}_8$ between those two sides is
well defined: a side of a tile in~$\mathcal{T}_4$ corresponds to a
point on the boundary of the corresponding tile in~$\mathcal{T}_8$
(and adjoining sides on adjacent tiles correspond to the same point on
the boundaries of two neighbouring tiles in~$\mathcal{T}_8$), so the
vector is just the vector between those corresponding points. 

We will also make use of a tiling $\mathcal{T}_4'$, derived from
$\mathcal{T}_4$ by dividing every chevron into two congruent rhombi.
All chevrons associated with a single orientation class of hats divide
into rhombi in the same two orientations, as shown in 
\fig{fig:kitetwocol} (centre).  The balance of orientation
classes in $\mathcal{T}$ therefore implies that rhomb orientations
will occur with equal proportion in $\mathcal{T}_4'$.  (In fact,
based on tile adjacencies in hat tilings, it can be shown that 
$\mathcal{T}_4'$ is the Laves tiling $[3.6.3.6]$.)

\begin{figure}[t]
\begin{center}
\includegraphics[width=\textwidth]{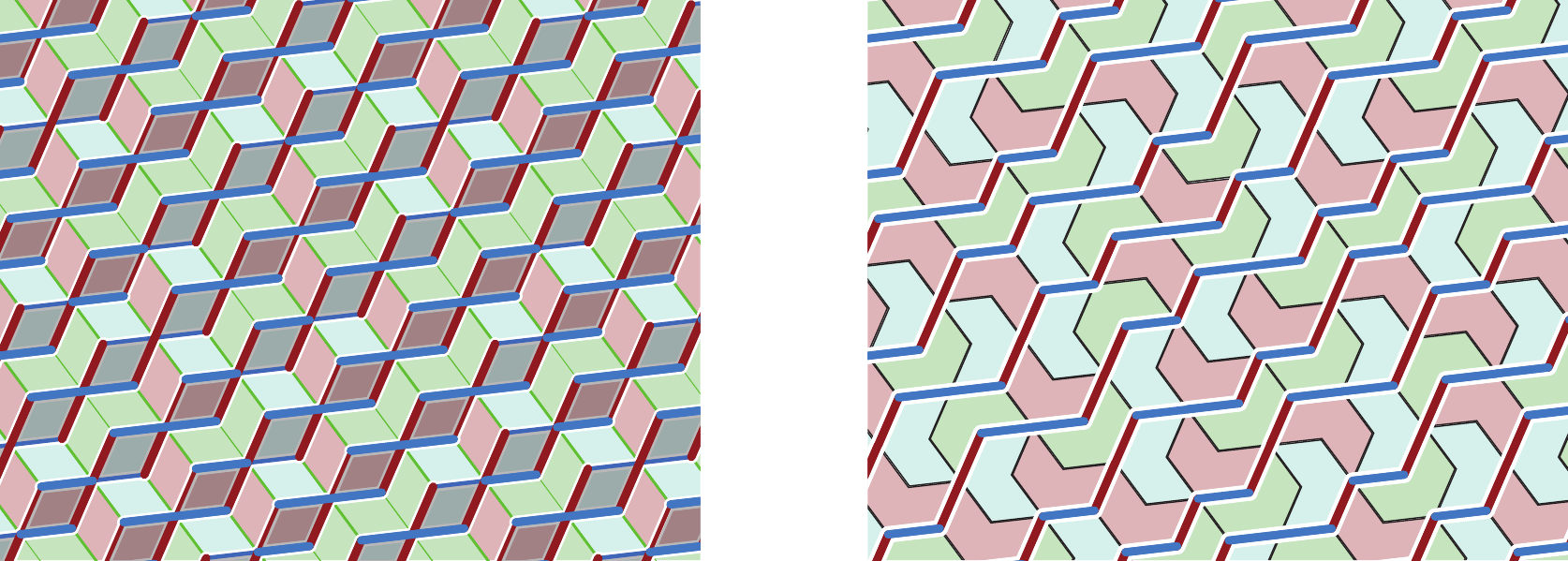}
\end{center}
\caption{\label{fig:istrips} Taking the green segments to be parallel to the lines in $\mathcal{L}_1$,  the light coloured rhombs at left form  $1$-worms in  $\mathcal{T}_4'$.  At right,  the corresponding $1$-strips are shown in  $\mathcal{T}_4$.}
\end{figure}

To show that the period-preserving affine map~$g$ 
must be a similarity, thereby deriving a contradiction, we examine
how~$g$ behaves on the partitions of~$\mathcal{T}_4$ into
structures we call ``$i$-strips''. The edges of the equilateral
triangles in the regular triangular tiling underlying $\mathcal{T}_4$
and~$\mathcal{T}_4'$ lie in three sets of parallel lines. Call those
sets $\mathcal{L}_1$, $\mathcal{L}_2$, and~$\mathcal{L}_3$. Segments
in these directions are coloured green, red, and blue in the figures.
For each $i\in\{1,2,3\}$ we can now identify a set of ``$i$-worms'' in
$\mathcal{T}_4'$.  These are pairwise disjoint, two-way infinite sequences
of rhombi, in which consecutive rhombi in one worm are adjacent 
along an edge parallel to those in~$\mathcal{L}_i$. 
\fig{fig:istrips} (left) illustrates the $1$-worms in $\mathcal{T}_4'$.
Note that the $i$-worms for any given~$i$ 
will collectively use~$2/3$ of the rhombi in $\mathcal{T}_4'$.

Clearly, the $i$-worms for a given $i$ cannot cross, and any line
parallel to those in~$\mathcal{L}_i$ passes through the $i$-worms
in the same order as any other such line passes
through them.  Furthermore, any translation symmetry preserves both 
$i$-worms themselves and the ordering of $i$-worms.   

Every $i$-worm in $\mathcal{T}_4'$ defines an $i$-strip in
$\mathcal{T}_4$, by assigning each chevron to the same strip as one
of its rhombi.  If a chevron's rhombi both belong to $i$-worms for
a given~$i$ in $\mathcal{T}_4'$, then they must be in the \emph{same}
$i$-strip in~$\mathcal{T}_4$, because the line segment between those
two rhombi lies on a line in~$\mathcal{L}_i$.  This assignment must
constitute a partition of the tiles in $\mathcal{T}_4$.  \fig{fig:istrips}
(right) illustrates the $1$-strips in $\mathcal{T}_4$.

Let $\mathbf{v}_i$ be a vector between two consecutive lines
in~$\mathcal{L}_i$, orthogonal to those lines, chosen so the pairwise
angles between those vectors are all~$120^\circ$.  Let $\mathbf{v}'_i$
be a vector orthogonal to~$\mathbf{v}_i$ and with length $1/\sqrt{3}$
times that of~$\mathbf{v}_i$, again chosen so the pairwise angles
between those vectors are all~$120^\circ$.  Note that $\sum_i
\mathbf{v}_i = 0$ and $\sum_i \mathbf{v}'_i = 0$.  Considering the
sides of rhombi in an $i$-worm in~$\mathcal{T}_4'$ that lie in
consecutive lines of~$\mathcal{L}_i$, the vector between the midpoints
of such sides is $\mathbf{v}_i \pm \mathbf{v}'_i$, where the sign
depends on the orientation of the rhomb.  Thus, if the vector
between the midpoints of any two $\mathcal{L}_i$-aligned rhomb sides in an
$i$-worm is $a \mathbf{v}_i + b \mathbf{v}'_i$,
then between those two sides there are $(a+b)/2$ rhombi of one
orientation and $(a-b)/2$ of the other orientation.

The translation symmetries of the strongly periodic tiling~$\mathcal{T}$ 
correspond to a subgroup of the symmetries of~$\mathcal{T}_4'$.  
There are only finitely many orbits of rhombi
under the action of the subgroup, so in any
$i$-worm~$\mathcal{S}$ there must be two rhombi in the same orbit.
The translation mapping one rhomb to the other is a translation symmetry of
$\mathcal{T}_4'$, and therefore maps $i$-worms to $i$-worms. Because it maps
$\mathcal{S}$ to itself and preserves the ordering of $i$-worms, it
must map every $i$-worm to itself.  If that translation is by a
vector $a \mathbf{v}_i + b \mathbf{v}'_i$, it follows that $b=0$,
because otherwise rhombi of the two orientations that make up
these $i$-worms would appear in the tiling in different proportions.

Thus for each $i$ we have some positive integer~$a_i$, such that a
translation by $a_i \mathbf{v}_i$ is a symmetry of both~$\mathcal{T}_4'$
and~$\mathcal{T}_4$.  In~$\mathcal{T}_4$, translation by this vector 
must map every $i$-strip to itself. 
Let~$a$ be the lowest common multiple of the~$a_i$.  Translation by 
$a \mathbf{v}_i$ is also a symmetry of $\mathcal{T}_4$ that 
sends each $i$-strip to itself.

By construction, translation by 
$a \mathbf{v}_i$ in $\mathcal{T}_4$ 
corresponds to a translation symmetry of $\mathcal{T}$\!,
and therefore also to some translation symmetry of~$\mathcal{T}_8$.
We may calculate the precise translation vectors in~$\mathcal{T}$ 
corresponding to each $a \mathbf{v}_i$ based on the tiles in
any $i$-strip in~$\mathcal{T}_4$ (between any two lines
in~$\mathcal{L}_i$ related by a translation by that vector). Every
such $i$-strip (and choice of lines) must produce the same vector
in~$\mathcal{T}_8$.  \fig{fig:kitetrans} shows the
corresponding translations between pairs of chevron edges, 
oriented by way of example to be parallel to  $\mathcal{L}_1$.
  For each such pair, first the
vector within the chevron is indicated, then the corresponding
comet vector between points on the boundary of the comet, then that vector
decomposed into components parallel to and orthogonal to the chevron 
sides between
which the vector is drawn.  The corresponding hats are shown to aid in 
verifying the calculation. 
Rotating, reflecting or reversing the
direction of the chevron vector has the same effect on the
comet  vector.

\begin{figure}[htp!]
\begin{center}
\includegraphics[width=\hsize]{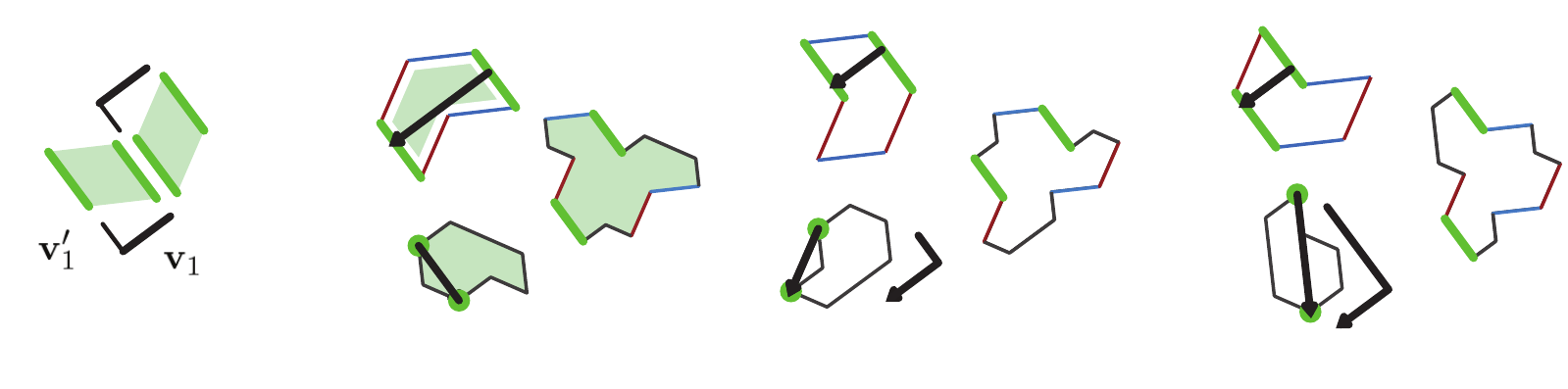}
\end{center}
\caption{Corresponding translations for the chevron and comet}
\label{fig:kitetrans}
\end{figure}

Note that in the first case, the comet vector is parallel to the
sides between which the chevron vector is drawn; the second and
third cases have equal components orthogonal to those sides.  For the
orthogonal component of the corresponding translations
in~$\mathcal{T}_8$ to be equal for all $i$-strips, it follows that
every $i$-strip must have the same proportion of the second and third
cases relative to the first case.  As the first case corresponds
exactly to one of the three sets of orientations that occur in equal
proportions in any tiling, the first case must thus be a third of the
chevrons in any $i$-strip, while the second and third cases (which
together correspond to the other two sets of orientations; however,
each case does not correspond to a single set of orientations) in that
figure must add to two thirds of the chevrons.

The chevrons  in the first case have orthogonal
translation vector $2{\bf v}_i$ in $\mathcal{T}_4$,  zero in
$\mathcal{T}_8$. The remaining two thirds of the chevrons have
orthogonal translation ${\bf v}_i$ in both $\mathcal{T}_4$ and
$\mathcal{T}_8$.  Thus if $a{\bf v}_i$ is a period of the $i$-strips
in $\mathcal{T}_4$, then each of its $i$-strips has  $a/4$ chevrons
from the first case and $a/2$ in the other two cases. This period
corresponds to  a translation symmetry of  $(a/2){\bf v}_i$ in
$\mathcal{T}_8$.  Since the sum of ${\bf v}_i$ over $i=1,2,3$ is
zero (as noted above), the sum of those three orthogonal components
of translation vectors in $\mathcal{T}_8$ is also zero.

Therefore their parallel components must also add to
zero.  But $\sum_i b_i \mathbf{v}'_i = 0$ if and only if all the
$b_i$~are equal; say they all equal~$b$.  That means the three
translation vectors in~$\mathcal{T}_8$ (which are~${\frac{a}{2}
\mathbf{v}_i + b \mathbf{v}'_i}$) are at $120^\circ$~angles to each
other.  In that case, the period-preserving affine transformation~$g$ 
must scale uniformly in every
direction, and is therefore a similarity.  But we know from the discussion
above that~$g$ cannot be a similarity, and so we arrive at a contradiction,
ruling out the initial supposition that~$\mathcal{T}$ was strongly periodic.

\section{Clustering of tiles}
\label{sec:clusters}

As discussed in \secref{sec:discussion}, tilings by the hat polykite
are composed of certain clusters of tiles.  These clusters can be 
used to define simplified tile shapes that we call \textit{metatiles}.
The metatiles inherit matching rules from the boundaries of the hats
that they contain. Furthermore, through a set of substitution rules 
they form larger, combinatorially equivalent
supertiles that fit together following the same matching rules.  In
this section, we give a precise definition of how tiles are assigned
to clusters, and a computer-assisted proof by case analysis that
this assignment does result in the clusters claimed, fitting together
in accordance with the matching rules given.  

\begin{figure}[ht!]
\captionsetup{margin=0pt}%
\begin{center}
\subfloat[Cluster $T$]{%
  \begin{tikzpicture}[x=5mm,y=5mm]
  \draw[\colcluster,ultra thick] \vcoords{2}{-2} -- \vcoords{8}{-2} --
    \vcoords{2}{4} -- cycle;
  \tileA{60}{0}{0}{$T_1$};
  \draw \vcoords{2}{3} -- \vcoords{2}{4};
  \draw \vcoords{7}{-1} -- \vcoords{8}{-2};
  \markpt{2}{-2};
  \markpt{8}{-2};
  \markpt{2}{4};
  \vctxt{0}{2}{$B^+$};
  \vctxt{5}{-3}{$A^-$};
  \vctxt{5}{2}{$A^-$};
\end{tikzpicture}%
} \qquad \subfloat[Cluster $H$]{%
\begin{tikzpicture}[x=5mm,y=5mm]
  \draw[\colcluster,ultra thick] \vcoords{2}{0} -- \vcoords{4}{-2} --
    \vcoords{12}{-2} -- \vcoords{12}{0} -- \vcoords{4}{8} --
    \vcoords{2}{8} -- cycle;
  \tileAr{240}{0}{2}{$H_1$};
  \tileA{60}{-1}{2}{$H_2$};
  \tileA{60}{1}{1}{$H_3$};
  \tileA{300}{0}{1}{$H_4$};
  \markpt{2}{0};
  \markpt{4}{-2};
  \markpt{10}{-2};
  \markpt{12}{-2};
  \markpt{12}{0};
  \markpt{6}{6};
  \markpt{4}{8};
  \markpt{2}{8};
  \markpt{2}{2};
  \vctxt{3}{-2.5}{$X^+$};
  \vctxt{8}{-3}{$B^-$};
  \vctxt{11.5}{-3}{$X^-$};
  \vctxt{14}{-2}{$X^+$};
  \vctxt{9}{4}{$B^-$};
  \vctxt{5.5}{7.5}{$X^-$};
  \vctxt{3}{9.5}{$X^+$};
  \vctxt{0}{5.5}{$A^+$};
  \vctxt{1}{1.5}{$X^-$};
\end{tikzpicture}%
} \\ \subfloat[Cluster $P$]{%
\begin{tikzpicture}[x=5mm,y=5mm]
  \draw[\colcluster,ultra thick] \vcoords{0}{0} -- \vcoords{4}{-4} --
    \vcoords{12}{-4} -- \vcoords{8}{0} -- cycle;
  \tileA{0}{0}{0}{$P_1$};
  \tileA{60}{1}{0}{$P_2$};
  \draw \vcoords{11}{-3} -- \vcoords{12}{-4};
  \markpt{0}{0};
  \markpt{2}{-2};
  \markpt{4}{-4};
  \markpt{6}{-4};
  \markpt{12}{-4};
  \markpt{10}{-2};
  \markpt{8}{0};
  \markpt{6}{0};
  \vctxt{0.5}{-1.5}{$L$};
  \vctxt{2.5}{-3.5}{$X^-$};
  \vctxt{5}{-5}{$X^+$};
  \vctxt{9}{-5}{$A^-$};
  \vctxt{11.5}{-2.5}{$L$};
  \vctxt{9.5}{-0.5}{$X^-$};
  \vctxt{7}{1}{$X^+$};
  \vctxt{3}{1.5}{$B^+$};
\end{tikzpicture}%
} \qquad \subfloat[Cluster $F$]{%
\begin{tikzpicture}[x=5mm,y=5mm]
  \draw[\colcluster,ultra thick] \vcoords{0}{0} -- \vcoords{4}{-4} --
    \vcoords{10}{-4} -- \vcoords{10}{-2} -- \vcoords{8}{0} -- cycle;
  \tileA{0}{0}{0}{$F_1$};
  \tileA{60}{1}{0}{$F_2$};
  \markpt{0}{0};
  \markpt{2}{-2};
  \markpt{4}{-4};
  \markpt{6}{-4};
  \markpt{8}{-4};
  \markpt{10}{-4};
  \markpt{10}{-2};
  \markpt{8}{0};
  \markpt{6}{0};
  \vctxt{0.5}{-1.5}{$L$};
  \vctxt{2.5}{-3.5}{$X^-$};
  \vctxt{5}{-5}{$X^+$};
  \vctxt{7.5}{-5}{$L$};
  \vctxt{9.5}{-5}{$X^-$};
  \vctxt{11.5}{-3.5}{$F^+$};
  \vctxt{9.5}{-0.5}{$F^-$};
  \vctxt{7}{1}{$X^+$};
  \vctxt{3}{1.5}{$B^+$};
\end{tikzpicture}%
}%
\end{center}
\caption{The four clusters}
\label{fig:tileaclusters}
\end{figure}
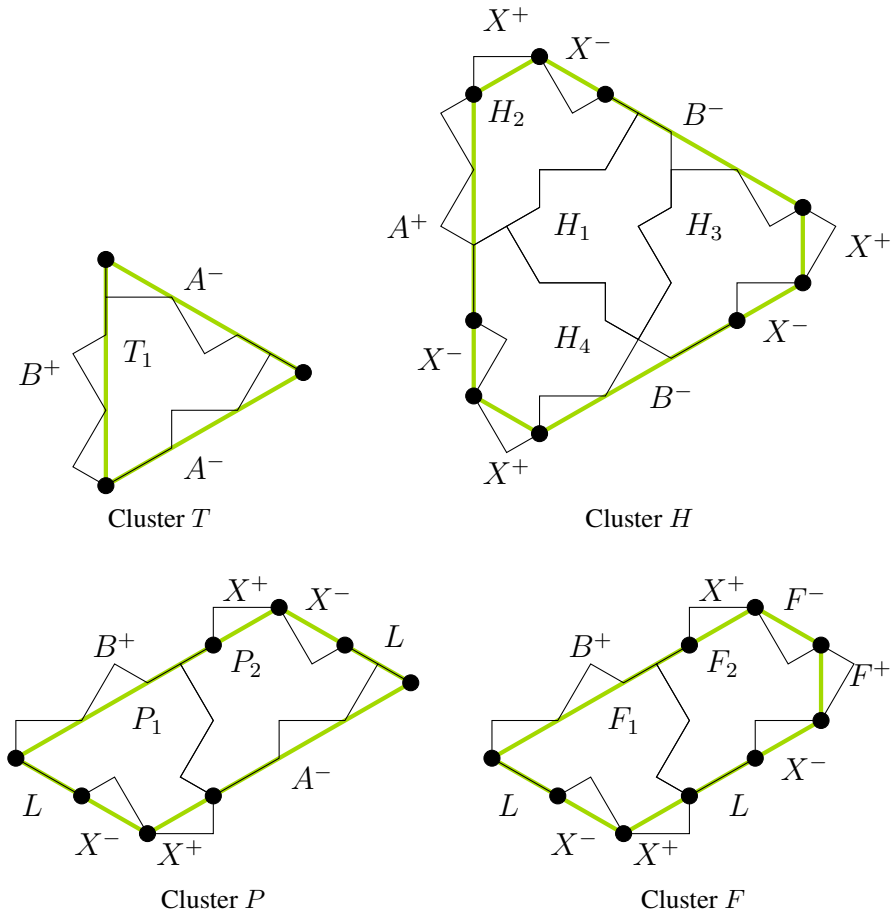

The clusters and their associated metatiles
are shown in Figure~\ref{fig:tileaclusters}.
Each metatile is a convex polyiamond outlined in
\textcolor{\colcluster}{\textbf{lime}}; its hats are overlaid, and 
each is given a unique label.
The union of the polykites in a cluster approximates the shape of
its metatile, but with some triangular indentations and protrusions along its
boundary.  At two corners of cluster~$T$, and one of cluster~$P$, an
additional line is drawn from a corner of a polykite to a corner of
the boundary of the polyiamond; this line clarifies  how an indentation to a
corner of the polyiamond is uniquely associated with one of that
polyiamond's sides.  

The boundaries of the four metatiles
are divided into labelled segments by marked points.
The labels represent matching rules to be obeyed in tilings 
by the metatiles.  To satisfy the matching rules, the four
metatiles must form a tiling using copies that are only rotated and
not reflected; edge segments marked $A^+$ and~$A^-$ must adjoin on
adjacent tiles of the tiling; likewise, edge segments $B^+$ and~$B^-$,
$X^+$ and~$X^-$, $F^+$ and~$F^-$, and $L$ and~$L$ must adjoin.
We will show in \secref{sec:subst} that any
tiling by the metatiles has a substitution structure: the tiles may be
grouped (after bisecting some tiles) into supertiles that satisfy
combinatorially equivalent matching rules.  This grouping process
implies that that no tiling by the metatiles is periodic.  Furthermore,
the
substitution structure allows the metatiles to tile arbitrarily
large regions of the plane, and hence the whole plane, implying that 
they form an aperiodic set.

In this section we establish the following result:

\begin{theorem}
\label{thm:clusters}
Any tiling by the hat polykite can be divided into the clusters shown
in Figure~\ref{fig:tileaclusters} (or reflections thereof, but not
mixing reflected and non-reflected clusters), satisfying the given
matching rules, with the resulting tiling by metatiles having the
same symmetries as the original tiling by polykites.
\end{theorem}

Since inspection
of the cluster shapes shows that, conversely, any tiling by metatiles
induces one by the hat polykite (for example, $A^+$ and $A^-$ are
equal and opposite modifications to the shape of an edge and are
consistent wherever they appear in the clusters), the division into
clusters suffices as part of showing that the hat polykite is an
aperiodic monotile.

The proof of Theorem~\ref{thm:clusters} is computer-assisted.  
We define rules (\secref{sec:clusters:rules}) for
assigning the labels from \fig{fig:tileaclusters} 
to tiles in any tiling by the hat polykite.  Those rules
assign a label to a tile based only on its immediate neighbours.
Because no arbitrary choices are involved in the rules, they preserve
all symmetries of the tiling.  It then remains to show that (a)~the
labels assigned do induce a division into the clusters shown, and
(b)~the clusters adjoin other clusters in accordance with the matching
rules.  Because the matching rules do not permit a reflected cluster
to adjoin a non-reflected cluster, it then follows that either no
clusters are reflected or all clusters are reflected. Without
loss of generality we assume in \secref{sec:subst} that no clusters
are reflected.

Both (a) and~(b) may be demonstrated by a case analysis of $2$-patches
of hats.  Ideally, we would restrict our analysis to precisely those
$2$-patches that appear in tilings by the hat.  Such an approach is
unrealistic, however, as it requires foreknowledge of the space of
tilings we are attempting to understand.  In practice the list of
$2$-patches can include false positives that do not occur in any
tilings, as long as our analysis produces valid results for them
as well (and as long as the list contains every $2$-patch that can
occur in a tiling).

For the purposes of our proof we worked with the~$188$
``surroundable $2$-patches'': $2$-patches that can be surrounded
at least once more to form a $3$-patch.  We generated this set of~$188$
patches computationally.  Specifically,
we modified Kaplan's SAT-based software~\cite{Kaplan} to
enumerate all distinct $3$-patches of hats, and extracted the unique 
$2$-patches in their centres.  We validated this list by creating an
independent implementation based on brute-force search with backtracking;
the source code for this implementation is available with our article.
This list certainly includes false positives---a more sophisticated case
analysis shows that at most $63$ of the $188$ surroundable $2$-patches
can actually appear in a tiling by hats. However, all $188$ of them 
satisfy the conditions given in this section, allowing us to obtain
the results we need with simpler and more transparent algorithms.

It is also possible to demonstrate both (a) and~(b) by a shorter 
case analysis using only $1$-patches.  
However, an analysis based on $1$-patches is
more complicated because the classification rules 
in \secref{sec:clusters:rules} assume that all the
neighbours of a tile are known.  Those rules can therefore not be applied
directly to the outer tiles in a $1$-patch, making it necessary to work
with partial information about which labels are consistent with such a
tile.  For more details of this alternative case analysis, see
Appendix~\ref{sec:patches}. 

An analysis of tilings based on the enumeration of patches 
depends on the
assumption that it is only necessary to consider tilings where all
polykites are aligned to the same underlying $[3.4.6.4]$ Laves
tiling.  This assumption is not in fact obvious for tilings by
polykites or other polyforms in general; it is justified in
Appendix~\ref{sec:align}.

For each of the 188 surroundable $2$-patches, the classification
rules of \secref{sec:clusters:rules} determine labels for the tiles
in the patch's interior (comprising the central tile and its
neighbours).  We may then demonstrate~(a) by verifying that when
the central tile of a patch has a given label from one of the
clusters shown in \fig{fig:tileaclusters}, its neighbours in that
cluster appear with the correct labels in the expected positions
and orientations within the patch.  This ``within-cluster''
verification process is explained in detail in
\secref{sec:clusters:within}.  Similarly, in \secref{sec:clusters:between}
we describe a ``between-cluster'' verification process that
demonstrates~(b).  In particular, we show that when a patch's central
tile is adjacent to a tile with a label from a different cluster,
their adajcency relationship is consistent with the labelled edge 
segments that define the matching rules for the clusters.
The reference software mentioned above performs
all of these checks on the 188 surroundable $2$-patches.

\FloatBarrier

\subsection{Classification rules for the hat polykite}
\label{sec:clusters:rules}

\fig{fig:class} presents the eight classification rules for tiles.
Each rule shows a (labelled) central tile and some of its neighbours.  The 
order
of the rules is significant: the first rule that matches determines
the label on the central tile.  For each rule, if all the neighbours
shown are present, and no previous rule matched, the tile acquires
the label indicated.  The last rule is not constrained by any
neighbours, and therefore always matches if no previous rule did.
Thus every tile is assigned some label.

These rules do not distinguish between the labels~$P_1$ and~$F_1$:
the last rule assigns all such tiles the common label~$FP_1$.  The
within-cluster and between-cluster checks that follow are all expressed
in terms of this composite label.  An~$FP_1$ tile can always
be relabelled as either~$P_1$ or~$F_1$ later, depending on whether it
has a neighbour labelled~$P_2$ or~$F_2$ in the correct position and
orientation.

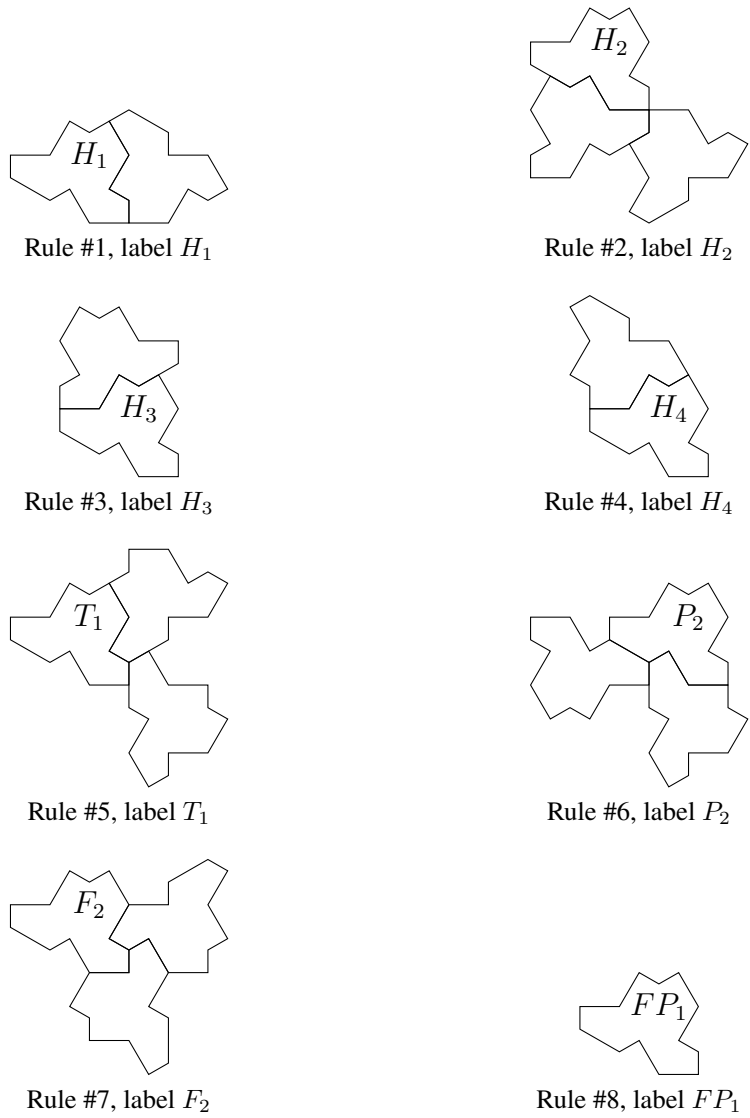
\begin{figure}[htp!]
\captionsetup{margin=0pt,justification=raggedright}%
\begin{center}
\subfloat[Rule \#1, label $H_1$]{%
\begin{minipage}[b]{6cm}
\begin{center}
\begin{tikzpicture}[x=3mm,y=3mm]
  \tileA{0}{0}{0}{$H_1$};
  \tileAr{300}{1}{0}{};
\end{tikzpicture}%
\end{center}%
\end{minipage}%
} \qquad \subfloat[Rule \#2, label $H_2$]{%
\begin{minipage}[b]{6cm}
\begin{center}
\begin{tikzpicture}[x=3mm,y=3mm]
  \tileA{0}{0}{0}{$H_2$};
  \tileAr{180}{1}{-1}{};
  \tileA{60}{2}{-2}{};
\end{tikzpicture}%
\end{center}%
\end{minipage}%
} \\ \subfloat[Rule \#3, label $H_3$]{%
\begin{minipage}[b]{6cm}
\begin{center}
\begin{tikzpicture}[x=3mm,y=3mm]
  \tileA{0}{0}{0}{$H_3$};
  \tileAr{180}{0}{1}{};
\end{tikzpicture}%
\end{center}%
\end{minipage}%
} \qquad \subfloat[Rule \#4, label $H_4$]{%
\begin{minipage}[b]{6cm}
\begin{center}
\begin{tikzpicture}[x=3mm,y=3mm]
  \tileA{0}{0}{0}{$H_4$};
  \tileAr{300}{-1}{1}{};
\end{tikzpicture}%
\end{center}%
\end{minipage}%
} \\ \subfloat[Rule \#5, label $T_1$]{%
\begin{minipage}[b]{6cm}
\begin{center}
\begin{tikzpicture}[x=3mm,y=3mm]
  \tileA{0}{0}{0}{$T_1$};
  \tileA{60}{1}{0}{};
  \tileA{300}{2}{-1}{};
\end{tikzpicture}%
\end{center}%
\end{minipage}%
} \qquad \subfloat[Rule \#6, label $P_2$]{%
\begin{minipage}[b]{6cm}
\begin{center}
\begin{tikzpicture}[x=3mm,y=3mm]
  \tileA{0}{0}{0}{$P_2$};
  \tileA{180}{0}{-1}{};
  \tileA{300}{1}{-1}{};
\end{tikzpicture}%
\end{center}%
\end{minipage}%
} \\ \subfloat[Rule \#7, label $F_2$]{%
\begin{minipage}[b]{6cm}
\begin{center}
\begin{tikzpicture}[x=3mm,y=3mm]
  \tileA{0}{0}{0}{$F_2$};
  \tileA{120}{2}{-2}{};
  \tileA{240}{2}{0}{};
\end{tikzpicture}%
\end{center}%
\end{minipage}%
} \qquad \subfloat[Rule \#8, label $FP_1$]{%
\begin{minipage}[b]{6cm}
\begin{center}
\begin{tikzpicture}[x=3mm,y=3mm]
  \tileA{0}{0}{0}{$FP_1$};
\end{tikzpicture}%
\end{center}%
\end{minipage}%
}%
\end{center}
\caption{Classification rules}
\label{fig:class}
\end{figure}

\FloatBarrier

\subsection{Within-cluster matching checks for the hat polykite}
\label{sec:clusters:within}

Let~$L_1$ and~$L_2$ be the labels of neighbouring tiles in one of the
clusters shown in \fig{fig:tileaclusters}.  To verify that
tiles can be grouped uniquely into copies of these clusters, we must
show that when the central tile of a surroundable $2$-patch has the
label~$L_1$, it has a neighbour labelled~$L_2$ in the expected position
and orientation shown in the cluster.  In practice, we do not need to
check all such pairs of labels---it suffices to choose a subset of
labels that define spanning trees of the neighbour relationships within each
cluster.  For~$H$, we choose the spanning tree that connects~$H_1$ to
its three neighbours.

\fig{fig:within} presents the eight within-cluster checks that must be
applied to each of the surroundable $2$-patches.  For each rule and each
patch, if the rule's shaded tile has the same label as the patch's
central tile, then the patch must also include the neighbour shown in the
rule.  As noted above, these rules do not distinguish between~$P_1$ and~$F_1$; 
it suffices to check that an~$FP_1$ tile has either of~$P_2$ or~$F_2$ as its
neighbour.  Because these rules hold for all surroundable $2$-patches,
the labels assigned in \secref{sec:clusters:rules} induce a division
of the tiles in any hat tiling into the~$H$, $T$, $P$, and~$F$ clusters.

\begin{figure}[htp!]
\captionsetup{margin=0pt,justification=raggedright}%
\begin{center}
\subfloat[$H_1$ neighbour $H_2$]{%
\begin{minipage}[b]{6cm}
\begin{center}
\begin{tikzpicture}[x=3mm,y=3mm]
  \ftileA{0}{0}{0}{$H_1$};
  \tileAr{180}{0}{1}{$H_2$};
\end{tikzpicture}%
\end{center}%
\end{minipage}%
} \qquad \subfloat[$H_1$ neighbour $H_3$]{%
\begin{minipage}[b]{6cm}
\begin{center}
\begin{tikzpicture}[x=3mm,y=3mm]
  \ftileA{0}{0}{0}{$H_1$};
  \tileAr{180}{1}{-1}{$H_3$};
\end{tikzpicture}%
\end{center}%
\end{minipage}%
} \\ \subfloat[$H_1$ neighbour $H_4$]{%
\begin{minipage}[b]{6cm}
\begin{center}
\begin{tikzpicture}[x=3mm,y=3mm]
  \ftileA{0}{0}{0}{$H_1$};
  \tileAr{300}{1}{0}{$H_4$};
\end{tikzpicture}%
\end{center}%
\end{minipage}%
} \qquad \subfloat[$H_2$ neighbour $H_1$]{%
\begin{minipage}[b]{6cm}
\begin{center}
\begin{tikzpicture}[x=3mm,y=3mm]
  \ftileA{0}{0}{0}{$H_2$};
  \tileAr{180}{1}{-1}{$H_1$};
\end{tikzpicture}%
\end{center}%
\end{minipage}%
} \\ \subfloat[$H_3$ neighbour $H_1$]{%
\begin{minipage}[b]{6cm}
\begin{center}
\begin{tikzpicture}[x=3mm,y=3mm]
  \ftileA{0}{0}{0}{$H_3$};
  \tileAr{180}{0}{1}{$H_1$};
\end{tikzpicture}%
\end{center}%
\end{minipage}%
} \qquad \subfloat[$H_4$ neighbour $H_1$]{%
\begin{minipage}[b]{6cm}
\begin{center}
\begin{tikzpicture}[x=3mm,y=3mm]
  \ftileA{0}{0}{0}{$H_4$};
  \tileAr{300}{-1}{1}{$H_1$};
\end{tikzpicture}%
\end{center}%
\end{minipage}%
}%
\end{center}
\caption{Within-cluster matching checks (part 1)}
\label{fig:within}
\end{figure}
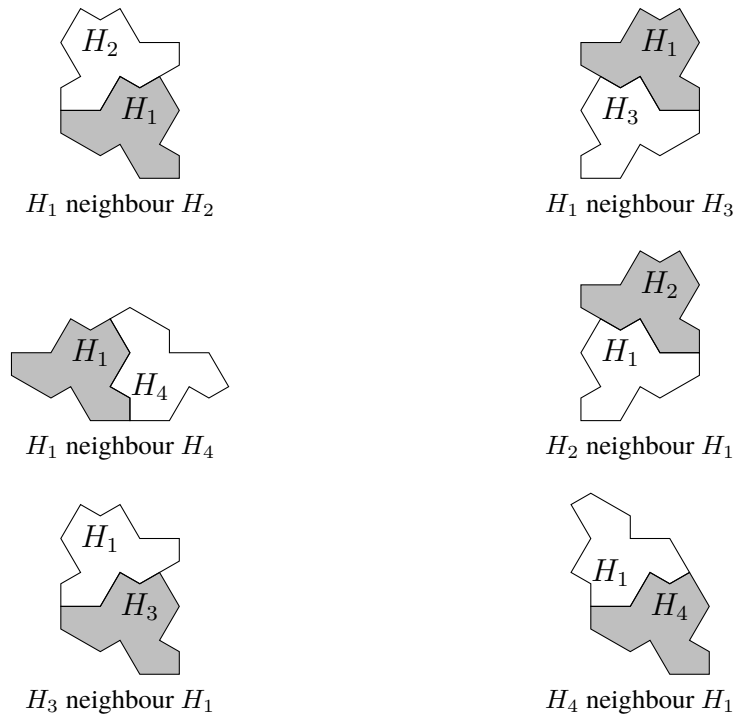
\begin{figure}[htp!]
\ContinuedFloat
\captionsetup{margin=0pt,justification=raggedright}%
\begin{center}
\subfloat[$FP_1$ neighbour $P_2$ or $F_2$ ($Z \in \{P_2, F_2\}$)]{%
\begin{minipage}[b]{6cm}
\begin{center}
\begin{tikzpicture}[x=3mm,y=3mm]
  \ftileA{0}{0}{0}{$FP_1$};
  \tileA{60}{1}{0}{$Z$};
\end{tikzpicture}%
\end{center}%
\end{minipage}%
} \qquad \subfloat[$P_2$ or $F_2$ neighbour $FP_1$ ($Y \in \{P_2, F_2\}$)]{%
\begin{minipage}[b]{6cm}
\begin{center}
\begin{tikzpicture}[x=3mm,y=3mm]
  \ftileA{0}{0}{0}{$Y$};
  \tileA{300}{-1}{1}{$FP_1$};
\end{tikzpicture}%
\end{center}%
\end{minipage}%
}%
\end{center}
\caption{Within-cluster matching checks (part 2)}
\label{fig:within:2}
\end{figure}
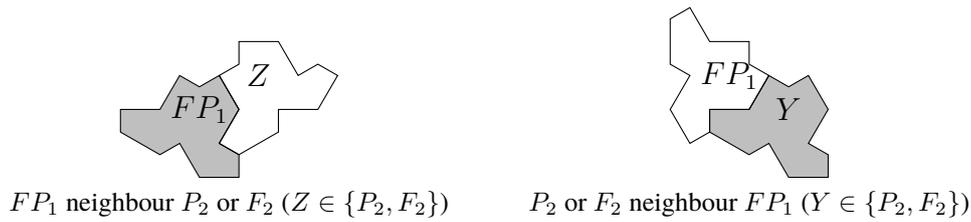

\FloatBarrier

\subsection{Between-cluster matching checks for the hat polykite}
\label{sec:clusters:between}

Let~$C$ be one of the four clusters in \fig{fig:tileaclusters}, and 
let~$E$ be any of its marked edge segments.  We can enumerate  all 
combinations of an edge segment~$E'$, belonging to a cluster~$C'$,
which are permitted to adjoin~$E$ according to the matching rules.
If any one tile in~$C'$ that adjoins~$E'$ is in the correct position and 
orientation relative to any one tile in~$C$ that adjoins~$E$, it follows as a
result of the within-cluster checks that the entire edge segment properly 
matches between the two clusters.  Furthermore, because the matching 
rules on the
boundaries of $F_1$ and~$P_1$ are identical, it suffices to handle both
using the single label~$FP_1$.  So for each~$E$ we
pick one tile in~$C$, and for each choice of~$E'$, we pick one tile
in~$C'$ that would be a neighbour of the tile picked in~$C$.  We then
check that, in each surroundable $2$-patch
whose central tile has the label of the tile picked in~$C$, there is a
neighbour in a position and orientation and with a label that matches
one of the possibilities for a tile picked in~$C'$ for one choice of~$E'$.

Figure~\ref{fig:between} presents the between-cluster checks that must be
applied to each of the surroundable $2$-patches.  Each diagram shows a 
shaded tile from cluster~$C$ 
and its neighbour from cluster~$C'$, with labels on both, and represents a
tile on one
side of a cluster edge and some options for a tile on the other side
of that edge.  In some cases, there are two alternatives listed
for the same edge, with separate figures for each, marked in the form
``(alternative~$k$ of~2)''.  Also, in some cases there are multiple 
options for the labels on one or both tiles, shown in a single figure.  The
central tile in every $2$-patch that can occur in a tiling should be
checked against all figures shown here with that central tile's label
as one of the options for the shaded tile; if, for all such
$2$-patches, one of the alternatives listed for that edge is present
with one of the labels indicated, then the clusters adjoin other
clusters in accordance with the matching rules.  (Where multiple
alternatives are listed for the same edge, only one of those
alternatives needs to pass the check.)

\begin{figure}[htp!]
\captionsetup{margin=0pt,justification=raggedright}%
\begin{center}
\subfloat[$H$ edge $A^+$ (alternative 1 of 2) ($Z \in \{T_1, P_2\}$)]{%
\begin{minipage}[b]{6cm}
\begin{center}
\begin{tikzpicture}[x=3mm,y=3mm]
  \ftileA{0}{0}{0}{$H_2$};
  \tileA{60}{-1}{0}{$Z$};
\end{tikzpicture}%
\end{center}%
\end{minipage}%
} \qquad \subfloat[$H$ edge $A^+$ (alternative 2 of 2)]{%
\begin{minipage}[b]{6cm}
\begin{center}
\begin{tikzpicture}[x=3mm,y=3mm]
  \ftileA{0}{0}{0}{$H_2$};
  \tileA{300}{-1}{1}{$T_1$};
\end{tikzpicture}%
\end{center}%
\end{minipage}%
} \\ \subfloat[$H$ upper edge $B^-$ ($Z \in \{T_1, FP_1\}$)]{%
\begin{minipage}[b]{6cm}
\begin{center}
\begin{tikzpicture}[x=3mm,y=3mm]
  \ftileA{0}{0}{0}{$H_1$};
  \tileAr{120}{0}{-1}{$Z$};
\end{tikzpicture}%
\end{center}%
\end{minipage}%
} \qquad \subfloat[$H$ lower edge $B^-$ ($Z \in \{T_1, FP_1\}$)]{%
\begin{minipage}[b]{6cm}
\begin{center}
\begin{tikzpicture}[x=3mm,y=3mm]
  \ftileA{0}{0}{0}{$H_3$};
  \tileA{300}{-1}{0}{$Z$};
\end{tikzpicture}%
\end{center}%
\end{minipage}%
} \\ \subfloat[$T$ upper edge $A^-$]{%
\begin{minipage}[b]{6cm}
\begin{center}
\begin{tikzpicture}[x=3mm,y=3mm]
  \ftileA{0}{0}{0}{$T_1$};
  \tileA{60}{1}{0}{$H_2$};
\end{tikzpicture}%
\end{center}%
\end{minipage}%
} \qquad \subfloat[$T$ or $P$ lower edge $A^-$ ($Y \in \{T_1, P_2\}$)]{%
\begin{minipage}[b]{6cm}
\begin{center}
\begin{tikzpicture}[x=3mm,y=3mm]
  \ftileA{0}{0}{0}{$Y$};
  \tileA{300}{1}{-1}{$H_2$};
\end{tikzpicture}%
\end{center}%
\end{minipage}%
} \\ \subfloat[$T$, $P$ or $F$ edge $B^+$ ($Y \in \{T_1, FP_1\}$) ($Z \in \{H_3, H_4\}$)]{%
\begin{minipage}[b]{6cm}
\begin{center}
\begin{tikzpicture}[x=3mm,y=3mm]
  \ftileA{0}{0}{0}{$Y$};
  \tileA{300}{-1}{1}{$Z$};
\end{tikzpicture}%
\end{center}%
\end{minipage}%
} \qquad \subfloat[$F$ edge $F^+$]{%
\begin{minipage}[b]{6cm}
\begin{center}
\begin{tikzpicture}[x=3mm,y=3mm]
  \ftileA{0}{0}{0}{$F_2$};
  \tileA{120}{2}{-2}{$F_2$};
\end{tikzpicture}%
\end{center}%
\end{minipage}%
} \\ \subfloat[$F$ edge $F^-$]{%
\begin{minipage}[b]{6cm}
\begin{center}
\begin{tikzpicture}[x=3mm,y=3mm]
  \ftileA{0}{0}{0}{$F_2$};
  \tileA{240}{2}{0}{$F_2$};
\end{tikzpicture}%
\end{center}%
\end{minipage}%
} \qquad \subfloat[$X^+$ edge at top of polykite ($Y \in \{H_2, P_2, F_2\}$) (alternative 1 of 2) ($Z \in \{H_2, P_2\}$)]{%
\begin{minipage}[b]{6cm}
\begin{center}
\begin{tikzpicture}[x=3mm,y=3mm]
  \ftileA{0}{0}{0}{$Y$};
  \tileA{240}{1}{1}{$Z$};
\end{tikzpicture}%
\end{center}%
\end{minipage}%
}%
\end{center}
\caption{Between-cluster matching checks (part 1)}
\label{fig:between}
\end{figure}
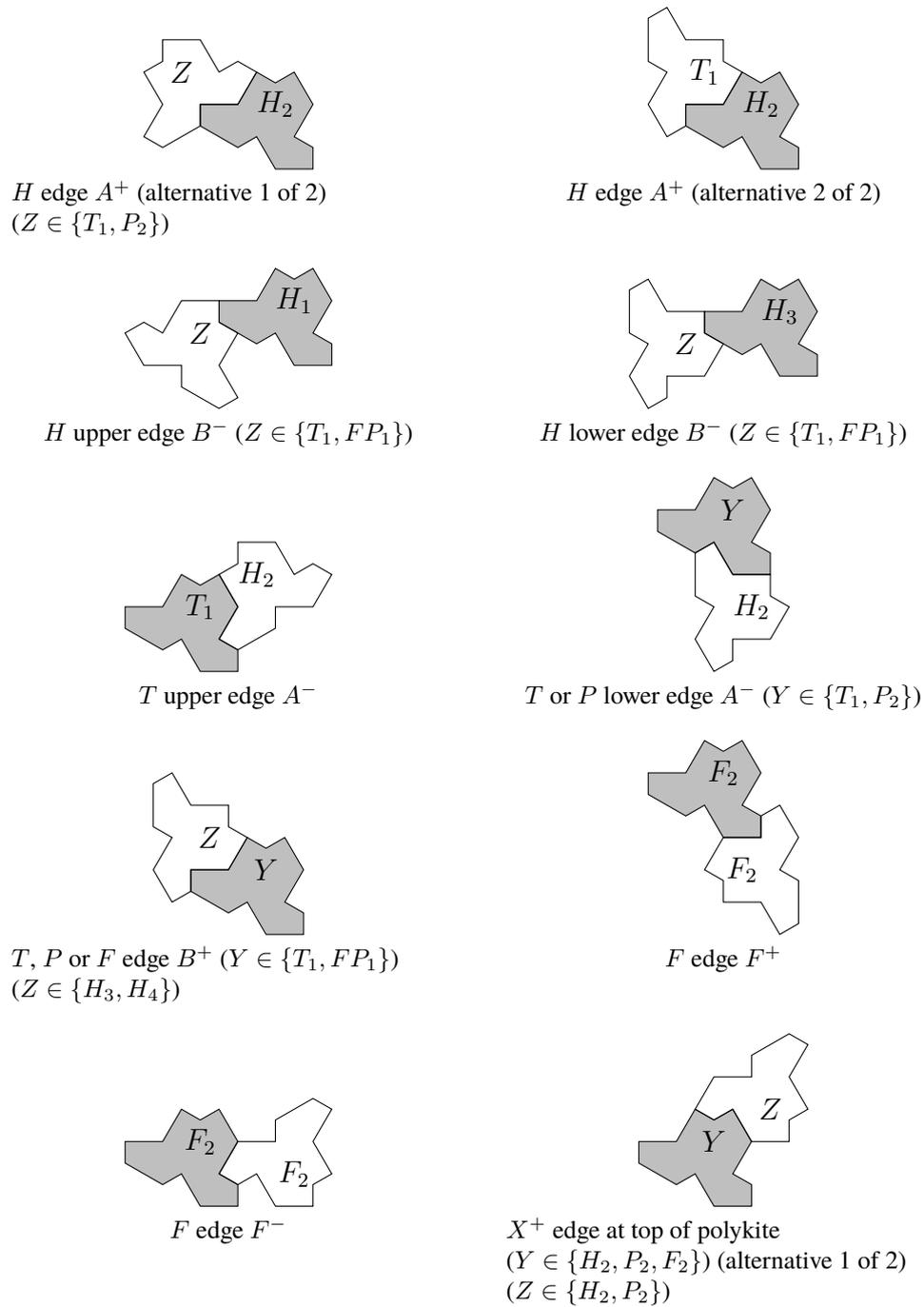
\begin{figure}[htp!]
\ContinuedFloat
\captionsetup{margin=0pt,justification=raggedright}%
\begin{center}
\subfloat[$X^+$ edge at top of polykite ($Y \in \{H_2, P_2, F_2\}$) (alternative 2 of 2) ($Z \in \{H_3, H_4, FP_1, F_2\}$)]{%
\begin{minipage}[b]{6cm}
\begin{center}
\begin{tikzpicture}[x=3mm,y=3mm]
  \ftileA{0}{0}{0}{$Y$};
  \tileA{0}{0}{1}{$Z$};
\end{tikzpicture}%
\end{center}%
\end{minipage}%
} \qquad \subfloat[$X^+$ edge at right of polykite ($Y \in \{H_3, H_4, FP_1\}$) (alternative 1 of 2) ($Z \in \{H_2, P_2\}$)]{%
\begin{minipage}[b]{6cm}
\begin{center}
\begin{tikzpicture}[x=3mm,y=3mm]
  \ftileA{0}{0}{0}{$Y$};
  \tileA{120}{2}{-2}{$Z$};
\end{tikzpicture}%
\end{center}%
\end{minipage}%
} \\ \subfloat[$X^+$ edge at right of polykite ($Y \in \{H_3, H_4, FP_1\}$) (alternative 2 of 2) ($Z \in \{H_3, H_4, FP_1, F_2\}$)]{%
\begin{minipage}[b]{6cm}
\begin{center}
\begin{tikzpicture}[x=3mm,y=3mm]
  \ftileA{0}{0}{0}{$Y$};
  \tileA{240}{2}{-1}{$Z$};
\end{tikzpicture}%
\end{center}%
\end{minipage}%
} \qquad \subfloat[$X^-$ edge at right of polykite ($Y \in \{H_2, P_2\}$) (alternative 1 of 2) ($Z \in \{H_2, F_2, P_2\}$)]{%
\begin{minipage}[b]{6cm}
\begin{center}
\begin{tikzpicture}[x=3mm,y=3mm]
  \ftileA{0}{0}{0}{$Y$};
  \tileA{120}{2}{-1}{$Z$};
\end{tikzpicture}%
\end{center}%
\end{minipage}%
} \\ \subfloat[$X^-$ edge at right of polykite ($Y \in \{H_2, P_2\}$) (alternative 2 of 2) ($Z \in \{H_3, H_4, FP_1\}$)]{%
\begin{minipage}[b]{6cm}
\begin{center}
\begin{tikzpicture}[x=3mm,y=3mm]
  \ftileA{0}{0}{0}{$Y$};
  \tileA{240}{2}{0}{$Z$};
\end{tikzpicture}%
\end{center}%
\end{minipage}%
} \qquad \subfloat[$X^-$ edge at bottom of polykite ($Y \in \{H_3, H_4, FP_1, F_2\}$) (alternative 1 of 2) ($Z \in \{H_2, F_2, P_2\}$)]{%
\begin{minipage}[b]{6cm}
\begin{center}
\begin{tikzpicture}[x=3mm,y=3mm]
  \ftileA{0}{0}{0}{$Y$};
  \tileA{0}{0}{-1}{$Z$};
\end{tikzpicture}%
\end{center}%
\end{minipage}%
} \\ \subfloat[$X^-$ edge at bottom of polykite ($Y \in \{H_3, H_4, FP_1, F_2\}$) (alternative 2 of 2) ($Z \in \{H_3, H_4, FP_1\}$)]{%
\begin{minipage}[b]{6cm}
\begin{center}
\begin{tikzpicture}[x=3mm,y=3mm]
  \ftileA{0}{0}{0}{$Y$};
  \tileA{120}{1}{-2}{$Z$};
\end{tikzpicture}%
\end{center}%
\end{minipage}%
} \qquad \subfloat[$L$ edge at right of polykite (alternative 1 of 2)]{%
\begin{minipage}[b]{6cm}
\begin{center}
\begin{tikzpicture}[x=3mm,y=3mm]
  \ftileA{0}{0}{0}{$P_2$};
  \tileA{180}{3}{-2}{$P_2$};
\end{tikzpicture}%
\end{center}%
\end{minipage}%
}%
\end{center}
\caption{Between-cluster matching checks (part 2)}
\label{fig:between:2}
\end{figure}
\begin{figure}[htp!]
\ContinuedFloat
\captionsetup{margin=0pt,justification=raggedright}%
\begin{center}
\subfloat[$L$ edge at right of polykite (alternative 2 of 2) ($Z \in \{FP_1, F_2\}$)]{%
\begin{minipage}[b]{6cm}
\begin{center}
\begin{tikzpicture}[x=3mm,y=3mm]
  \ftileA{0}{0}{0}{$P_2$};
  \tileA{300}{2}{-1}{$Z$};
\end{tikzpicture}%
\end{center}%
\end{minipage}%
} \qquad \subfloat[$L$ edge at bottom of polykite ($Y \in \{FP_1, F_2\}$) (alternative 1 of 2)]{%
\begin{minipage}[b]{6cm}
\begin{center}
\begin{tikzpicture}[x=3mm,y=3mm]
  \ftileA{0}{0}{0}{$Y$};
  \tileA{60}{-1}{-1}{$P_2$};
\end{tikzpicture}%
\end{center}%
\end{minipage}%
} \\ \subfloat[$L$ edge at bottom of polykite ($Y \in \{FP_1, F_2\}$) (alternative 2 of 2) ($Z \in \{FP_1, F_2\}$)]{%
\begin{minipage}[b]{6cm}
\begin{center}
\begin{tikzpicture}[x=3mm,y=3mm]
  \ftileA{0}{0}{0}{$Y$};
  \tileA{180}{0}{-1}{$Z$};
\end{tikzpicture}%
\end{center}%
\end{minipage}%
}%
\end{center}
\caption{Between-cluster matching checks (part 3)}
\label{fig:between:3}
\end{figure}
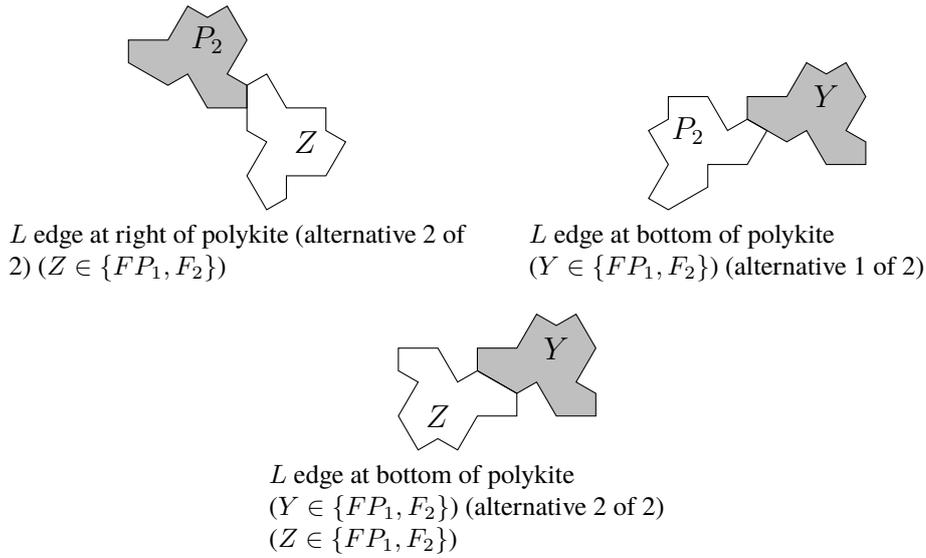

\FloatBarrier

\section{A four-tile substitution system}
\label{sec:subst}

Consider the four metatiles, with matching rules as in
Figure~\ref{fig:tileaclusters}, which are depicted in this section in
the form shown in Figure~\ref{fig:subst}.  Edges $A$ are
\textcolor{\colA}{\textbf{red}}, $B$ are
\textcolor{\colB}{\textbf{blue}}, $X$ are
\textcolor{\colX}{\textbf{green}}, $F$ are
\textcolor{\colF}{\textbf{pink}}, and $L$ are
\textcolor{\colL}{\textbf{grey}}.  Edges are marked with small
geometrical decorations to indicate the signs (outward on the
${}^+$~side, inward on the ${}^-$~side): equilateral triangles
for~$A$, semicircles for~$B$, orthogonal line segments for~$X$, short
slanted line segments for~$F$.  Note that the $A$ and $B$ on $H$ are
the opposite signs to those on $T$, $P$, and $F$.  Also note that the
tiles in this substitution system may not be reflected, only rotated.

\begin{figure}[htp!]
\begin{center}
\begin{tikzpicture}[x=5mm,y=5mm,ultra thick]
  \Ttile{0}{0}{0}{$T$};
  \Htile{0}{4}{-2}{$H$};
  \Ptile{0}{10}{-5}{$P$};
  \Ftile{0}{17}{-8.5}{$F$};
\end{tikzpicture}
\end{center}
\caption{Metatiles $T$, $H$, $P$, and $F$}
\label{fig:subst}
\end{figure}
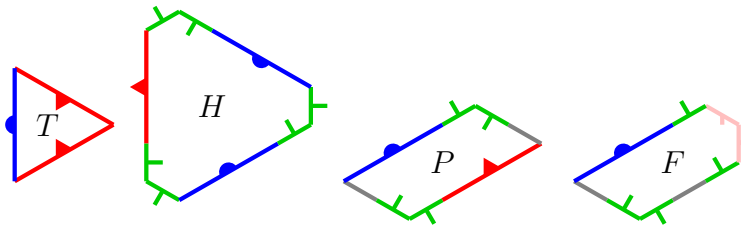

Later in the argument it is convenient to bisect some tiles~$P$
and~$F$, as shown in Figure~\ref{fig:substbisect}.  We refer to the
edges resulting from the bisection of~$P$ as $P^+$ (in the sub-tile
that has an edge~$B^+$) and~$P^-$, coloured
\textcolor{\colP}{\textbf{yellow}} and decorated with a rectangle, and
to the edges resulting from the bisection of $F$ as $G^+$ (in the
sub-tile that has an edge~$B^+$) and $G^-$, coloured
\textcolor{\colG}{\textbf{violet}} and decorated with an obtuse
triangle.  We also refer to the halves with a $B^+$~edge as the upper
halves, and the other halves as the lower halves.  We will show the
following:

\begin{theorem}
\label{thm:subst}
In any tiling by the four metatiles, after bisecting~$P$
and~$F$ metatiles as described above, the metatiles fit together to form larger,
combinatorially equivalent supertiles, thereby forming a substitution
system.  The tiling by the supertiles has the same symmetries as the
tiling by the metatiles.
\end{theorem}

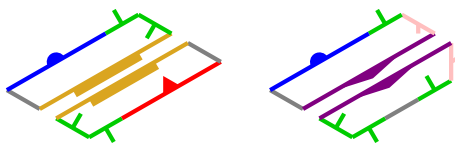
\begin{figure}[htp!]
\begin{center}
\begin{tikzpicture}[x=5mm,y=5mm,ultra thick]
  \Pplus{0}{0}{0}{};
  \Pminus{0}{1.5}{-1.5}{};
  \Fplus{0}{8}{-4}{};
  \Fminus{0}{9.5}{-5.5}{};
\end{tikzpicture}
\end{center}
\caption{Bisection of tiles $P$ and $F$}
\label{fig:substbisect}
\end{figure}

The bisection of tiles is not strictly necessary, in that the
bisecting lines can be arbitrary curves---and, in particular, can go
entirely along one side or other of the $F$ or $P$ tiles (keeping the
same end points), effectively allocating an entire tile to one of
two neighbouring supertiles.  However, the bisected tiles are convenient
for proving that the supertiles obey matching rules equivalent to those
of the original tiles.  In particular, bisection causes adjacencies
between supertiles to be more clearly encoded in 
the boundaries of the supertiles themselves, without also relying on
information about forced tiles that are not part of the supertiles.
In some situations it may be more useful to assign whole tiles to supertiles
at every level of substitution, with no bisection.  
For example, these whole tiles may be more convenient for analyzing sizes
or growth rates of patches in the inflation process.
If needed, we can
define a symmetry-preserving bijection between the supertiles
shown here and any alternative choice of supertiles that avoids bisection.

Sections~\ref{sec:subst:t} and~\ref{sec:subst:nott} present
a branching network of cases in diagrammatic form, building up to patches 
that can be found in tilings by metatiles.
The diagrams should be interpreted as follows.
There are some unnumbered tiles that define the case being considered,
then some numbered tiles that are forced in the sequence given by
their numbers.  If it is then necessary to split into multiple next steps,
the position at which multiple choices of tile must be considered is
marked on the diagram with a filled circle, and there are then
separate diagrams for each choice (in which the previous forced tiles
are now unnumbered, but newly forced tiles are numbered).

The configuration of two~$P$ metatiles shown in Figure~\ref{fig:PP},
denoted~$PP$, often appears in the case
analysis.  The two adjoining copies of~$P$ in the same orientation
force a contradiction because after adding the two forced~$H$ metatiles, 
nothing fits at the marked point.
Subsequently, when identifying forced tiles, as well as considering a
tile as forced when it is the only one that would fit in a given place
consistent with the matching rules, we also consider a tile as
forced when the only alternative consistent with the matching
rules would be to place a $P$~tile in a way that yields this~$PP$ 
configuration. 

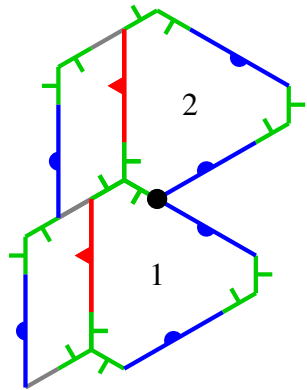
\begin{figure}[htp!]
\begin{center}
\begin{tikzpicture}[x=5mm,y=5mm,ultra thick]
  \Ptile{60}{0}{0}{};
  \Ptile{60}{1}{4}{};
  \Htile{0}{2}{0}{1};
  \Htile{0}{3}{4}{2};
  \markpt{4}{3};
\end{tikzpicture}
\end{center}
\caption{A common impossible configuration, referred to as $PP$}
\label{fig:PP}
\end{figure}

\subsection{Cases involving $T$}
\label{sec:subst:t}

The two $A^-$ edges of~$T$ must be adjacent to the $A^+$ edge of~$H$,
while the $B^+$~edge of~$T$ may be adjacent to either of the
$B^-$~edges of~$H$.  Thus we have two cases for the configuration
around a $T$~tile, which we refer to as $T1$ and $T2$
(Figure~\ref{fig:T1T2}).  As explained in the captions to this and
subsequent figures, a sequence of deductions shows that any~$T$ in a
tiling must occur in case~$T1PF$ (Figure~\ref{fig:T1PF}).

\begin{figure}[htp!]
\begin{center}
\subfloat[Case $T1$]{%
  \begin{tikzpicture}[x=5mm,y=5mm,ultra thick]
  \Ttile{0}{0}{0}{};
  \Htile{-60}{-1}{0}{};
  \Htile{60}{4}{-1}{};
  \Htile{60}{-1}{0}{};
  \markpt{-1}{0};
  \markpt{4}{-1};
  \markpt{0}{4};
\end{tikzpicture}%
} \qquad \subfloat[Case $T2$]{%
\begin{tikzpicture}[x=5mm,y=5mm,ultra thick]
  \Ttile{0}{0}{0}{};
  \Htile{-60}{-1}{0}{};
  \Htile{60}{4}{-1}{};
  \Htile{-60}{-5}{5}{};
  \markpt{-1}{0};
  \markpt{4}{-1};
  \markpt{0}{4};
\end{tikzpicture}%
}%
\end{center}
\caption{Cases $T1$ and $T2$.  Consider the three marked places in
  each of $T1$ and~$T2$.  These can be filled with either $P$ or~$H$.
  On a side of the figure where there are two $B^-$~edges, the marked
  place cannot be filled with~$H$, because that would result in a
  $60^\circ$~angle between two $B^-$~edges, which cannot be filled.
  So both those sides must have $P$~in the marked place, while the
  third side may have $H$ (oriented to avoid such a $60^\circ$~angle
  between two $B^-$~edges; subsequently, when the same situation
  arises, we just consider the orientation of the~$H$ to be forced
  without further comment) or~$P$.  This results in four cases, which
  we call $T1P$ (Figure~\ref{fig:T1P}), $T2P$ (Figure~\ref{fig:T2P}),
  $T1H$ (Figure~\ref{fig:T1H}) and $T2H$ (Figure~\ref{fig:T2H}), and
  we proceed to draw further forced tiles in each of those cases.}
\label{fig:T1T2}
\end{figure}
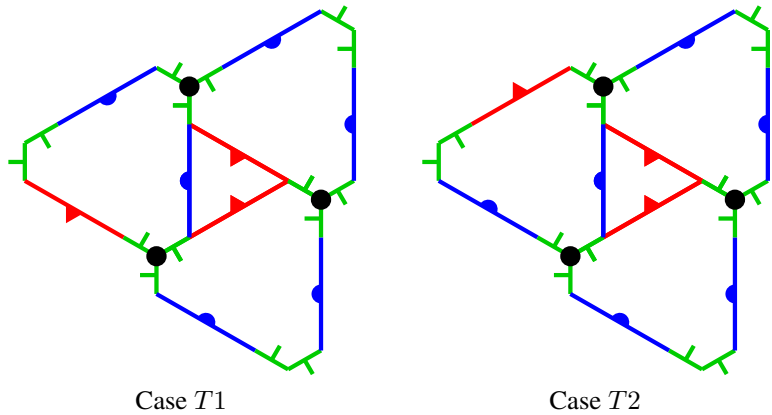

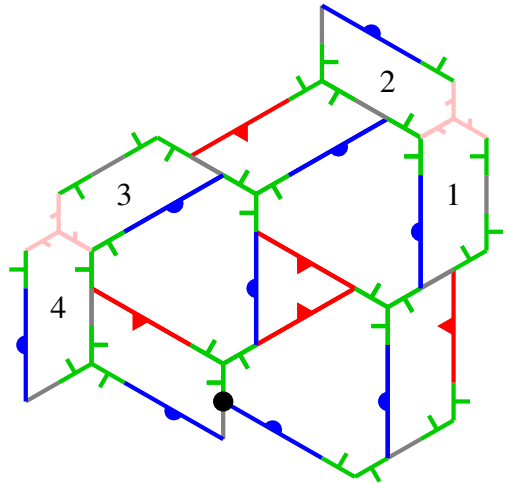
\begin{figure}[htp!]
\begin{center}
\begin{tikzpicture}[x=5mm,y=5mm,ultra thick]
  \Ttile{0}{0}{0}{};
  \Htile{-60}{-1}{0}{};
  \Htile{60}{4}{-1}{};
  \Htile{60}{-1}{0}{};
  \Ptile{60}{4}{-5}{};
  \Ptile{180}{4}{4}{};
  \Ptile{120}{-1}{-2}{};
  \Ftile{60}{5}{-1}{1};
  \Ftile{-60}{2}{8}{2};
  \Ftile{180}{-1}{5}{3};
  \Ftile{60}{-7}{2}{4};
  \markpt{-1}{-1};
\end{tikzpicture}
\end{center}
\caption{Case $T1P$.  The marked place can be filled with $F$ or~$P$,
  resulting in cases we call $T1PF$ (Figure~\ref{fig:T1PF}) and $T1PP$
  (Figure~\ref{fig:T1PP}).}
\label{fig:T1P}
\end{figure}

\begin{figure}[htp!]
\begin{center}
\begin{tikzpicture}[x=5mm,y=5mm,ultra thick]
  \Ttile{0}{0}{0}{};
  \Htile{-60}{-1}{0}{};
  \Htile{60}{4}{-1}{};
  \Htile{-60}{-5}{5}{};
  \Ptile{60}{4}{-5}{};
  \Ptile{-60}{-5}{4}{};
  \Ptile{180}{4}{4}{};
  \Ptile{0}{-7}{7}{1};
  \Ftile{120}{-5}{3}{2};
  \Htile{-120}{-5}{2}{3};
  \Htile{120}{-5}{2}{4};
  \Ftile{-60}{-1}{-1}{5};
  \Ptile{0}{-5}{-3}{6};
  \Ttile{180}{-6}{2}{7};
  \Htile{-120}{-10}{3}{8};
  \Ptile{0}{-10}{-2}{9};
  \markpt{-4}{-4};
\end{tikzpicture}
\end{center}
\caption{Case $T2P$, eliminated because $PP$ occurs at the marked point.}
\label{fig:T2P}
\end{figure}

\begin{figure}[htp!]
\begin{center}
\begin{tikzpicture}[x=5mm,y=5mm,ultra thick]
  \Ttile{0}{0}{0}{};
  \Htile{-60}{-1}{0}{};
  \Htile{60}{4}{-1}{};
  \Htile{60}{-1}{0}{};
  \Ptile{60}{4}{-5}{};
  \Ptile{180}{4}{4}{};
  \Htile{180}{-1}{0}{};
  \Ttile{-120}{-1}{-1}{1};
  \Htile{-60}{-2}{-4}{2};
  \Ptile{60}{3}{-9}{3};
  \markpt{5}{-5};
\end{tikzpicture}
\end{center}
\caption{Case $T1H$, eliminated because $PP$ occurs at the marked point.}
\label{fig:T1H}
\end{figure}

\begin{figure}[htp!]
\begin{center}
\begin{tikzpicture}[x=5mm,y=5mm,ultra thick]
  \Ttile{0}{0}{0}{};
  \Htile{-60}{-1}{0}{};
  \Htile{60}{4}{-1}{};
  \Htile{-60}{-5}{5}{};
  \Ptile{60}{4}{-5}{};
  \Ptile{-60}{-5}{4}{};
  \Htile{180}{1}{8}{};
  \Ttile{120}{4}{4}{1};
  \Htile{60}{5}{3}{2};
  \Ptile{60}{5}{-1}{3};
  \markpt{6}{-1};
\end{tikzpicture}
\end{center}
\caption{Case $T2H$, eliminated because $PP$ occurs at the marked point.}
\label{fig:T2H}
\end{figure}

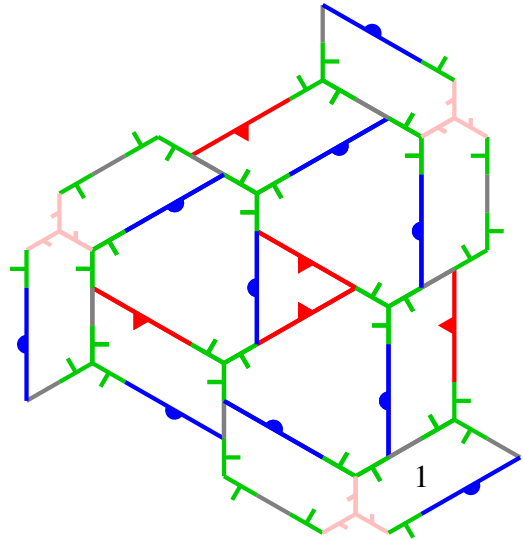
\begin{figure}[htp!]
\begin{center}
\begin{tikzpicture}[x=5mm,y=5mm,ultra thick]
  \Ttile{0}{0}{0}{};
  \Htile{-60}{-1}{0}{};
  \Htile{60}{4}{-1}{};
  \Htile{60}{-1}{0}{};
  \Ptile{60}{4}{-5}{};
  \Ptile{180}{4}{4}{};
  \Ptile{120}{-1}{-2}{};
  \Ftile{60}{5}{-1}{};
  \Ftile{-60}{2}{8}{};
  \Ftile{180}{-1}{5}{};
  \Ftile{60}{-7}{2}{};
  \Ftile{-60}{-1}{-1}{};
  \Ftile{180}{8}{-7}{1};
\end{tikzpicture}
\end{center}
\caption{Case $T1PF$.  Any $T$ in a tiling must occur in this case.
  Bisecting the $P$ and~$F$ tiles in that case produces the
  configuration of Figure~\ref{fig:Hsuper}, which we call~$H'$ and
  which combinatorially acts like~$H$ (with the edge segments
  indicated marked for matching rules) in a tiling along with
  configurations $T'$, $P'$, and~$F'$.}
\label{fig:T1PF}
\end{figure}

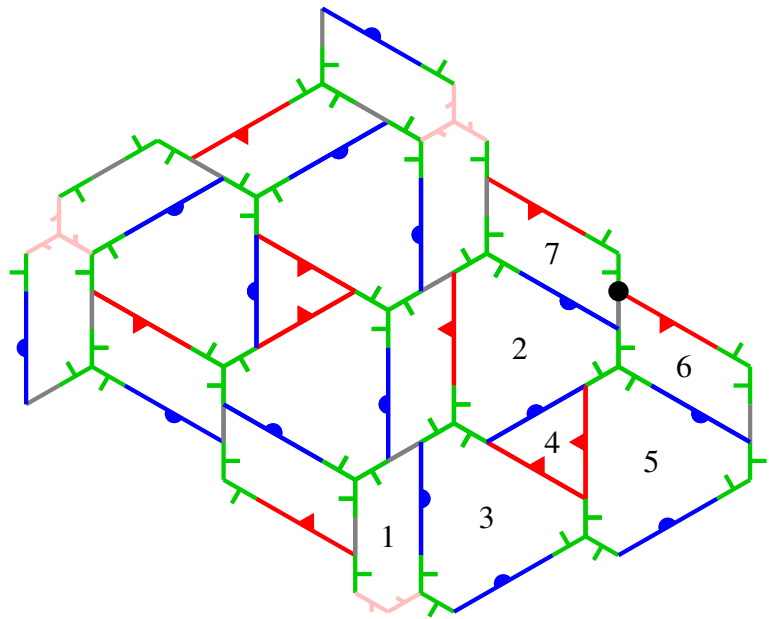
\begin{figure}[htp!]
\begin{center}
\begin{tikzpicture}[x=5mm,y=5mm,ultra thick]
  \Ttile{0}{0}{0}{};
  \Htile{-60}{-1}{0}{};
  \Htile{60}{4}{-1}{};
  \Htile{60}{-1}{0}{};
  \Ptile{60}{4}{-5}{};
  \Ptile{180}{4}{4}{};
  \Ptile{120}{-1}{-2}{};
  \Ftile{60}{5}{-1}{};
  \Ftile{-60}{2}{8}{};
  \Ftile{180}{-1}{5}{};
  \Ftile{60}{-7}{2}{};
  \Ptile{-60}{-1}{-1}{};
  \Ftile{-120}{5}{-5}{1};
  \Htile{0}{6}{-5}{2};
  \Htile{-120}{6}{-5}{3};
  \Ttile{-60}{7}{-6}{4};
  \Htile{0}{10}{-10}{5};
  \Ptile{120}{15}{-10}{6};
  \Ptile{120}{11}{-5}{7};
  \markpt{11}{-4};
\end{tikzpicture}
\end{center}
\caption{Case $T1PP$, eliminated because $PP$ occurs at the marked point.}
\label{fig:T1PP}
\end{figure}

\FloatBarrier

\subsection{Cases with $H$ not adjacent to $T$}
\label{sec:subst:nott}

Any $H$ not adjacent to a $T$~tile must have a $P$~tile adjacent to
its $A^+$ edge, while the $B^-$ edges may each be adjacent to $P$
or~$F$.  This results in four cases, which we call $HPP$
(Figure~\ref{fig:HPP}), $HPF$ (Figure~\ref{fig:HPF}), $HFP$
(Figure~\ref{fig:HFP}), and $HFF$ (Figure~\ref{fig:HFF}), and we
proceed to draw further forced tiles in each of those cases, with
consequences explained in the captions to those figures.

\begin{figure}[htp!]
\begin{center}
\begin{tikzpicture}[x=5mm,y=5mm,ultra thick]
  \Htile{0}{0}{0}{};
  \Ptile{60}{-2}{0}{};
  \Ptile{120}{5}{0}{};
  \Ptile{0}{1}{-1}{};
  \Ftile{60}{-1}{4}{1};
  \Ftile{-60}{5}{1}{2};
  \Ftile{180}{2}{-2}{3};
\end{tikzpicture}
\end{center}
\caption{Case $HPP$.  Bisecting the $P$ tiles and removing the forced
  $F$~tiles produces the configuration of Figure~\ref{fig:Tsuper},
  which we call~$T'$ and which combinatorially acts like~$T$ (with the
  edge segments indicated marked for matching rules) in a tiling
  with the other supertiles.  Although the forced $F$~tiles are not
  included in~$T'$, the fact that they are forced will be used in the
  proof that the supertiles must follow the matching rules where
  they are adjacent to each other.}
\label{fig:HPP}
\end{figure}
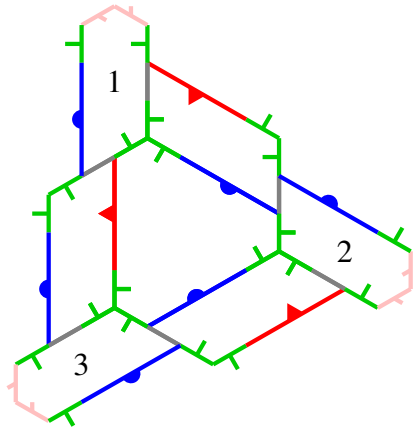

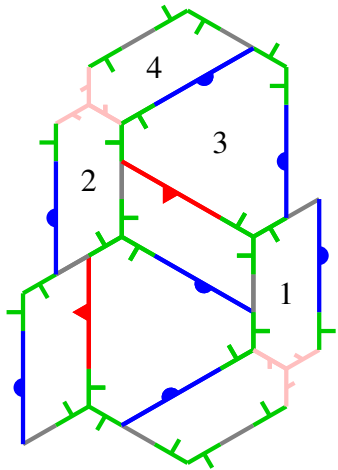
\begin{figure}[htp!]
\begin{center}
\begin{tikzpicture}[x=5mm,y=5mm,ultra thick]
  \Htile{0}{0}{0}{};
  \Ptile{60}{-2}{0}{};
  \Ptile{120}{5}{0}{};
  \Ftile{0}{1}{-1}{};
  \Ftile{-120}{7}{2}{1};
  \Ftile{60}{-1}{4}{2};
  \Htile{60}{5}{2}{3};
  \Ftile{180}{5}{7}{4};
\end{tikzpicture}
\end{center}
\caption{Case $HPF$.  The newly added $H$ cannot be adjacent to
  a~$T$.  It must therefore itself be in case $HFP$ (\fig{fig:HFP})
  or~$HFF$ (\fig{fig:HFF}).
  Because the tiles of the original~$HPF$ cluster are themselves
  forced in those two cases, they handle all patches that can arise here.}
\label{fig:HPF}
\end{figure}

\begin{figure}[htp!]
\begin{center}
\begin{tikzpicture}[x=5mm,y=5mm,ultra thick]
  \Htile{0}{0}{0}{};
  \Ptile{60}{-2}{0}{};
  \Ftile{120}{5}{0}{};
  \Ptile{0}{1}{-1}{};
  \Ftile{0}{-4}{6}{1};
  \Ftile{180}{2}{-2}{2};
  \Ftile{-60}{-7}{4}{3};
  \Htile{-60}{-7}{5}{4};
  \Ptile{0}{-9}{7}{5};
\end{tikzpicture}
\end{center}
\caption{Case $HFP$.  Bisecting all $P$ and~$F$ tiles produces the
  configuration of Figure~\ref{fig:Psuper}, which we call $P'$ and
  which combinatorially acts like~$P$ (with the edge segments
  indicated marked for matching rules) in a tiling with the other
  supertiles.}
\label{fig:HFP}
\end{figure}

\begin{figure}[htp!]
\begin{center}
\begin{tikzpicture}[x=5mm,y=5mm,ultra thick]
  \Htile{0}{0}{0}{};
  \Ptile{60}{-2}{0}{};
  \Ftile{120}{5}{0}{};
  \Ftile{0}{1}{-1}{};
  \Ftile{-120}{7}{2}{1};
  \Ftile{0}{-4}{6}{2};
  \Ftile{180}{2}{-2}{3};
  \Ftile{-60}{-7}{4}{4};
  \Htile{-60}{-7}{5}{5};
  \Ptile{0}{-9}{7}{6};
\end{tikzpicture}
\end{center}
\caption{Case $HFF$.  Bisecting all $P$ and~$F$ tiles produces the
  configuration of Figure~\ref{fig:Fsuper}, which we call $F'$ and
  which combinatorially acts like~$F$ (with the edge segments
  indicated marked for matching rules) in a tiling with the other
  supertiles.}
\label{fig:HFF}
\end{figure}

\FloatBarrier

\subsection{The supertiles}

\begin{figure}[htp!]
\begin{center}
\begin{tikzpicture}[x=5mm,y=5mm,ultra thick]
  \Htile{0}{0}{0}{};
  \Pminus{60}{-1}{0}{};
  \Pplus{120}{5}{0}{};
  \Pplus{0}{1}{-1}{};
  \markpt{1}{-1};
  \markpt{5}{0};
  \markpt{0}{4};
  \vctxt{5}{-3}{$A^-_2$};
  \vctxt{4}{3}{$A^-_2$};
  \vctxt{-2}{3}{$B^+_2$};
  \Ttile{0}{10}{-5}{$T$};
\end{tikzpicture}
\end{center}
\caption{Supertile $T'$, alongside corresponding $T$}
\label{fig:Tsuper}
\end{figure}

\begin{figure}[htp!]
\begin{center}
\begin{tikzpicture}[x=5mm,y=5mm,ultra thick]
  \Ttile{0}{0}{0}{};
  \Htile{-60}{-1}{0}{};
  \Htile{60}{4}{-1}{};
  \Htile{60}{-1}{0}{};
  \Pplus{60}{4}{-5}{};
  \Pplus{180}{4}{4}{};
  \Pminus{120}{-1}{-1}{};
  \Fplus{60}{5}{-1}{};
  \Fminus{-60}{2}{7}{};
  \Fplus{180}{-1}{5}{};
  \Fminus{60}{-6}{2}{};
  \Fplus{-60}{-1}{-1}{};
  \Fminus{180}{7}{-6}{};
  \markpt{-6}{2};
  \markpt{-1}{-2};
  \markpt{3}{-6};
  \markpt{7}{-6};
  \markpt{6}{-1};
  \markpt{6}{3};
  \markpt{2}{7};
  \markpt{-2}{6};
  \markpt{-6}{6};
  \vctxt{-4}{0}{$A^+_2$};
  \vctxt{1}{-5}{$X^-_2$};
  \vctxt{6}{-8}{$X^+_2$};
  \vctxt{6}{-3}{$B^-_2$};
  \vctxt{7}{1}{$X^-_2$};
  \vctxt{4}{6}{$X^+_2$};
  \vctxt{0}{7}{$B^-_2$};
  \vctxt{-5}{7}{$X^-_2$};
  \vctxt{-7}{4}{$X^+_2$};
  \Htile{60}{14}{-7}{$H$};
\end{tikzpicture}
\end{center}
\caption{Supertile $H'$, alongside corresponding $H$}
\label{fig:Hsuper}
\end{figure}

\begin{figure}[htp!]
\begin{center}
\begin{tikzpicture}[x=5mm,y=5mm,ultra thick]
  \Htile{0}{0}{0}{};
  \Ptile{60}{-2}{0}{};
  \Fplus{120}{5}{0}{};
  \Pplus{0}{1}{-1}{};
  \Fminus{0}{-3}{5}{};
  \Fminus{180}{1}{-1}{};
  \Fplus{-60}{-7}{4}{};
  \Htile{-60}{-7}{5}{};
  \Pminus{0}{-8}{6}{};
  \markpt{-7}{4};
  \markpt{-7}{3};
  \markpt{-3}{-1};
  \markpt{1}{-1};
  \markpt{5}{0};
  \markpt{5}{1};
  \markpt{1}{5};
  \markpt{-3}{5};
  \vctxt{-8}{4}{$L_2$};
  \vctxt{-6}{1}{$X^-_2$};
  \vctxt{0}{-2}{$X^+_2$};
  \vctxt{5}{-3}{$A^-_2$};
  \vctxt{6}{0}{$L_2$};
  \vctxt{4}{3}{$X^-_2$};
  \vctxt{-1}{6}{$X^+_2$};
  \vctxt{-6}{7}{$B^+_2$};
  \Ptile{0}{10}{-5}{$P$};
\end{tikzpicture}
\end{center}
\caption{Supertile $P'$, alongside corresponding $P$}
\label{fig:Psuper}
\end{figure}

\begin{figure}[htp!]
\begin{center}
\begin{tikzpicture}[x=5mm,y=5mm,ultra thick]
  \Htile{0}{0}{0}{};
  \Ptile{60}{-2}{0}{};
  \Fplus{120}{5}{0}{};
  \Fplus{0}{1}{-1}{};
  \Fminus{-120}{6}{2}{};
  \Fminus{0}{-3}{5}{};
  \Fminus{180}{1}{-1}{};
  \Fplus{-60}{-7}{4}{};
  \Htile{-60}{-7}{5}{};
  \Pminus{0}{-8}{6}{};
  \markpt{-7}{4};
  \markpt{-7}{3};
  \markpt{-3}{-1};
  \markpt{1}{-1};
  \markpt{2}{-2};
  \markpt{6}{-2};
  \markpt{5}{2};
  \markpt{1}{5};
  \markpt{-3}{5};
  \vctxt{-8}{4}{$L_2$};
  \vctxt{-6}{1}{$X^-_2$};
  \vctxt{-0.5}{-2}{$X^+_2$};
  \vctxt{1.25}{-2}{$L_2$};
  \vctxt{5}{-3}{$X^-_2$};
  \vctxt{7}{0}{$F^+_2$};
  \vctxt{4}{3}{$F^-_2$};
  \vctxt{-1}{6}{$X^+_2$};
  \vctxt{-6}{7}{$B^+_2$};
  \Ftile{0}{10}{-5}{$F$};
\end{tikzpicture}
\end{center}
\caption{Supertile $F'$, alongside corresponding $F$}
\label{fig:Fsuper}
\end{figure}
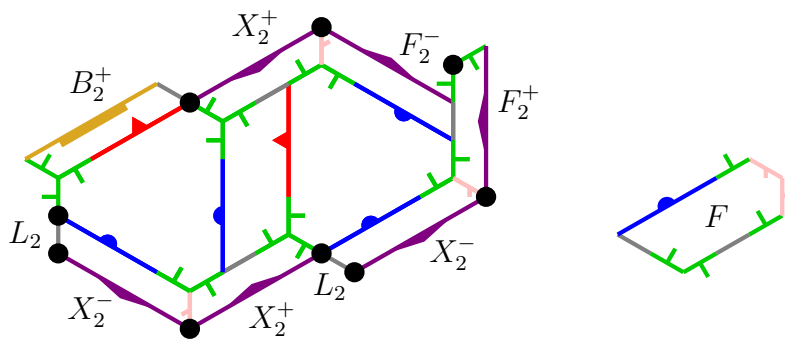

The previous arguments have shown that every $H$ or $T$ tile appears
in a configuration corresponding to the supertiles $T'$, $H'$, $P'$ or
$F'$.  We now provide more detailed rules for allocating each $H$ or
$T$ tile, and the halves of each bisected~$P$ or~$F$ tile, 
to groupings of tiles.  This process ensures that each tile is 
allocated to exactly one grouping, that the groupings all have the form of
one of the supertiles, and that all symmetries of the original tiling are also
symmetries of the tiling by supertiles (this property follows immediately from
the form of the rules, which do not involve any arbitrary choices that
could break symmetry).  As shown in Figures~\ref{fig:Tsuper}--\ref{fig:Fsuper},
we also label the exposed edges of the bisected~$P$ and~$F$ tiles;
the supertiles will be shown to adjoin each other in
accordance with the implied matching rules ($A^+_2$ adjoining
$A^-_2$, $B^+_2$ adjoining $B^-_2$, $X^+_2$ adjoining $X^-_2$, $F^+_2$
adjoining $F^-_2$, and $L_2$ adjoining $L_2$).

We use the following allocation rules to build groupings of metatiles.

\begin{itemize}
  \item Each $T$ tile is allocated to an $H'$~supertile, along with all
    the $H$~tiles adjacent to that $T$.
  \item Each $H$ tile in case $HPP$ is allocated to a $T'$~supertile.
  \item Each $H$ tile in case $HFP$ is allocated to a $P'$~supertile,
    along with the $H$~tile in case~$HPF$ shown in
    Figure~\ref{fig:HFP}.
  \item Each $H$ tile in case~$HFF$ is allocated to an $F'$~supertile,
    along with the $H$~tile in case~$HPF$ shown in
    Figure~\ref{fig:HFF}.
  \item Each $H$ tile in case $HPF$ was allocated to a supertile by
    exactly one of the previous two rules.
  \item Each half of a $P$ tile, and the upper half of each $F$~tile,
    is adjacent to exactly one $H$~tile along its $A^-$ or $B^+$~edge,
    and is allocated to the same supertile as that $H$~tile.  (We
	simplify Figures~\ref{fig:Psuper} and~\ref{fig:Fsuper} by eliding
	the bisection in the central~$P$ tile.)
  \item It remains to allocate the lower halves of $F$~tiles.  Each
    such lower half has an $X^-$~edge between an $L$~edge and an
    $F^+$~edge; it is allocated to the same supertile as the $H$~tile
    adjacent to that $X^-$ edge.  For this allocation rule to be
	well defined, we need to show that this $X^-$ edge is indeed
    adjacent to a $H$~tile.  The only other possibility not violating
    the metatile matching rules would be the configuration shown in 
    Figure~\ref{fig:FF}.  This configuration cannot arise in a tiling
	by metatiles, because no tile can be adjoined at the marked point.
\end{itemize}

The $X^-$ edge referenced in the last allocation rule cannot be
adjacent to any of the exposed~$X^+$~edges of $H$~tiles in supertiles
$T'$, $P'$ or~$F'$ without violating the matching rules.  Thus all
lower halves of $F$~tiles are the ones that appear on the diagrams of
the supertiles, and we have shown that the tiling is partitioned into
the supertiles.

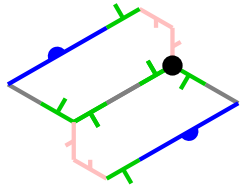
\begin{figure}[htp!]
\begin{center}
\begin{tikzpicture}[x=5mm,y=5mm,ultra thick]
  \Ftile{0}{0}{0}{};
  \Ftile{180}{7}{-4}{};
  \markpt{5}{-2};
\end{tikzpicture}
\end{center}
\caption{Impossible adjacency of two $F$ tiles}
\label{fig:FF}
\end{figure}

We now show that the supertiles must adjoin each other in accordance with
the matching rules indicated.  First, we examine $P^+$~edges
(appearing in~$A^-_2$ and~$B^-_2$) and $P^-$~edges (appearing in
$A^+_2$ and~$B^+_2$).  $B^-_2$ and~$A^+_2$ appear only in~$H'$, where
their $P^+$ and~$P^-$~edges cannot meet without tiles intersecting.
So $B^-_2$ can only join to~$B^+_2$ and $A^+_2$ can only join to~$A^-_2$.

Next we show that the converse holds: $A^-_2$ can only join to~$A^+_2$
and $B^+_2$ can only join to~$B^-_2$.  For a contradiction, suppose
that the $P^+$ and $P^-$ edges in some $A^-_2$ and~$B^+_2$ are joined.
If the~$B^+_2$ comes from a $P'$~supertile, then that $P^-$ edge bisects
tile~5 in case~$HFP$.  Adjacent tiles 5 and~1 in that configuration
both have $B^+$~edges, which must both be adjacent to $H$~tiles. Those
$H$~tiles are adjacent to each other, placing this configuration within 
an~$H'$ supertile, which does not have an $A^-_2$~edge.  The same
argument applies in the case of an $F'$~supertile, considering tiles~6
and~2 in case~$HFF$.  If the~$P^+$ in an~$A^-_2$ from~$P'$ is joined
to a~$B^+_2$, a similar argument applies (considering tile 2 and an
adjacent unnumbered tile in case~$HFP$).  So the only remaining case
would be if both edges come from supertile~$T'$, but that possibility is
inconsistent with the $F$~tiles forced in case~$HPP$.

Keeping in mind that 
$A^+_2$, $A^-_2$, $B^+_2$, and $B^-_2$ must obey the
matching rules for the supertiles, note next that the only $X^+$
and $X^-$ metatile edges on the boundaries of the supertiles that are not
part
of $A^+_2$, $A^-_2$, $B^+_2$ or $B^-_2$ are those forming part of
$F^+_2$ and~$F^-_2$.  Thus it follows that $F^+_2$ and~$F^-_2$ must
also adjoin each other.  The only $G^+$ and~$G^-$ metatile edges still
unaccounted for are those that form $X^-_2$ and $X^+_2$ edges of the
supertiles, meaning that those also match. Finally, the remaining 
$L$~edges form~$L_2$, which must also match.

To show that the supertiles are fully combinatorially equivalent to
the original tiles, one more thing must be checked: that the same
combinations of supertiles fit together at vertices as combinations of
tiles fit together at vertices.  Each supertile has been drawn with a
copy of the corresponding tile alongside it, in a corresponding
orientation.  By inspection, if we take any class of edges of the
metatiles, including both sides of the edge (for example, $A^+$ and
$A^-$), and take any line segment in the corresponding edges of the
supertiles, the (directed) angle between the (directed) edges in the
tile and in the supertile is consistent across all the diagrams.

This consistency of angles between edges of metatiles and of supertiles
means that the angles at vertices of supertiles around a point, each
consecutive pair having matching edges, add up to the same amount as
the corresponding angles for the corresponding metatiles (an angle at a
vertex of a supertile equals the angle at the corresponding vertex of
the corresponding metatile, plus the difference between the
metatile--supertile angles for the two edges, and those differences cancel
when adding up around the point).

The supertiles are therefore combinatorially equivalent to the metatiles,
and so
the above arguments apply inductively to ensure that the composition
of tiles into supertiles may be applied~$n$~times for all~$n$.  Since
the radius of a ball contained in the supertiles goes to infinity
with~$n$ (a fact that does not depend on the geometry used to bisect~$P$
and~$F$ metatiles, but that may be easier to show with alternative 
supertiles that avoid bisection), 
and the tilings by supertiles have all the
symmetries of the original tiling, it follows that the original tiling
cannot have a translation as a symmetry.  Furthermore, the
substitution structure implies that the metatiles tile arbitrarily
large finite regions of the plane, and hence the whole plane.

Because of the symmetry-preserving correspondence between tilings by
metatiles and tilings by hat polykites, 
we have completed a proof of Theorem~\ref{thm:subst_tiling}.

\section{A family of aperiodic monotiles}
\label{sec:family}

In the previous sections, we showed that the hat polykite
is an aperiodic monotile.  This polykite is
formed of eight kites from the $[3.4.6.4]$ Laves tiling.  Likewise
the turtle polykite, formed of ten kites and shown in Figure~\ref{fig:tileb},
is also aperiodic.  We have verified via a computer search that there are no 
other aperiodic $n$-kites for $n \le 24$.

\begin{figure}[htp!]
\begin{center}
\begin{tikzpicture}[x=5mm,y=5mm]
  \tileB{0}{0}{0}{};
\end{tikzpicture}
\end{center}
\caption{An aperiodic $10$-kite called the ``turtle''}
\label{fig:tileb}
\end{figure}
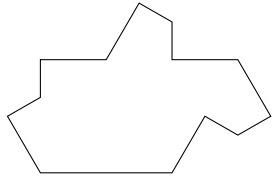

These two aperiodic polykites are two examples of a family of
aperiodic monotiles, all of which have combinatorially equivalent sets
of tilings, and which are determined by the choice of two side
lengths.

The hat polykite has sides of lengths $1$, $2$, and~$\sqrt{3}$; for the
purposes of this section, we consider the side of length~$2$ as two
consecutive sides of length~$1$ with a $180^\circ$~angle between them.
The tile of Figure~\ref{fig:tileb} has the same angles, but with the
side lengths of $1$ and~$\sqrt{3}$ swapped.

Let $a$ and~$b$ be nonnegative reals, not both zero, and if $a\ne 0$
let $r = b / a$.
Define $\mathrm{Tile}(a, b)$ to be the polygon resulting from replacing the sides of
length~$1$ in the hat polykite with sides of length~$a$ (we
refer to the resulting sides as \emph{$1$-sides}) and replacing the
sides of length~$\sqrt{3}$ in the hat polykite with sides of
length~$b$ (we refer to the resulting sides as \emph{$r$-sides}).
Thus the hat is $\mathrm{Tile}(1, \sqrt{3})$ and the turtle
is $\mathrm{Tile}(\sqrt{3}, 1)$.
This process results in a closed curve (because the vectors of the
$1$-sides add up to~$0$, as do those of the $r$-sides) that can
easily be shown to be free of self-intersections.
It is a $13$-gon (or one with a smaller number of sides if $a$ or~$b$ is
zero), but considered as a $14$-gon for the purposes of this section.
$\mathrm{Tile}(a,b)$ has area~$\sqrt{3}(2a^2+\sqrt{3}ab+b^2)$.  

For nonzero~$a$, the value 
of~$r$ determines the tile up to similarity.  In acknowledgment of 
these similarity classes, we write $\mathrm{Tile}(r)$ as a
shorthand for $\mathrm{Tile}(1,r)$.  We will show this tile is
aperiodic for any positive $r \ne 1$.
In fact, $\mathrm{Tile}(1, k\sqrt{3})$ and $\mathrm{Tile}(k\sqrt{3},
1)$ are polykites for all odd positive integers~$k$, implying that this
continuum of aperiodic monotiles contains a countably infinite family of
aperiodic polykites.

\begin{figure}[htp!]
\begin{center}
\begin{tikzpicture}[x=5mm,y=5mm]
  \tiler{1}{0}{0}{0}{0};
  \tiler{1}{-1}{2}{3}{0};
  \tiler{1}{-2}{4}{6}{0};
  \tiler{1}{-3}{6}{9}{0};
  \tilerr{1}{3}{-2}{1}{-2};
  \tilerr{1}{2}{0}{4}{-2};
  \tilerr{1}{1}{2}{7}{-2};
  \tilerr{1}{0}{4}{10}{-2};
  \tiler{1}{4}{-2}{6}{-6};
  \tiler{1}{3}{0}{9}{-6};
  \tiler{1}{2}{2}{12}{-6};
  \tiler{1}{1}{4}{15}{-6};
  \tilerr{1}{7}{-4}{7}{-8};
  \tilerr{1}{6}{-2}{10}{-8};
  \tilerr{1}{5}{0}{13}{-8};
  \tilerr{1}{4}{2}{16}{-8};
\end{tikzpicture}
\end{center}
\caption{Periodic tiling by $\mathrm{Tile}(1, 1)$}
\label{fig:tile1periodic}
\end{figure}

We define a notion of combinatorial equivalence between tilings of
these tiles, for two positive values of~$r$, as follows: two tilings
are combinatorially equivalent if there exists a bijection between
their tiles, and a bijection between the maximal line segments in the
unions of the boundaries of the tiles, such that corresponding tiles
and line segments in the two tilings are in the same orientation,
corresponding tiles adjoin corresponding line segments, on the same
side of those line segments, in the two tilings, and corresponding
tiles on the corresponding sides of corresponding line segments appear
in the same order along those segments.  All the interior angles of
the tile are at least~$90^\circ$, and no two $90^\circ$~angles appear
consecutively, so any maximal line segment has at most two sides of
tiles on each side of the line segment (and in particular is finite).

We now prove the following result:

\begin{theorem}
\label{thm:tilercomb}
Suppose $r \ne 1$ and $r'\ne 1$ are positive.  Then there is a
bijection between combinatorially equivalent tilings for $\mathrm{Tile}(r)$ and
$\mathrm{Tile}(r')$, given by changing the lengths of all $r$-sides from $r$
to~$r'$, while preserving angles, orientations, and adjacencies to
maximal line segments.
\end{theorem}

Suppose first that $r$ is irrational. 
If a maximal line segment in the union of the
boundaries of the tiles has $p$~$1$-sides and $q$~$r$-sides on one
side of the line segment, it also has $p$~$1$-sides and~$q$~$r$-sides
on the other side of the line segment. Because a maximal line segment
has at most two sides of tiles on each side of the segment, the same
argument also applies for any rational~$r$ except possibly $\frac{1}{2}$, $1$,
and~$2$.

If $r = 2$, there is the additional possibility that two $1$-sides
align with one $r$-side.  When there are two consecutive $1$-sides on
one side of a line, with $90^\circ$~corners of the two tiles between
those two sides (or the $180^\circ$~corner of a single tile), the
other ends of those sides have corners with angles $120^\circ$
or~$240^\circ$.  But for every $r$-side, one corner has angle
$90^\circ$ or~$270^\circ$, and the angles of the tile do not permit
$120^\circ$ or $240^\circ$ at the same vertex of a tiling as
$90^\circ$ or~$270^\circ$.

Similarly, in the case $r = \frac{1}{2}$, the only additional
possibility is that two $r$-sides align with one $1$-side.  The outer
corners of the two $r$-sides have angles $120^\circ$ or~$240^\circ$,
one corner of every $1$-side has angle $90^\circ$, $180^\circ$
or~$270^\circ$, and those cannot appear at the same vertex.

Thus for any positive $r\ne 1$, we have shown that if a maximal line segment 
in the union of the boundaries of the tiles has $p$~$1$-sides and
$q$~$r$-sides on one side of the line segment, it also has~$p$~$1$-sides and $q$~$r$-sides on the other side of the line segment.
We can now construct the required bijection.  Because side vectors
around any tile add up to zero, and
the sides of tiles on both sides of a maximal line segment add up to
the same length, the specified process converts a tiling by $\mathrm{Tile}(r)$ into one
by $\mathrm{Tile}(r')$ that is combinatorially equivalent~\cite[Lemma~1.1]{regprod}.

(This argument relies on the fact that the plane is simply connected.
A tiling by $\mathrm{Tile}(r)$ of a region with a hole that cannot be filled
with tiles might not convert to a tiling by $\mathrm{Tile}(r')$ of a region with
a combinatorially equivalent hole. Indeed, for some vectors defining
the sides of the hole, there might not exist any combinatorially equivalent 
hole if the vectors of the $1$-sides among the sides of the hole do not add up
to~$0$.)

As shown in Lemma~\ref{lemma:tileaalign}, all tilings by
$\mathrm{Tile}(\sqrt{3})$ are aligned to an underlying $[3.4.6.4]$ Laves tiling,
so in fact each maximal line segment is made up only of $1$-sides or
only of $r$-sides.

Finally, $\mathrm{Tile}(1)$---or more generally $\mathrm{Tile}(a, a)$---is
not aperiodic, as shown by the periodic tiling in
Figure~\ref{fig:tile1periodic}.  The polyiamonds $\mathrm{Tile}(a, 0)$ and
$\mathrm{Tile}(0, b)$ are also not aperiodic.  A tiling by $\mathrm{Tile}(r)$ for
positive $r\ne 1$ can still be mapped to a corresponding tiling by
$\mathrm{Tile}(a, a)$, $\mathrm{Tile}(a, 0)$, or $\mathrm{Tile}(0, b)$ following the process
described above, but the map is not a bijection.

\section{Conclusion}
\label{sec:conclusion}

We have exhibited an einstein, the first topological disk that tiles
aperiodically with no additional constraints or matching rules.
The hat polykite is in fact a member of a continuous family of aperiodic
monotiles that admit combinatorially equivalent tilings.
The hat forces
tilings with hierarchical structure, as is the case for many aperiodic
sets of tiles in the plane, but a new method introduced in
\secref{sec:coupling} also suffices to show the lack of periodic
tilings without needing that hierarchical structure, beyond
demonstrating the existence of a tiling.

Our substitution system satisfies the relatively mild conditions needed
to guarantee an uncountable infinity of combinatorially distinct tilings,
all of which are hierarchical~\cite[Section 7.6.2]{Senechal}.  But not
every tiling by hats is necessarily produced purely through substitution.
As with Robinson's aperiodic set of six shapes~\cite{Robinson},
it is conceivable that hats could tile infinite sectors of the plane, which
could then be combined into tilings with infinite ``fault lines'' that
lie on the boundaries of supertiles at all levels.  Future work should
examine the possibility
of tilings with fault lines, as part of characterizing the full space of 
hat tilings.  In particular, it should be determined whether every
finite patch that appears in some hat tiling must appear infinitely
often in all hat tilings, or whether there are patches that only
appear on fault lines and not in the interior of a supertile.

The hat is a $13$-sided non-convex polygon.  A convex polygon cannot
be an aperiodic mono\-tile, and all non-convex quadrilaterals can 
easily be seen to tile periodically.  Therefore, in terms of number of 
sides, the ``simplest'' aperiodic $n$-gon must have $5\le n \le 13$.
Subsequent research could chip away at this range, by finding aperiodic
$n$-gons for $n<13$ or ruling them out for~$n\ge 5$.

Tilings by the hat necessarily include both reflected and unreflected
tiles.  We might therefore ask whether there exists an aperiodic
monotile for which reflections are not needed, either because the tile
has bilateral symmetry or because it covers the plane using only
translations and rotations.\footnote{In subsequent
work~\cite{spectre}, we show that
$\mathrm{Tile}(1,1)$ is such a tile if its boundary is modified to
prevent the use of reflections.}

Finding such a monotile pushes the boundaries of complexity known to
be achievable by the tiling behaviour of a single closed topological
disk.  It does not, however, settle various other unresolved
questions about that complexity.  For example, all of the following
questions remain open.

\begin{itemize}
  \item Are Heesch numbers unbounded?  That is, does there exist, for
  	every positive integer $n$, a topological disk that does not tile the 
	plane and has Heesch number at least $n$?  We conjecture that there is
	no bound on Heesch numbers.

  \item Are isohedral numbers unbounded?  That is, does there exist, for
  	every positive integer $n$, a topological disk that tiles the plane 
	periodically
	but only admits tilings with at least $n$ transitivity classes?  Again,
	we conjecture that no bound exists.  If the requirement of periodicity
	is omitted here, then the hat polykite requires infinitely many 
	transitivity classes in any tiling. 
	Socolar~\cite{Socolar} showed that if the tile is not required to be
    a closed topological disk, then tiles exist with every positive
	isohedral number.

  \item Is it computationally undecidable whether a polygon (or
    indeed a more general single tile in the plane) admits a tiling?
    It would again be reasonable to conjecture yes, which would also
    imply unbounded Heesch numbers.  Greenfeld and
    Tao~\cite{GT1,greenfeld2023undecidability} demonstrated
    undecidability in a more general
    context.  For sets of tiles in the plane, Ollinger~\cite{Ollinger}
    proved  undecidability for sets of five polyominoes.

  \item Is it computationally undecidable whether a polygon (or
    indeed a more general single tile in the plane) admits a periodic
    tiling?  It would again be reasonable to conjecture yes.  Such
    an answer would imply unbounded isohedral numbers.
\end{itemize}

Although we have provided a description of tilings by the
hat polykite and related tiles described here (all such tilings are
given by the substitution system of \secref{sec:subst}, as applied to
the clusters of tiles from \secref{sec:clusters}, subject to the
possibility mentioned above of tilings with fault lines, where each
sector is produced by the substitution system), there are various
informal observations in \secref{sec:discussion} that have not been
fully explored or given precise statements.  Those observations could
provide starting points for possible future investigation of the tiles
described here and their tilings, the metatiles used in classifying
tilings by the hat polykite, and other related substitution tilings.
It is not clear which ideas from this work will be most promising for
future work, so we have generally erred on the side of including 
observations that might be of use, rather than making the paper focus
more narrowly on a single proof of a single main result.

We believe that the approach presented in \secref{sec:coupling}, of coupling 
two separate tilings to show that a third tiling cannot be periodic, is a
new way to prove that a set of tiles is aperiodic.  It would
be worth investigating whether it can be applied in other contexts.
In particular, polykites (and more generally
poly-$[4.6.12]$-tiles, a subset of the shapes known as
\emph{polydrafters}) may be unusually well-suited to this method of
proof, because their edges lie on lines belonging to two regular
triangle tilings.  It might also be applicable to some
poly-$[4.8.8]$-tiles (a subset of the
\emph{polyaboloes}).\footnote{Note that if Lemma~\ref{lemma:align} is
applied to poly-$[4.6.12]$-tiles or poly-$[4.8.8]$-tiles, the
conclusion is weaker than that of Lemma~\ref{lemma:polykitealign} for
polykites, so tilings may need to be considered that are only aligned
in this weaker sense.} 
This style of proof might help explain how small polykites
proved to be aperiodic when polyominoes, polyiamonds and polyhexes up
to high orders yielded no einsteins.  However, as noted in
\secref{sec:family}, searches of polykites have not found other
aperiodic examples outside the family described in this paper.



\section*{Acknowledgements}
\label{sec:acks}

Thanks to Ava Pun for her work on software for computing Heesch numbers
and displaying patches.  Thanks also to Jaap Scherphuis for creating a
number of useful free software tools for exploring tilings.  These tools
all played a crucial role in the explorations that led to the discovery of
our polykite.  Thanks to Marjorie Senechal, Michael Baake, and Pablo Rosell
for suggesting improvements to earlier drafts of this article, and to the
anonymous referees for their many constructive and insightful suggestions.
Thanks to the members of The Tiling List for years of making connections
and engaging in discussion.

\appendix
\section{Aligned and unaligned tilings of polyforms}
\label{sec:align}

The proof that the hat polykite is aperiodic
involves a case analysis for ways of surrounding a copy of that tile,
and that case analysis in turn involves considering possibilities for
how an individual kite in a copy of the hat polykite could fill a
particular kite on an underlying $[3.4.6.4]$ Laves tiling.  By itself
a proof founded on such a case analysis shows the absence of
periodic tilings only when all tiles are \textit{aligned} to the same 
underlying $[3.4.6.4]$ Laves tiling.  This argument leaves open the
possibility that polykites might be able to tile periodically if they
may be translated, rotated and reflected without regard to the underlying grid.

In this appendix we prove that for the purposes of establishing the
aperiodicity of the hat, it suffices to consider only aligned tilings
by polykites.  Specifically, if a polykite admits any periodic tiling,
it must also admit one that is aligned.  In fact, we present a more
general result (Lemma~\ref{lemma:align}) that gives sufficient conditions
under which one may restrict attention to aligned tilings.  Our result
covers a broad class of polyforms that includes polykites, and
a broad family of tiling properties that includes periodicity.  We offer
it because it may be useful in related contexts where combinatorial
arguments help establish tiling properties of polyforms.
We also show that the
hat in particular does not admit any unaligned tilings 
(Lemma~\ref{lemma:tileaalign}).

In principle, the same issue arises for polyominoes and polyiamonds.
However, the only unaligned tilings by congruent squares consist of
offset parallel rows of squares 
(Figure~\ref{fig:slidsquares}), and much the same applies to tilings
by congruent equilateral triangles (Figure~\ref{fig:slidtriangles}),
and so it is clear that no interesting examples of unaligned tilings
by polyominoes or polyiamonds can arise.
However, kites can form nontrivial unaligned tilings such as that of
Figure~\ref{fig:slidhextrikites}, and so there is genuinely something to
be proved here, something less obvious than it is for polyominoes and
polyiamonds.

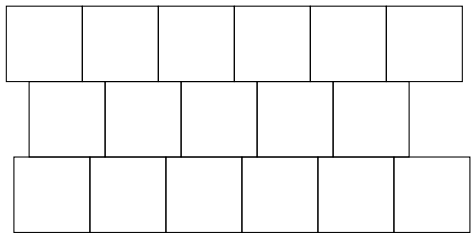
\begin{figure}[htp!]
\begin{center}
\begin{tikzpicture}[x=1cm,y=1cm]
  \drawsquare{0}{0};
  \drawsquare{1}{0};
  \drawsquare{2}{0};
  \drawsquare{3}{0};
  \drawsquare{4}{0};
  \drawsquare{5}{0};
  \drawsquare{0.2}{1};
  \drawsquare{1.2}{1};
  \drawsquare{2.2}{1};
  \drawsquare{3.2}{1};
  \drawsquare{4.2}{1};
  \drawsquare{-0.1}{2};
  \drawsquare{0.9}{2};
  \drawsquare{1.9}{2};
  \drawsquare{2.9}{2};
  \drawsquare{3.9}{2};
  \drawsquare{4.9}{2};
\end{tikzpicture}
\end{center}
\caption{Sliding rows of squares}
\label{fig:slidsquares}
\end{figure}

\begin{figure}[htp!]
\begin{center}
\begin{tikzpicture}[x=1cm,y=1cm]
  \drawtriangle{0}{0}{0};
  \drawtriangle{-60}{0}{1};
  \drawtriangle{0}{0}{1};
  \drawtriangle{-60}{0}{2};
  \drawtriangle{0}{0}{2};
  \drawtriangle{-60}{0}{3};
  \drawtriangle{0}{0}{3};
  \drawtriangle{-60}{0}{4};
  \drawtriangle{0}{0}{4};
  \drawtriangle{0}{1}{-0.7};
  \drawtriangle{-60}{1}{0.3};
  \drawtriangle{0}{1}{0.3};
  \drawtriangle{-60}{1}{1.3};
  \drawtriangle{0}{1}{1.3};
  \drawtriangle{-60}{1}{2.3};
  \drawtriangle{0}{1}{2.3};
  \drawtriangle{-60}{1}{3.3};
  \drawtriangle{0}{1}{3.3};
  \drawtriangle{0}{2}{-1.3};
  \drawtriangle{-60}{2}{-0.3};
  \drawtriangle{0}{2}{-0.3};
  \drawtriangle{-60}{2}{0.7};
  \drawtriangle{0}{2}{0.7};
  \drawtriangle{-60}{2}{1.7};
  \drawtriangle{0}{2}{1.7};
  \drawtriangle{-60}{2}{2.7};
  \drawtriangle{0}{2}{2.7};
\end{tikzpicture}
\end{center}
\caption{Sliding columns of equilateral triangles}
\label{fig:slidtriangles}
\end{figure}
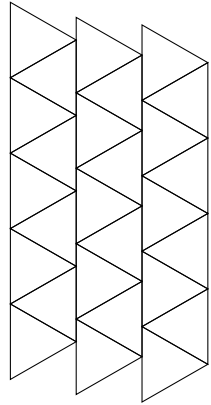

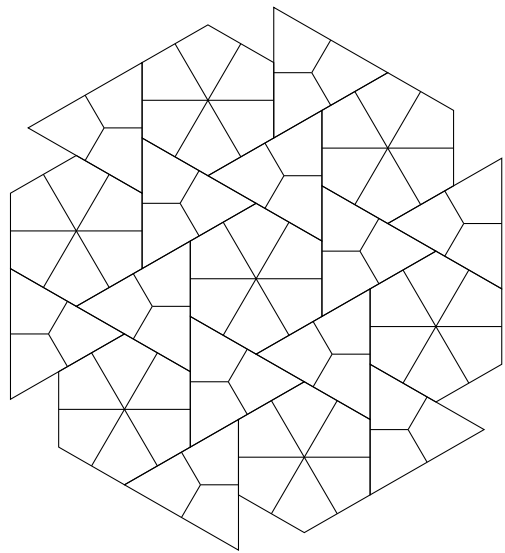
\begin{figure}[htp!]
\begin{center}
\begin{tikzpicture}[x=5mm,y=5mm]
  \hexkites{0}{0}{0}{0};
  \trikites{-120}{0}{0}{0}{0};
  \trikites{-60}{2}{-2}{0}{0};
  \trikites{0}{4}{-2}{0}{0};
  \trikites{60}{4}{0}{0}{0};
  \trikites{120}{2}{2}{0}{0};
  \trikites{180}{0}{2}{0}{0};
  \hexkites{2}{-4}{2}{0};
  \hexkites{4}{-2}{0}{2};
  \hexkites{2}{2}{-2}{2};
  \hexkites{-2}{4}{-2}{0};
  \hexkites{-4}{2}{0}{-2};
  \hexkites{-2}{-2}{2}{-2};
  \trikites{-120}{2}{-4}{2}{0};
  \trikites{-60}{6}{-4}{0}{2};
  \trikites{0}{6}{0}{-2}{2};
  \trikites{60}{2}{4}{-2}{0};
  \trikites{120}{-2}{4}{0}{-2};
  \trikites{180}{-2}{0}{2}{-2};
\end{tikzpicture}
\end{center}
\caption{Unaligned tiling of kites}
\label{fig:slidhextrikites}
\end{figure}

In order to apply the results presented here to other classes of 
polyforms such as polyaboloes,
we state the required conditions on the tiles in fairly general and
technical form.

Let $S$ be a set of real numbers that are linearly
independent over~$\mathbb{Q}$.  Let $\mathcal{P}$ be a finite set of
closed topological disk polygonal tiles, such that all the angles
of corners of tiles in~$\mathcal{P}$ are rational sub-multiples of~$\pi$,
all the lengths of sides of tiles in~$\mathcal{P}$ are integer
multiples of elements of~$S$, and such that, if a polygon
in~$\mathcal{P}$ has two or more collinear sides,
the lengths of those sides are integer multiples of the same element
of~$S$, as are the distances between their endpoints.

We now consider clusters of tiles built using copies of the polygons
in $\mathcal{P}$.  Let $\mathcal{Q}$ be a nonempty set of tiles, each
one congruent
to one of the polygons in $\mathcal{P}$, with disjoint interiors.  
The set $\mathcal{Q}$ may cover the entire plane, or just part of it.
The union of the boundaries of the tiles in $\mathcal{Q}$ decomposes
into a set of maximal line segments, rays, and infinite lines, which we will
refer to generically as \textit{segments}.  These segments are maximal
in the sense that no segment is a subset of a longer segment contained in
the union of the tile boundaries.

Given one such maximal segment~$\ell$, and two tiles $A, B\in\mathcal{Q}$
(which may be identical), we say that $A$ and~$B$ are \emph{$\ell$-aligned}
if they both have sides that are subsets of~$\ell$, all sides of $A$
or~$B$ that lie in~$\ell$ have lengths that are integer multiples of
the same $s\in S$, and the distance between any endpoint of
one of those sides that lies in~$\ell$ and any other such endpoint is
also an integer multiple of~$s$.

The set $\mathcal{Q}$ naturally induces a graph whose vertices correspond
to the tiles in the set.  Two tiles $A$ and $B$ are connected by an edge
in the graph if there is a maximal segment~$\ell$ such that $A$ and $B$
are $\ell$-aligned and intersect in a line segment of positive length
that is a subset of $\ell$.
We say that $\mathcal{Q}$ is \emph{weakly aligned} if this graph is 
connected.  We say that it is \emph{strongly aligned} if it is weakly 
aligned and, for every maximal segment~$\ell$ determined by $\mathcal{Q}$,
and all tiles $A$ and~$B$ that have sides
lying in~$\ell$, $A$ and~$B$ are $\ell$-aligned.  We say that
\emph{$\mathcal{P}$ has the alignment property for side
lengths~$S$} if every weakly aligned set is strongly
aligned.  Here we drop the qualifiers ``strongly'' and ``weakly'' and refer 
to $\mathcal{Q}$, given the combination
of~$\mathcal{P}$ and~$S$, simply as \emph{aligned}.

\begin{lemma}
Any finite set of polykites, where the underlying kites have side
lengths $1$ and~$\sqrt{3}$, has the alignment property for side
lengths $\{1,\sqrt{3}\}$.
\end{lemma}

\begin{proof}
In the Laves tiling $[3.4.6.4]$, subdivide each kite into
24 $30^\circ$--$60^\circ$--$90^\circ$ triangles as shown in
Figure~\ref{fig:kitedecomp24}, forming a $[4.6.12]$
Laves tiling.  Furthermore, if a kite congruent to
one of those from the original $[3.4.6.4]$ adjoins edge-to-edge a kite
that is a union of triangles from that $[4.6.12]$ tiling, then it too
is such a union.  Thus all polykites
in any weakly aligned set are unions of tiles from the same $[4.6.12]$
tiling.  On any line in the union of the boundaries of the tiles from
$[4.6.12]$ that contains sides with rational length, sides of such
kites can only be at integer offsets from each other, and on the other
lines (containing sides with length a rational multiple
of~$\sqrt{3}$), sides of such kites can only be at offsets from each
other that are integer multiples of~$\sqrt{3}$ (both of these facts
follow from consideration of which vertices have the correct angles to
form a corner of such a kite).  So every weakly aligned set is
strongly aligned.
\end{proof}

\begin{figure}[htp!]
\begin{center}
\begin{tikzpicture}[x=1cm,y=1cm]
  \draw \vcoords{0}{0} -- \vcoords{3}{-3} -- \vcoords{6}{0} --
    \vcoords{0}{3} -- cycle;
  \draw \vcoords{0}{0} -- \vcoords{6}{0};
  \draw \vcoords{3}{-3} -- \vcoords{3}{1.5};
  \draw \vcoords{0}{3} -- \vcoords{4.5}{-1.5};
  \draw \vcoords{3}{-3} -- \vcoords{0}{3};
  \draw \vcoords{4}{-2} -- \vcoords{2}{2};
  \draw \vcoords{5}{-1} -- \vcoords{4}{1};
  \draw \vcoords{0}{0} -- \vcoords{2}{2};
  \draw \vcoords{1.5}{-1.5} -- \vcoords{4}{1};
  \draw \vcoords{0}{0} -- \vcoords{4}{-2};
  \draw \vcoords{0}{1.5} -- \vcoords{5}{-1};
\end{tikzpicture}
\end{center}
\caption{Decomposition of a kite into $24$~triangles}
\label{fig:kitedecomp24}
\end{figure}
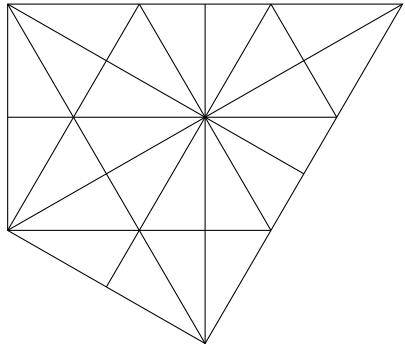

We now consider a set of tiles $\mathcal{P}$ that has the alignment 
property for side
lengths~$S$, and proceed to show that, in an appropriate
sense, only aligned tilings need to be considered.  Note that for
polykites, at this point ``aligned'' means only that the kites adjoin
edge-to-edge, which is weaker than all tiles coming from the same
underlying $[3.4.6.4]$ tiling; there will be further lemmas specific
to polykites to show that we need only consider tilings where all
tiles come from the same underlying $[3.4.6.4]$.

The alignment property implies that the tiles of any tiling can be
partitioned into strongly aligned sets such that for any maximal line
segment~$\ell$ in the union of the boundaries of the tiles, and any
two tiles in different sets that have sides sharing a segment of
positive length lying on~$\ell$, those two tiles are not
$\ell$-aligned.  We refer to these as the \emph{aligned components} of
the tiling.  Each aligned component is a connected set (possibly
unbounded), with connected interior.

Suppose $\mathcal{C}$ is an aligned component in a tiling, and
$\mathcal{D}$~is a connected component of the complement
of~$\mathcal{C}$ ($\mathcal{D}$ might be the interior of another
aligned component, or might be the interior of the union of more than
one aligned component).  The boundary of~$\mathcal{D}$ consists of a single
polygonal curve, either closed or infinite, and as with any other
polygon we may speak of its corners and sides.  Furthermore, that
curve cannot pass through the same point more than once; if it did,
either $\mathcal{D}$ (an open set) would not be connected, or
$\mathcal{C}$ would not have connected interior.

Consider traversing the boundary of~$\mathcal{D}$; note that
$\mathcal{D}$ always lies on the same side of the boundary during that
traversal.  When the traversal
encounters a corner, say that corner is convex if an open line segment
between two points on the curve sufficiently close to that corner but
on opposite sides of it is entirely within~$\mathcal{D}$.

When the boundary of~$\mathcal{D}$ is a closed curve, we must also
initially allow for $\mathcal{D}$ being inside that curve (a hole
in~$\mathcal{C}$) or outside (in which case $\mathcal{C}$ is
bounded).  The following lemma shows that the first of those cases
cannot occur, since if $\mathcal{D}$ is inside the curve it must have
at least three convex corners.

\begin{lemma}
The boundary of~$\mathcal{D}$ has no convex corners if it is a closed
curve (so in that case $\mathcal{C}$ must be a bounded convex set),
and at most one convex corner if it is an infinite curve.
\end{lemma}

\begin{proof}
If we consider any finite side of the
boundary of~$\mathcal{D}$, lying in some maximal line segment~$\ell$,
all the tiles lying on the other side of the side from~$\mathcal{D}$
are not $\ell$-aligned with any of those in~$\mathcal{D}$, meaning
that at least one of the two corners at the ends of that side lies in
the middle of a side of such a tile.  If $v$~is a convex corner, and
$v_1$, $v_2$, \ldots\ are successive vertices traversing the boundary
curve in one direction from~$v$, then we conclude that $v_1$~lies in
the middle of a side of a tile on the same line as~$vv_1$, then that
$v_2$~lies in the middle of a side of a tile on the same line
as~$v_1v_2$, and so on.  This results in a contradiction if we
encounter another convex corner (see Figure~\ref{fig:twoconvex} for an
illustration), or encounter $v$~again on a closed curve (see
Figure~\ref{fig:closedconvex}).
\end{proof}

\begin{figure}[htp!]
\begin{center}
\begin{tikzpicture}[x=1cm,y=1cm]
  \draw (0,0) -- (1,0) -- (2,1) -- (3,1) -- (4,0) -- (5,0);
  \draw (3,1) -- (2.5,1.5);
  \draw (2,1) -- (1.5,1);
  \draw (4.2,0.2) node {$v$};
  \draw (3.2,1.2) node {$v_1$};
  \draw (1.8,1.2) node {$v_2$};
\end{tikzpicture}
\end{center}
\caption{Boundary curve with two convex corners}
\label{fig:twoconvex}
\end{figure}
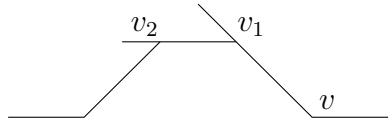

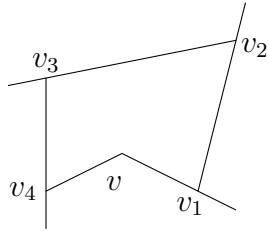
\begin{figure}[htp!]
\begin{center}
\begin{tikzpicture}[x=1cm,y=1cm]
  \draw (0,0) -- (1,0.5) -- (2,0) -- (2.5,2) -- (0,1.5) -- cycle;
  \draw (2,0) -- (2.5,-0.25);
  \draw (2.5,2) -- (2.625,2.5);
  \draw (0,1.5) -- (-0.5,1.4);
  \draw (0,0) -- (0,-0.5);
  \draw (0.9,0.1) node {$v$};
  \draw (1.9,-0.2) node {$v_1$};
  \draw (2.75,1.9) node {$v_2$};
  \draw (0,1.7) node {$v_3$};
  \draw (-0.3,0) node {$v_4$};
\end{tikzpicture}
\end{center}
\caption{Closed boundary curve with a convex corner}
\label{fig:closedconvex}
\end{figure}

Now we need to list the tiling properties for which our argument says
we do not need to consider unaligned tilings.  Let $H$ be one of the
following predicates on a tiling~$\mathcal{T}$; here, $k$~may be any
positive integer.

\begin{itemize}
\item $\mathcal{T}$ is a tiling (the trivial predicate).
\item $\mathcal{T}$ is a strongly periodic tiling.
\item $\mathcal{T}$ is a weakly periodic tiling.
\item $\mathcal{T}$ is a tiling with at most $k$~orbits of tiles
  under the action of its symmetry group.
\item $\mathcal{T}$ is an isohedral tiling by $180^\circ$ rotation.
\item $\mathcal{T}$ is an isohedral tiling by translation.
\end{itemize}

\begin{lemma}
\label{lemma:align}
Let $\mathcal{P}$ be a set of tiles with the alignment property for
a set $\mathcal{S}$ of side lengths.
If $\mathcal{P}$ admits a tiling with property~$H$, it admits an
aligned tiling with property~$H$.
\end{lemma}

\begin{proof}
Consider a tiling~$\mathcal{T}$ with property~$H$ and look at the
forms that aligned components take in that tiling.  By the previous lemma,
such components must be simply connected; either bounded, or unbounded
and with each boundary curve having at most one convex corner (in the
sense defined above, i.e., convex considered as a corner of a
connected component of the complement of the aligned component).

If such a component is the whole plane, the tiling is aligned and we
are done.  If it is a half-plane, form an aligned tiling of that
component and its reflection, and that tiling has property~$H$ (which
can only be the trivial property or ``weakly periodic'' in that case).
If it is a strip infinite in both directions, with straight lines as
its boundaries on both sides, repeat that strip by translation if
there is such a tiling that is aligned, and otherwise repeat it by
$180^\circ$ rotation; by considering each possible
predicate~$H$ separately, the resulting aligned tiling has property~$H$.

Otherwise, if there is any unbounded component, it does not have a
translation as a symmetry and $H$~is the trivial predicate.  If an
unbounded component contains balls of radius~$R$ for all~$R$, there
are aligned tilings of arbitrarily large regions of the plane, and so
of the whole plane.  The only way an unbounded component can avoid
containing such balls (given that each boundary curve has at most one
convex corner and all angles are rational sub-multiples of~$\pi$, which
implies that all boundary curves end in rays in finitely many
directions) is for it to include a semi-infinite strip (bounded on
either side by rays).  But there are only finitely many ways for
aligned tiles to cross the width of the strip at any point, so tiling
a semi-infinite strip implies the existence of a periodic aligned
tiling of an infinite strip, and thus a periodic aligned tiling of the
whole plane.

It remains to consider the case where there is no unbounded component.
If some component is a triangle or a quadrilateral, tiling that by
$180^\circ$~rotation yields an aligned tiling of the whole plane,
which must have property~$H$ (if property~$H$ is `isohedral tiling by
translation', this case cannot occur; a component could be an infinite
strip, but not a single parallelogram).  If components contain
unbounded balls, the tiling has no translation as a symmetry and
aligned tilings of arbitrarily large regions of the plane imply
aligned tilings of the whole plane.  If components do not contain
unbounded balls but also are not contained in bounded balls (i.e.,
they are of unbounded size in one direction only), they must have
pairs of opposite parallel sides, of unbounded length but a bounded
distance apart; the tiling is at most weakly periodic, and the same
argument as for components including a semi-infinite strip applies
since there are only finitely many possible distances between those
opposite parallel sides.

Otherwise, all components are convex polygons of bounded size with at
least five sides; we will show this case leads to a contradiction.  
Observe that every
vertex of the induced tiling by these polygons lies in the middle of a
side of one of the polygons and has degree exactly~$3$.
If a vertex does not lie in the middle of a side, or has
degree $4$ or more, there is a vertex~$v$ of a polygon~$P$, either not
in the middle of a side or in the middle of a side that is not
collinear with either of the sides $vv_1$ and $vv_2$ of~$P$ next
to~$v$, and the same argument that excluded convex corners on a closed
curve earlier serves to exclude this possibility as well.

We now apply Euler's theorem for plane maps.  Suppose, for some
sufficiently large~$R$, a ball of radius~$R$ contains $t_k$ components
that are $k$-gons (where $\sum_k t_k = \Omega(R^2)$).  A vertex of the
tiling is incident with two corners of tiles and a point in the middle
of a side, so there are~$\sum_k k t_k / 2 + O(R)$ vertices in that
ball.  The number of sides of edges in the tiling in that ball (i.e.,
twice the number of edges) is $\sum_k \frac{3}{2}k t_k + O(R)$, since
there are $\sum_k k t_k + O(R)$ sides of polygons, and each vertex is
in the middle of a side so serves to increase the number of sides of
edges by~$1$.  But now $\sum_k t_k + \sum_k k t_k / 2 = \sum_k
\frac{3}{4}k t_k + O(R)$, so $\sum_k t_k = \sum_k k t_k / 4 + O(R)$.
Since all $k \ge 5$, we have $\sum_k k t_k / 4 \ge \sum_k \frac{5}{4}
t_k$, contradicting that equality.
\end{proof}

Now we strengthen this lemma to a stricter notion of aligned tilings by
polykites, by considering what edge-to-edge tilings by the monokite
are possible.

\begin{lemma}
The only edge-to-edge tilings by the monokite are (a) the tiling
resulting from a~$180^\circ$~rotation about the midpoint of each side
(Figure~\ref{fig:kite180tiling}), and (b) tilings composed of rows of
equilateral triangles each composed of three kites, where some of
those rows may be translated relative to each other (by the length of the
long side of the kite) so they are no longer aligned as in the Laves
tiling.
\end{lemma}

\begin{figure}[htp!]
\begin{center}
\begin{tikzpicture}[x=2mm,y=2mm]
  \threekite{0}{0}{0}{};
  \threekite{180}{-1.5}{0}{};
  \threekite{0}{1}{1}{};
  \threekite{180}{-0.5}{1}{};
  \threekite{0}{2}{2}{};
  \threekite{180}{0.5}{2}{};
  \threekite{0}{1.5}{-1.5}{};
  \threekite{180}{0}{-1.5}{};
  \threekite{0}{2.5}{-0.5}{};
  \threekite{180}{1}{-0.5}{};
  \threekite{0}{3.5}{0.5}{};
  \threekite{180}{2}{0.5}{};
  \threekite{0}{3}{-3}{};
  \threekite{180}{1.5}{-3}{};
  \threekite{0}{4}{-2}{};
  \threekite{180}{2.5}{-2}{};
  \threekite{0}{5}{-1}{};
  \threekite{180}{3.5}{-1}{};
\end{tikzpicture}
\end{center}
\caption{$180^\circ$ tiling by kites}
\label{fig:kite180tiling}
\end{figure}
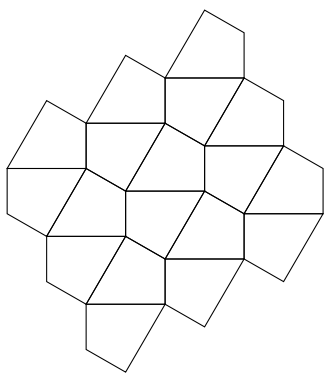

\begin{proof}
There are exactly two possible vertex figures in an edge-to-edge
tiling by the monokite that do not appear in the Laves tiling: one
with angles of $90^\circ$, $120^\circ$, $90^\circ$, and $60^\circ$ in
that order (Figure~\ref{fig:kite180vertex}), and one with angles of
$90^\circ$, $90^\circ$, $60^\circ$, $60^\circ$, and $60^\circ$ in that
order (Figure~\ref{fig:kiteslidvertex}).  If the first one occurs in a
tiling, successive surrounding tiles are forced (in the order
numbered) that force all neighbouring vertices, and so all vertices,
to have that vertex figure.  If the other one occurs in a tiling,
successive surrounding tiles are forced (in the order numbered, taking
into account that the first vertex figure cannot appear anywhere in
the tiling) that force two neighbouring vertices to have that vertex
figure, and thus force two rows of equilateral triangles, slid
relative to each other, and then the only possibilities on either side
of such a row are another such row in either of two positions.
\end{proof}

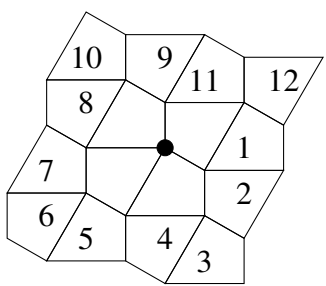
\begin{figure}[htp!]
\begin{center}
\begin{tikzpicture}[x=2mm,y=2mm]
  \threekite{0}{0}{0}{};
  \threekite{180}{-1.5}{0}{};
  \threekite{0}{1}{1}{};
  \threekite{180}{0}{-1.5}{};
  \threekite{180}{1}{-0.5}{1};
  \threekite{0}{2.5}{-0.5}{2};
  \threekite{180}{1.5}{-3}{3};
  \threekite{0}{1.5}{-1.5}{4};
  \threekite{180}{-1}{-2.5}{5};
  \threekite{0}{-1}{-1}{6};
  \threekite{180}{-2.5}{-1}{7};
  \threekite{0}{-1.5}{1.5}{8};
  \threekite{0}{-0.5}{2.5}{9};
  \threekite{180}{-3}{1.5}{10};
  \threekite{180}{-0.5}{1}{11};
  \threekite{0}{2}{2}{12};
  \markpt{2}{0};
\end{tikzpicture}
\end{center}
\caption{Vertex with angles of $90^\circ$, $120^\circ$, $90^\circ$, and
  $60^\circ$ in that order}
\label{fig:kite180vertex}
\end{figure}

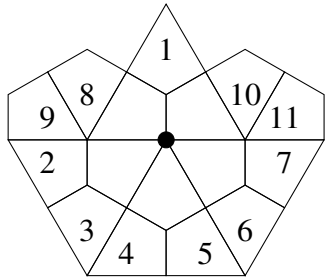
\begin{figure}[htp!]
\begin{center}
\begin{tikzpicture}[x=2mm,y=2mm]
  \threekite{0}{0}{0}{};
  \threekite{60}{0}{0}{};
  \threekite{120}{0}{0}{};
  \threekite{180}{-1.5}{0}{};
  \threekite{-60}{1.5}{0}{};
  \threekite{60}{-1.5}{3}{1};
  \threekite{120}{-3}{0}{2};
  \threekite{-120}{0}{-3}{3};
  \threekite{180}{0}{-3}{4};
  \threekite{-60}{3}{-3}{5};
  \threekite{-120}{3}{-3}{6};
  \threekite{0}{3}{0}{7};
  \threekite{-120}{-1.5}{0}{8};
  \threekite{-60}{-1.5}{0}{9};
  \threekite{-120}{1.5}{0}{10};
  \threekite{180}{1.5}{0}{11};
  \markpt{2}{0};
\end{tikzpicture}
\end{center}
\caption{Vertex with angles of $90^\circ$, $90^\circ$, $60^\circ$,
  $60^\circ$, and $60^\circ$ in that order}
\label{fig:kiteslidvertex}
\end{figure}

\begin{lemma}
\label{lemma:polykitealign}
If $\mathcal{P}$ is a finite set of closed topological disk polykites,
all from the same underlying Laves tiling, and $\mathcal{P}$ admits a
tiling with property~$H$, it admits a tiling with property~$H$ where
all polykites in the tiling are aligned to the same underlying Laves
tiling, except possibly when $\mathcal{P}$ contains the monokite and
$H$ is ``isohedral tiling by $180^\circ$ rotation''.
\end{lemma}

\begin{proof}
If the edge-to-edge tiling with property~$H$ (``aligned'' in the more
general sense) is the~$180^\circ$ tiling by the monokite, we are done
because the monokite admits an isohedral tiling.  Otherwise, taking a
minimal block of consecutive rows of equilateral triangles filled
exactly with tiles from~$\mathcal{P}$ and translating it so as to be
aligned with the underlying Laves tiling produces a tiling with
property~$H$.
\end{proof}

\begin{lemma}
\label{lemma:tileaalign}
All tilings by the hat polykite are aligned
to an underlying $[3.4.6.4]$ Laves tiling.
\end{lemma}

\begin{proof}
Note that any maximal segment in the union of the boundaries of tiles in such
a tiling can contain no more than two sides of tiles, since any
$90^\circ$~angle is adjacent on either side to angles greater
than~$90^\circ$.  In particular, there are no infinite rays contained
in the union of the boundaries of tiles.

This constraint immediately excludes the case of aligned tilings decomposing 
into rows of equilateral triangles slid relative to each other, and no two
adjacent kites in the polykite are consistent with the
$180^\circ$~rotation tiling by the monokite.  So any tiling not
aligned with an underlying $[3.4.6.4]$ is also unaligned in the more
general sense, and we consider aligned components.  Because there are
no infinite rays among the boundaries of tiles, such aligned
components must be bounded convex sets.  The corners of those sets
must be corners of a single polykite (since any two angles of the
polykite add to at least~$180^\circ$).  But no corner of the polykite
can be a corner of a convex set tiled by the polykite: four have
reflex angles, seven are adjacent to a reflex angle, and for the
remaining two, extending one of the sides from that vertex cuts off a
region too small to be filled by polykites
(Figure~\ref{fig:polykiteconvex}).
\end{proof}

\begin{figure}[htp!]
\begin{center}
\begin{tikzpicture}[x=5mm,y=5mm]
  \draw[\colextside,ultra thick] \vcoords{0}{0} -- \vcoords{4}{-4};
  \draw[\colextside,ultra thick] \vcoords{6}{-5} -- \vcoords{6}{-2};
  \tileA{0}{0}{0}{};
  \markpt{0}{0};
  \markpt{6}{-5};
\end{tikzpicture}
\end{center}
\caption{Extending a side of the polykite from either vertex that is
  neither a reflex angle nor adjacent to one}
\label{fig:polykiteconvex}
\end{figure}
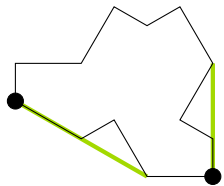

\section{Case analysis for $1$-patches}
\label{sec:patches}

We present here details of a computer-generated but human-verifiable
case analysis, based on consideration of $1$-patches rather than
$2$-patches.  This analysis can be used to complete a variant of the
proof in \secref{sec:clusters} that
when tiles in a tiling by the hat polykite
are assigned labels following the rules given there, then (a)~the labels
assigned do induce a division into the clusters shown, and (b)~the
clusters adjoin other clusters in accordance with the matching rules.
As is justified in Appendix~\ref{sec:align}, we only consider tilings
where all tiles are aligned with an underlying $[3.4.6.4]$ Laves
tiling.

\subsection{Enumeration of neighbours}

First we produce a list of possible neighbours of the hat polykite in a
tiling.  There are $58$~possible neighbours when we only require such
a neighbour not to intersect the original polykite; these are shown in
Figure~\ref{fig:nbr}, with the original polykite shaded.  The first
$41$ of these neighbours remain in consideration for the enumeration
of $1$-patches.  The final~$17$ are immediately eliminated (in the
order shown) because they cannot be extended to a tiling: either there
is no possible neighbour that can contain kite shaded in green (without
resulting in an intersection, or a pair of tiles that were previously
eliminated as possible neighbours), or we eliminated $Y$ as a
neighbour of~$X$ and so can also eliminate $X$ as a neighbour of~$Y$.

\begin{figure}[htp!]
\renewcommand\thesubfigure{\arabic{subfigure}}%
\captionsetup{margin=0pt,justification=raggedright}%
\begin{center}
\subfloat[Possible neighbour~$1$]{%
\begin{minipage}[b]{6cm}
\begin{center}
\begin{tikzpicture}[x=3mm,y=3mm]
  \ftileA{0}{0}{0}{};
  \tileA{0}{-2}{1}{};
\end{tikzpicture}%
\end{center}%
\end{minipage}%
} \qquad \subfloat[Possible neighbour~$2$]{%
\begin{minipage}[b]{6cm}
\begin{center}
\begin{tikzpicture}[x=3mm,y=3mm]
  \ftileA{0}{0}{0}{};
  \tileA{60}{-1}{-1}{};
\end{tikzpicture}%
\end{center}%
\end{minipage}%
} \\ \subfloat[Possible neighbour~$3$]{%
\begin{minipage}[b]{6cm}
\begin{center}
\begin{tikzpicture}[x=3mm,y=3mm]
  \ftileA{0}{0}{0}{};
  \tileA{60}{-1}{0}{};
\end{tikzpicture}%
\end{center}%
\end{minipage}%
} \qquad \subfloat[Possible neighbour~$4$]{%
\begin{minipage}[b]{6cm}
\begin{center}
\begin{tikzpicture}[x=3mm,y=3mm]
  \ftileA{0}{0}{0}{};
  \tileA{120}{-1}{0}{};
\end{tikzpicture}%
\end{center}%
\end{minipage}%
} \\ \subfloat[Possible neighbour~$5$]{%
\begin{minipage}[b]{6cm}
\begin{center}
\begin{tikzpicture}[x=3mm,y=3mm]
  \ftileA{0}{0}{0}{};
  \tileA{180}{-1}{0}{};
\end{tikzpicture}%
\end{center}%
\end{minipage}%
} \qquad \subfloat[Possible neighbour~$6$]{%
\begin{minipage}[b]{6cm}
\begin{center}
\begin{tikzpicture}[x=3mm,y=3mm]
  \ftileA{0}{0}{0}{};
  \tileA{300}{-1}{0}{};
\end{tikzpicture}%
\end{center}%
\end{minipage}%
}%
\end{center}
\caption{Possible neighbours (part 1)}
\label{fig:nbr}
\end{figure}
\begin{figure}[htp!]
\renewcommand\thesubfigure{\arabic{subfigure}}%
\ContinuedFloat
\captionsetup{margin=0pt,justification=raggedright}%
\begin{center}
\subfloat[Possible neighbour~$7$]{%
\begin{minipage}[b]{6cm}
\begin{center}
\begin{tikzpicture}[x=3mm,y=3mm]
  \ftileA{0}{0}{0}{};
  \tileAr{60}{-1}{0}{};
\end{tikzpicture}%
\end{center}%
\end{minipage}%
} \qquad \subfloat[Possible neighbour~$8$]{%
\begin{minipage}[b]{6cm}
\begin{center}
\begin{tikzpicture}[x=3mm,y=3mm]
  \ftileA{0}{0}{0}{};
  \tileA{240}{-1}{1}{};
\end{tikzpicture}%
\end{center}%
\end{minipage}%
} \\ \subfloat[Possible neighbour~$9$]{%
\begin{minipage}[b]{6cm}
\begin{center}
\begin{tikzpicture}[x=3mm,y=3mm]
  \ftileA{0}{0}{0}{};
  \tileA{300}{-1}{1}{};
\end{tikzpicture}%
\end{center}%
\end{minipage}%
} \qquad \subfloat[Possible neighbour~$10$]{%
\begin{minipage}[b]{6cm}
\begin{center}
\begin{tikzpicture}[x=3mm,y=3mm]
  \ftileA{0}{0}{0}{};
  \tileAr{0}{-1}{1}{};
\end{tikzpicture}%
\end{center}%
\end{minipage}%
} \\ \subfloat[Possible neighbour~$11$]{%
\begin{minipage}[b]{6cm}
\begin{center}
\begin{tikzpicture}[x=3mm,y=3mm]
  \ftileA{0}{0}{0}{};
  \tileAr{300}{-1}{1}{};
\end{tikzpicture}%
\end{center}%
\end{minipage}%
} \qquad \subfloat[Possible neighbour~$12$]{%
\begin{minipage}[b]{6cm}
\begin{center}
\begin{tikzpicture}[x=3mm,y=3mm]
  \ftileA{0}{0}{0}{};
  \tileAr{240}{-1}{2}{};
\end{tikzpicture}%
\end{center}%
\end{minipage}%
} \\ \subfloat[Possible neighbour~$13$]{%
\begin{minipage}[b]{6cm}
\begin{center}
\begin{tikzpicture}[x=3mm,y=3mm]
  \ftileA{0}{0}{0}{};
  \tileA{0}{0}{-1}{};
\end{tikzpicture}%
\end{center}%
\end{minipage}%
} \qquad \subfloat[Possible neighbour~$14$]{%
\begin{minipage}[b]{6cm}
\begin{center}
\begin{tikzpicture}[x=3mm,y=3mm]
  \ftileA{0}{0}{0}{};
  \tileA{180}{0}{-1}{};
\end{tikzpicture}%
\end{center}%
\end{minipage}%
} \\ \subfloat[Possible neighbour~$15$]{%
\begin{minipage}[b]{6cm}
\begin{center}
\begin{tikzpicture}[x=3mm,y=3mm]
  \ftileA{0}{0}{0}{};
  \tileA{240}{0}{-1}{};
\end{tikzpicture}%
\end{center}%
\end{minipage}%
} \qquad \subfloat[Possible neighbour~$16$]{%
\begin{minipage}[b]{6cm}
\begin{center}
\begin{tikzpicture}[x=3mm,y=3mm]
  \ftileA{0}{0}{0}{};
  \tileAr{120}{0}{-1}{};
\end{tikzpicture}%
\end{center}%
\end{minipage}%
}%
\end{center}
\caption{Possible neighbours (part 2)}
\label{fig:nbr:2}
\end{figure}
\begin{figure}[htp!]
\renewcommand\thesubfigure{\arabic{subfigure}}%
\ContinuedFloat
\captionsetup{margin=0pt,justification=raggedright}%
\begin{center}
\subfloat[Possible neighbour~$17$]{%
\begin{minipage}[b]{6cm}
\begin{center}
\begin{tikzpicture}[x=3mm,y=3mm]
  \ftileA{0}{0}{0}{};
  \tileA{0}{0}{1}{};
\end{tikzpicture}%
\end{center}%
\end{minipage}%
} \qquad \subfloat[Possible neighbour~$18$]{%
\begin{minipage}[b]{6cm}
\begin{center}
\begin{tikzpicture}[x=3mm,y=3mm]
  \ftileA{0}{0}{0}{};
  \tileA{60}{0}{1}{};
\end{tikzpicture}%
\end{center}%
\end{minipage}%
} \\ \subfloat[Possible neighbour~$19$]{%
\begin{minipage}[b]{6cm}
\begin{center}
\begin{tikzpicture}[x=3mm,y=3mm]
  \ftileA{0}{0}{0}{};
  \tileA{120}{0}{1}{};
\end{tikzpicture}%
\end{center}%
\end{minipage}%
} \qquad \subfloat[Possible neighbour~$20$]{%
\begin{minipage}[b]{6cm}
\begin{center}
\begin{tikzpicture}[x=3mm,y=3mm]
  \ftileA{0}{0}{0}{};
  \tileA{180}{0}{1}{};
\end{tikzpicture}%
\end{center}%
\end{minipage}%
} \\ \subfloat[Possible neighbour~$21$]{%
\begin{minipage}[b]{6cm}
\begin{center}
\begin{tikzpicture}[x=3mm,y=3mm]
  \ftileA{0}{0}{0}{};
  \tileAr{0}{0}{1}{};
\end{tikzpicture}%
\end{center}%
\end{minipage}%
} \qquad \subfloat[Possible neighbour~$22$]{%
\begin{minipage}[b]{6cm}
\begin{center}
\begin{tikzpicture}[x=3mm,y=3mm]
  \ftileA{0}{0}{0}{};
  \tileAr{60}{0}{1}{};
\end{tikzpicture}%
\end{center}%
\end{minipage}%
} \\ \subfloat[Possible neighbour~$23$]{%
\begin{minipage}[b]{6cm}
\begin{center}
\begin{tikzpicture}[x=3mm,y=3mm]
  \ftileA{0}{0}{0}{};
  \tileAr{120}{0}{1}{};
\end{tikzpicture}%
\end{center}%
\end{minipage}%
} \qquad \subfloat[Possible neighbour~$24$]{%
\begin{minipage}[b]{6cm}
\begin{center}
\begin{tikzpicture}[x=3mm,y=3mm]
  \ftileA{0}{0}{0}{};
  \tileAr{180}{0}{1}{};
\end{tikzpicture}%
\end{center}%
\end{minipage}%
} \\ \subfloat[Possible neighbour~$25$]{%
\begin{minipage}[b]{6cm}
\begin{center}
\begin{tikzpicture}[x=3mm,y=3mm]
  \ftileA{0}{0}{0}{};
  \tileA{300}{0}{2}{};
\end{tikzpicture}%
\end{center}%
\end{minipage}%
} \qquad \subfloat[Possible neighbour~$26$]{%
\begin{minipage}[b]{6cm}
\begin{center}
\begin{tikzpicture}[x=3mm,y=3mm]
  \ftileA{0}{0}{0}{};
  \tileA{120}{1}{-2}{};
\end{tikzpicture}%
\end{center}%
\end{minipage}%
}%
\end{center}
\caption{Possible neighbours (part 3)}
\label{fig:nbr:3}
\end{figure}
\begin{figure}[htp!]
\renewcommand\thesubfigure{\arabic{subfigure}}%
\ContinuedFloat
\captionsetup{margin=0pt,justification=raggedright}%
\begin{center}
\subfloat[Possible neighbour~$27$]{%
\begin{minipage}[b]{6cm}
\begin{center}
\begin{tikzpicture}[x=3mm,y=3mm]
  \ftileA{0}{0}{0}{};
  \tileA{300}{1}{-1}{};
\end{tikzpicture}%
\end{center}%
\end{minipage}%
} \qquad \subfloat[Possible neighbour~$28$]{%
\begin{minipage}[b]{6cm}
\begin{center}
\begin{tikzpicture}[x=3mm,y=3mm]
  \ftileA{0}{0}{0}{};
  \tileAr{180}{1}{-1}{};
\end{tikzpicture}%
\end{center}%
\end{minipage}%
} \\ \subfloat[Possible neighbour~$29$]{%
\begin{minipage}[b]{6cm}
\begin{center}
\begin{tikzpicture}[x=3mm,y=3mm]
  \ftileA{0}{0}{0}{};
  \tileA{60}{1}{0}{};
\end{tikzpicture}%
\end{center}%
\end{minipage}%
} \qquad \subfloat[Possible neighbour~$30$]{%
\begin{minipage}[b]{6cm}
\begin{center}
\begin{tikzpicture}[x=3mm,y=3mm]
  \ftileA{0}{0}{0}{};
  \tileAr{300}{1}{0}{};
\end{tikzpicture}%
\end{center}%
\end{minipage}%
} \\ \subfloat[Possible neighbour~$31$]{%
\begin{minipage}[b]{6cm}
\begin{center}
\begin{tikzpicture}[x=3mm,y=3mm]
  \ftileA{0}{0}{0}{};
  \tileA{180}{1}{1}{};
\end{tikzpicture}%
\end{center}%
\end{minipage}%
} \qquad \subfloat[Possible neighbour~$32$]{%
\begin{minipage}[b]{6cm}
\begin{center}
\begin{tikzpicture}[x=3mm,y=3mm]
  \ftileA{0}{0}{0}{};
  \tileA{240}{1}{1}{};
\end{tikzpicture}%
\end{center}%
\end{minipage}%
} \\ \subfloat[Possible neighbour~$33$]{%
\begin{minipage}[b]{6cm}
\begin{center}
\begin{tikzpicture}[x=3mm,y=3mm]
  \ftileA{0}{0}{0}{};
  \tileA{60}{2}{-2}{};
\end{tikzpicture}%
\end{center}%
\end{minipage}%
} \qquad \subfloat[Possible neighbour~$34$]{%
\begin{minipage}[b]{6cm}
\begin{center}
\begin{tikzpicture}[x=3mm,y=3mm]
  \ftileA{0}{0}{0}{};
  \tileA{120}{2}{-2}{};
\end{tikzpicture}%
\end{center}%
\end{minipage}%
} \\ \subfloat[Possible neighbour~$35$]{%
\begin{minipage}[b]{6cm}
\begin{center}
\begin{tikzpicture}[x=3mm,y=3mm]
  \ftileA{0}{0}{0}{};
  \tileA{0}{2}{-1}{};
\end{tikzpicture}%
\end{center}%
\end{minipage}%
} \qquad \subfloat[Possible neighbour~$36$]{%
\begin{minipage}[b]{6cm}
\begin{center}
\begin{tikzpicture}[x=3mm,y=3mm]
  \ftileA{0}{0}{0}{};
  \tileA{120}{2}{-1}{};
\end{tikzpicture}%
\end{center}%
\end{minipage}%
} \\ \subfloat[Possible neighbour~$37$]{%
\begin{minipage}[b]{6cm}
\begin{center}
\begin{tikzpicture}[x=3mm,y=3mm]
  \ftileA{0}{0}{0}{};
  \tileA{240}{2}{-1}{};
\end{tikzpicture}%
\end{center}%
\end{minipage}%
} \qquad \subfloat[Possible neighbour~$38$]{%
\begin{minipage}[b]{6cm}
\begin{center}
\begin{tikzpicture}[x=3mm,y=3mm]
  \ftileA{0}{0}{0}{};
  \tileA{300}{2}{-1}{};
\end{tikzpicture}%
\end{center}%
\end{minipage}%
}%
\end{center}
\caption{Possible neighbours (part 4)}
\label{fig:nbr:4}
\end{figure}
\begin{figure}[htp!]
\renewcommand\thesubfigure{\arabic{subfigure}}%
\ContinuedFloat
\captionsetup{margin=0pt,justification=raggedright}%
\begin{center}
\subfloat[Possible neighbour~$39$]{%
\begin{minipage}[b]{6cm}
\begin{center}
\begin{tikzpicture}[x=3mm,y=3mm]
  \ftileA{0}{0}{0}{};
  \tileAr{240}{2}{-1}{};
\end{tikzpicture}%
\end{center}%
\end{minipage}%
} \qquad \subfloat[Possible neighbour~$40$]{%
\begin{minipage}[b]{6cm}
\begin{center}
\begin{tikzpicture}[x=3mm,y=3mm]
  \ftileA{0}{0}{0}{};
  \tileA{240}{2}{0}{};
\end{tikzpicture}%
\end{center}%
\end{minipage}%
} \\ \subfloat[Possible neighbour~$41$]{%
\begin{minipage}[b]{6cm}
\begin{center}
\begin{tikzpicture}[x=3mm,y=3mm]
  \ftileA{0}{0}{0}{};
  \tileA{180}{3}{-2}{};
\end{tikzpicture}%
\end{center}%
\end{minipage}%
} \qquad \subfloat[Possible neighbour~$42$ (eliminated by considering neighbours containing the shaded kite)]{%
\begin{minipage}[b]{6cm}
\begin{center}
\begin{tikzpicture}[x=3mm,y=3mm]
  \ffkite{120}{1}{-1};
  \ftileA{0}{0}{0}{};
  \tileAr{300}{0}{-1}{};
\end{tikzpicture}%
\end{center}%
\end{minipage}%
} \\ \subfloat[Possible neighbour~$43$ (eliminated by considering neighbours containing the shaded kite)]{%
\begin{minipage}[b]{6cm}
\begin{center}
\begin{tikzpicture}[x=3mm,y=3mm]
  \ffkite{120}{1}{-1};
  \ftileA{0}{0}{0}{};
  \tileA{60}{1}{-2}{};
\end{tikzpicture}%
\end{center}%
\end{minipage}%
} \qquad \subfloat[Possible neighbour~$44$ (eliminated together with possible neighbour~$43$)]{%
\begin{minipage}[b]{6cm}
\begin{center}
\begin{tikzpicture}[x=3mm,y=3mm]
  \ftileA{0}{0}{0}{};
  \tileA{300}{1}{1}{};
\end{tikzpicture}%
\end{center}%
\end{minipage}%
} \\ \subfloat[Possible neighbour~$45$ (eliminated by considering neighbours containing the shaded kite)]{%
\begin{minipage}[b]{6cm}
\begin{center}
\begin{tikzpicture}[x=3mm,y=3mm]
  \ffkite{0}{1}{-1};
  \ftileA{0}{0}{0}{};
  \tileAr{0}{1}{-2}{};
\end{tikzpicture}%
\end{center}%
\end{minipage}%
} \qquad \subfloat[Possible neighbour~$46$ (eliminated by considering neighbours containing the shaded kite)]{%
\begin{minipage}[b]{6cm}
\begin{center}
\begin{tikzpicture}[x=3mm,y=3mm]
  \ffkite{120}{1}{-1};
  \ftileA{0}{0}{0}{};
  \tileAr{60}{1}{-2}{};
\end{tikzpicture}%
\end{center}%
\end{minipage}%
}%
\end{center}
\caption{Possible neighbours (part 5)}
\label{fig:nbr:5}
\end{figure}
\begin{figure}[htp!]
\renewcommand\thesubfigure{\arabic{subfigure}}%
\ContinuedFloat
\captionsetup{margin=0pt,justification=raggedright}%
\begin{center}
\subfloat[Possible neighbour~$47$ (eliminated together with possible neighbour~$46$)]{%
\begin{minipage}[b]{6cm}
\begin{center}
\begin{tikzpicture}[x=3mm,y=3mm]
  \ftileA{0}{0}{0}{};
  \tileAr{60}{2}{-1}{};
\end{tikzpicture}%
\end{center}%
\end{minipage}%
} \qquad \subfloat[Possible neighbour~$48$ (eliminated by considering neighbours containing the shaded kite)]{%
\begin{minipage}[b]{6cm}
\begin{center}
\begin{tikzpicture}[x=3mm,y=3mm]
  \ffkite{60}{1}{0};
  \ftileA{0}{0}{0}{};
  \tileAr{240}{1}{1}{};
\end{tikzpicture}%
\end{center}%
\end{minipage}%
} \\ \subfloat[Possible neighbour~$49$ (eliminated by considering neighbours containing the shaded kite)]{%
\begin{minipage}[b]{6cm}
\begin{center}
\begin{tikzpicture}[x=3mm,y=3mm]
  \ffkite{120}{1}{-1};
  \ftileA{0}{0}{0}{};
  \tileA{180}{2}{-2}{};
\end{tikzpicture}%
\end{center}%
\end{minipage}%
} \qquad \subfloat[Possible neighbour~$50$ (eliminated by considering neighbours containing the shaded kite)]{%
\begin{minipage}[b]{6cm}
\begin{center}
\begin{tikzpicture}[x=3mm,y=3mm]
  \ffkite{0}{1}{-1};
  \ftileA{0}{0}{0}{};
  \tileAr{120}{2}{-2}{};
\end{tikzpicture}%
\end{center}%
\end{minipage}%
} \\ \subfloat[Possible neighbour~$51$ (eliminated together with possible neighbour~$50$)]{%
\begin{minipage}[b]{6cm}
\begin{center}
\begin{tikzpicture}[x=3mm,y=3mm]
  \ftileA{0}{0}{0}{};
  \tileAr{120}{2}{0}{};
\end{tikzpicture}%
\end{center}%
\end{minipage}%
} \qquad \subfloat[Possible neighbour~$52$ (eliminated by considering neighbours containing the shaded kite)]{%
\begin{minipage}[b]{6cm}
\begin{center}
\begin{tikzpicture}[x=3mm,y=3mm]
  \ffkite{60}{1}{0};
  \ftileA{0}{0}{0}{};
  \tileAr{0}{2}{-1}{};
\end{tikzpicture}%
\end{center}%
\end{minipage}%
} \\ \subfloat[Possible neighbour~$53$ (eliminated together with possible neighbour~$52$)]{%
\begin{minipage}[b]{6cm}
\begin{center}
\begin{tikzpicture}[x=3mm,y=3mm]
  \ftileA{0}{0}{0}{};
  \tileAr{0}{-1}{-1}{};
\end{tikzpicture}%
\end{center}%
\end{minipage}%
} \qquad \subfloat[Possible neighbour~$54$ (eliminated by considering neighbours containing the shaded kite)]{%
\begin{minipage}[b]{6cm}
\begin{center}
\begin{tikzpicture}[x=3mm,y=3mm]
  \ffkite{60}{1}{0};
  \ftileA{0}{0}{0}{};
  \tileA{180}{2}{0}{};
\end{tikzpicture}%
\end{center}%
\end{minipage}%
} \\ \subfloat[Possible neighbour~$55$ (eliminated by considering neighbours containing the shaded kite)]{%
\begin{minipage}[b]{6cm}
\begin{center}
\begin{tikzpicture}[x=3mm,y=3mm]
  \ffkite{60}{1}{0};
  \ftileA{0}{0}{0}{};
  \tileAr{180}{2}{0}{};
\end{tikzpicture}%
\end{center}%
\end{minipage}%
} \qquad \subfloat[Possible neighbour~$56$ (eliminated by considering neighbours containing the shaded kite)]{%
\begin{minipage}[b]{6cm}
\begin{center}
\begin{tikzpicture}[x=3mm,y=3mm]
  \ffkite{240}{0}{-1};
  \ftileA{0}{0}{0}{};
  \tileAr{180}{-1}{0}{};
\end{tikzpicture}%
\end{center}%
\end{minipage}%
}%
\end{center}
\caption{Possible neighbours (part 6)}
\label{fig:nbr:6}
\end{figure}
\begin{figure}[htp!]
\renewcommand\thesubfigure{\arabic{subfigure}}%
\ContinuedFloat
\captionsetup{margin=0pt,justification=raggedright}%
\begin{center}
\subfloat[Possible neighbour~$57$ (eliminated by considering neighbours containing the shaded kite)]{%
\begin{minipage}[b]{6cm}
\begin{center}
\begin{tikzpicture}[x=3mm,y=3mm]
  \ffkite{180}{0}{-1};
  \ftileA{0}{0}{0}{};
  \tileAr{240}{-1}{0}{};
\end{tikzpicture}%
\end{center}%
\end{minipage}%
} \qquad \subfloat[Possible neighbour~$58$ (eliminated together with possible neighbour~$57$)]{%
\begin{minipage}[b]{6cm}
\begin{center}
\begin{tikzpicture}[x=3mm,y=3mm]
  \ftileA{0}{0}{0}{};
  \tileAr{240}{0}{-1}{};
\end{tikzpicture}%
\end{center}%
\end{minipage}%
}%
\end{center}
\caption{Possible neighbours (part 7)}
\label{fig:nbr:7}
\end{figure}

\FloatBarrier

\subsection{Enumeration of $1$-patches}

Having produced a list of possible neighbours, we now proceed to
enumerating possible $1$-patches.  When we have a partial $1$-patch
(some number of neighbours for the original, shaded polykite), we pick
some kite neighbouring that original polykite and enumerate the
possible neighbouring polykites containing that kite, excluding any
that would result in the patch containing two polykites that either
intersect or form a pair of neighbours previously ruled out; the kite
we use is chosen so that the number of choices for the neighbour added
is minimal.  This process results in $37$~possible $1$-patches; the
partial patches from the search process are shown in
Figure~\ref{fig:partpatch} and the $1$-patches are shown in
Figures~\ref{fig:patch}.

Some of the $1$-patches found can be immediately eliminated at this
point, by identifying a tile in the $1$-patch that cannot itself be
surrounded by any of the $1$-patches (that has not yet been
eliminated) without resulting in either an intersection or a pair of
neighbours that were previously ruled out.  In the $12$ cases 
implicated here, the tile that cannot be surrounded is shaded, and they are
eliminated in the order shown, leaving $25$~remaining $1$-patches.
For each of those remaining $1$-patches, the classification of the
central tile by the
rules in \secref{sec:clusters} is shown.

\begin{figure}[htp!]
\renewcommand\thesubfigure{\arabic{subfigure}}%
\captionsetup{margin=0pt,justification=raggedright}%
\begin{center}
\subfloat[Partial patch 1 (extends to partial patches 2--4)]{%
\begin{minipage}[b]{6cm}
\begin{center}
\begin{tikzpicture}[x=3mm,y=3mm]
  \ffkite{180}{0}{-1};
  \ftileA{0}{0}{0}{};
\end{tikzpicture}%
\end{center}%
\end{minipage}%
} \qquad \subfloat[Partial patch 2 (extends to partial patches 5--6)]{%
\begin{minipage}[b]{6cm}
\begin{center}
\begin{tikzpicture}[x=3mm,y=3mm]
  \ffkite{240}{0}{-1};
  \tileA{0}{0}{-1}{};
  \ftileA{0}{0}{0}{};
\end{tikzpicture}%
\end{center}%
\end{minipage}%
}%
\end{center}
\caption{Partial patches (part 1)}
\label{fig:partpatch}
\end{figure}
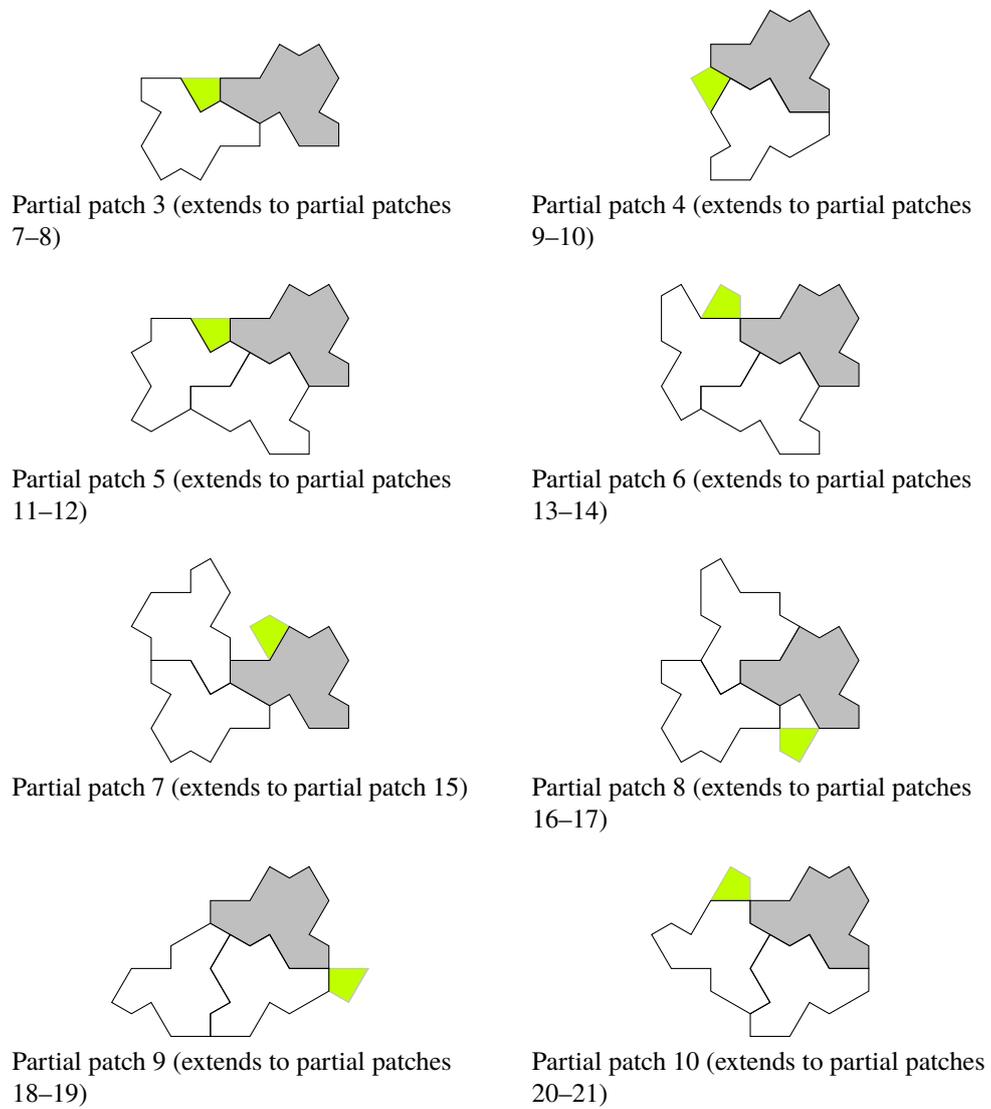
\begin{figure}[htp!]
\renewcommand\thesubfigure{\arabic{subfigure}}%
\ContinuedFloat
\captionsetup{margin=0pt,justification=raggedright}%
\begin{center}
\subfloat[Partial patch 3 (extends to partial patches 7--8)]{%
\begin{minipage}[b]{6cm}
\begin{center}
\begin{tikzpicture}[x=3mm,y=3mm]
  \ffkite{120}{-1}{0};
  \tileA{180}{0}{-1}{};
  \ftileA{0}{0}{0}{};
\end{tikzpicture}%
\end{center}%
\end{minipage}%
} \qquad \subfloat[Partial patch 4 (extends to partial patches 9--10)]{%
\begin{minipage}[b]{6cm}
\begin{center}
\begin{tikzpicture}[x=3mm,y=3mm]
  \ffkite{240}{0}{-1};
  \ftileA{0}{0}{0}{};
  \tileAr{180}{1}{-1}{};
\end{tikzpicture}%
\end{center}%
\end{minipage}%
} \\ \subfloat[Partial patch 5 (extends to partial patches 11--12)]{%
\begin{minipage}[b]{6cm}
\begin{center}
\begin{tikzpicture}[x=3mm,y=3mm]
  \ffkite{120}{-1}{0};
  \tileA{60}{-1}{-1}{};
  \tileA{0}{0}{-1}{};
  \ftileA{0}{0}{0}{};
\end{tikzpicture}%
\end{center}%
\end{minipage}%
} \qquad \subfloat[Partial patch 6 (extends to partial patches 13--14)]{%
\begin{minipage}[b]{6cm}
\begin{center}
\begin{tikzpicture}[x=3mm,y=3mm]
  \ffkite{180}{-1}{0};
  \tileA{300}{-1}{0}{};
  \tileA{0}{0}{-1}{};
  \ftileA{0}{0}{0}{};
\end{tikzpicture}%
\end{center}%
\end{minipage}%
} \\ \subfloat[Partial patch 7 (extends to partial patch 15)]{%
\begin{minipage}[b]{6cm}
\begin{center}
\begin{tikzpicture}[x=3mm,y=3mm]
  \ffkite{240}{0}{0};
  \tileA{120}{-1}{0}{};
  \tileA{180}{0}{-1}{};
  \ftileA{0}{0}{0}{};
\end{tikzpicture}%
\end{center}%
\end{minipage}%
} \qquad \subfloat[Partial patch 8 (extends to partial patches 16--17)]{%
\begin{minipage}[b]{6cm}
\begin{center}
\begin{tikzpicture}[x=3mm,y=3mm]
  \ffkite{0}{1}{-1};
  \tileA{300}{-1}{1}{};
  \tileA{180}{0}{-1}{};
  \ftileA{0}{0}{0}{};
\end{tikzpicture}%
\end{center}%
\end{minipage}%
} \\ \subfloat[Partial patch 9 (extends to partial patches 18--19)]{%
\begin{minipage}[b]{6cm}
\begin{center}
\begin{tikzpicture}[x=3mm,y=3mm]
  \ffkite{0}{2}{-1};
  \tileA{240}{0}{-1}{};
  \ftileA{0}{0}{0}{};
  \tileAr{180}{1}{-1}{};
\end{tikzpicture}%
\end{center}%
\end{minipage}%
} \qquad \subfloat[Partial patch 10 (extends to partial patches 20--21)]{%
\begin{minipage}[b]{6cm}
\begin{center}
\begin{tikzpicture}[x=3mm,y=3mm]
  \ffkite{180}{-1}{0};
  \tileAr{120}{0}{-1}{};
  \ftileA{0}{0}{0}{};
  \tileAr{180}{1}{-1}{};
\end{tikzpicture}%
\end{center}%
\end{minipage}%
}%
\end{center}
\caption{Partial patches (part 2)}
\label{fig:partpatch:2}
\end{figure}
\begin{figure}[htp!]
\renewcommand\thesubfigure{\arabic{subfigure}}%
\ContinuedFloat
\captionsetup{margin=0pt,justification=raggedright}%
\begin{center}
\subfloat[Partial patch 11 (extends to partial patch 22)]{%
\begin{minipage}[b]{6cm}
\begin{center}
\begin{tikzpicture}[x=3mm,y=3mm]
  \ffkite{240}{0}{0};
  \tileA{60}{-1}{-1}{};
  \tileA{120}{-1}{0}{};
  \tileA{0}{0}{-1}{};
  \ftileA{0}{0}{0}{};
\end{tikzpicture}%
\end{center}%
\end{minipage}%
} \qquad \subfloat[Partial patch 12 (extends to partial patches 23--24)]{%
\begin{minipage}[b]{6cm}
\begin{center}
\begin{tikzpicture}[x=3mm,y=3mm]
  \ffkite{60}{1}{-1};
  \tileA{60}{-1}{-1}{};
  \tileA{300}{-1}{1}{};
  \tileA{0}{0}{-1}{};
  \ftileA{0}{0}{0}{};
\end{tikzpicture}%
\end{center}%
\end{minipage}%
} \\ \subfloat[Partial patch 13 (extends to partial patch 25)]{%
\begin{minipage}[b]{6cm}
\begin{center}
\begin{tikzpicture}[x=3mm,y=3mm]
  \ffkite{240}{0}{0};
  \tileA{0}{-2}{1}{};
  \tileA{300}{-1}{0}{};
  \tileA{0}{0}{-1}{};
  \ftileA{0}{0}{0}{};
\end{tikzpicture}%
\end{center}%
\end{minipage}%
} \qquad \subfloat[Partial patch 14 (extends to partial patch 26)]{%
\begin{minipage}[b]{6cm}
\begin{center}
\begin{tikzpicture}[x=3mm,y=3mm]
  \ffkite{240}{0}{0};
  \tileA{300}{-1}{0}{};
  \tileA{240}{-1}{1}{};
  \tileA{0}{0}{-1}{};
  \ftileA{0}{0}{0}{};
\end{tikzpicture}%
\end{center}%
\end{minipage}%
} \\ \subfloat[Partial patch 15 (extends to partial patches 27--28)]{%
\begin{minipage}[b]{6cm}
\begin{center}
\begin{tikzpicture}[x=3mm,y=3mm]
  \ffkite{120}{0}{1};
  \tileA{120}{-1}{0}{};
  \tileAr{300}{-1}{1}{};
  \tileA{180}{0}{-1}{};
  \ftileA{0}{0}{0}{};
\end{tikzpicture}%
\end{center}%
\end{minipage}%
} \qquad \subfloat[Partial patch 16 (extends to partial patches 29--30)]{%
\begin{minipage}[b]{6cm}
\begin{center}
\begin{tikzpicture}[x=3mm,y=3mm]
  \ffkite{60}{1}{-1};
  \tileA{300}{-1}{1}{};
  \tileA{180}{0}{-1}{};
  \ftileA{0}{0}{0}{};
  \tileA{120}{1}{-2}{};
\end{tikzpicture}%
\end{center}%
\end{minipage}%
} \\ \subfloat[Partial patch 17 (extends to partial patches 31--33)]{%
\begin{minipage}[b]{6cm}
\begin{center}
\begin{tikzpicture}[x=3mm,y=3mm]
  \ffkite{0}{2}{-1};
  \tileA{300}{-1}{1}{};
  \tileA{180}{0}{-1}{};
  \ftileA{0}{0}{0}{};
  \tileA{300}{1}{-1}{};
\end{tikzpicture}%
\end{center}%
\end{minipage}%
} \qquad \subfloat[Partial patch 18 (extends to partial patch 34)]{%
\begin{minipage}[b]{6cm}
\begin{center}
\begin{tikzpicture}[x=3mm,y=3mm]
  \ffkite{300}{2}{-1};
  \tileA{240}{0}{-1}{};
  \ftileA{0}{0}{0}{};
  \tileAr{180}{1}{-1}{};
  \tileA{60}{2}{-2}{};
\end{tikzpicture}%
\end{center}%
\end{minipage}%
}%
\end{center}
\caption{Partial patches (part 3)}
\label{fig:partpatch:3}
\end{figure}
\begin{figure}[htp!]
\renewcommand\thesubfigure{\arabic{subfigure}}%
\ContinuedFloat
\captionsetup{margin=0pt,justification=raggedright}%
\begin{center}
\subfloat[Partial patch 19 (extends to partial patch 35)]{%
\begin{minipage}[b]{6cm}
\begin{center}
\begin{tikzpicture}[x=3mm,y=3mm]
  \ffkite{300}{2}{-1};
  \tileA{240}{0}{-1}{};
  \ftileA{0}{0}{0}{};
  \tileAr{180}{1}{-1}{};
  \tileAr{240}{2}{-1}{};
\end{tikzpicture}%
\end{center}%
\end{minipage}%
} \qquad \subfloat[Partial patch 20 (extends to partial patch 36)]{%
\begin{minipage}[b]{6cm}
\begin{center}
\begin{tikzpicture}[x=3mm,y=3mm]
  \ffkite{240}{0}{0};
  \tileA{180}{-1}{0}{};
  \tileAr{120}{0}{-1}{};
  \ftileA{0}{0}{0}{};
  \tileAr{180}{1}{-1}{};
\end{tikzpicture}%
\end{center}%
\end{minipage}%
} \\ \subfloat[Partial patch 21 (extends to partial patch 37)]{%
\begin{minipage}[b]{6cm}
\begin{center}
\begin{tikzpicture}[x=3mm,y=3mm]
  \ffkite{240}{0}{0};
  \tileAr{60}{-1}{0}{};
  \tileAr{120}{0}{-1}{};
  \ftileA{0}{0}{0}{};
  \tileAr{180}{1}{-1}{};
\end{tikzpicture}%
\end{center}%
\end{minipage}%
} \qquad \subfloat[Partial patch 22 (extends to partial patches 38--39)]{%
\begin{minipage}[b]{6cm}
\begin{center}
\begin{tikzpicture}[x=3mm,y=3mm]
  \ffkite{120}{0}{1};
  \tileA{60}{-1}{-1}{};
  \tileA{120}{-1}{0}{};
  \tileAr{300}{-1}{1}{};
  \tileA{0}{0}{-1}{};
  \ftileA{0}{0}{0}{};
\end{tikzpicture}%
\end{center}%
\end{minipage}%
} \\ \subfloat[Partial patch 23 (extends to partial patches 40--41)]{%
\begin{minipage}[b]{6cm}
\begin{center}
\begin{tikzpicture}[x=3mm,y=3mm]
  \ffkite{60}{1}{0};
  \tileA{60}{-1}{-1}{};
  \tileA{300}{-1}{1}{};
  \tileA{0}{0}{-1}{};
  \ftileA{0}{0}{0}{};
  \tileA{120}{2}{-2}{};
\end{tikzpicture}%
\end{center}%
\end{minipage}%
} \qquad \subfloat[Partial patch 24 (extends to partial patch 42)]{%
\begin{minipage}[b]{6cm}
\begin{center}
\begin{tikzpicture}[x=3mm,y=3mm]
  \ffkite{60}{1}{0};
  \tileA{60}{-1}{-1}{};
  \tileA{300}{-1}{1}{};
  \tileA{0}{0}{-1}{};
  \ftileA{0}{0}{0}{};
  \tileA{240}{2}{-1}{};
\end{tikzpicture}%
\end{center}%
\end{minipage}%
} \\ \subfloat[Partial patch 25 (extends to partial patches 43--44)]{%
\begin{minipage}[b]{6cm}
\begin{center}
\begin{tikzpicture}[x=3mm,y=3mm]
  \ffkite{120}{0}{1};
  \tileA{0}{-2}{1}{};
  \tileA{300}{-1}{0}{};
  \tileAr{300}{-1}{1}{};
  \tileA{0}{0}{-1}{};
  \ftileA{0}{0}{0}{};
\end{tikzpicture}%
\end{center}%
\end{minipage}%
} \qquad \subfloat[Partial patch 26 (extends to partial patches 45--46)]{%
\begin{minipage}[b]{6cm}
\begin{center}
\begin{tikzpicture}[x=3mm,y=3mm]
  \ffkite{60}{1}{-1};
  \tileA{300}{-1}{0}{};
  \tileA{240}{-1}{1}{};
  \tileA{0}{0}{-1}{};
  \ftileA{0}{0}{0}{};
  \tileAr{180}{0}{1}{};
\end{tikzpicture}%
\end{center}%
\end{minipage}%
}%
\end{center}
\caption{Partial patches (part 4)}
\label{fig:partpatch:4}
\end{figure}
\begin{figure}[htp!]
\renewcommand\thesubfigure{\arabic{subfigure}}%
\ContinuedFloat
\captionsetup{margin=0pt,justification=raggedright}%
\begin{center}
\subfloat[Partial patch 27 (extends to partial patch 47)]{%
\begin{minipage}[b]{6cm}
\begin{center}
\begin{tikzpicture}[x=3mm,y=3mm]
  \ffkite{180}{1}{0};
  \tileA{120}{-1}{0}{};
  \tileAr{300}{-1}{1}{};
  \tileA{180}{0}{-1}{};
  \ftileA{0}{0}{0}{};
  \tileA{120}{0}{1}{};
\end{tikzpicture}%
\end{center}%
\end{minipage}%
} \qquad \subfloat[Partial patch 28 (extends to partial patch 48)]{%
\begin{minipage}[b]{6cm}
\begin{center}
\begin{tikzpicture}[x=3mm,y=3mm]
  \ffkite{180}{1}{0};
  \tileA{120}{-1}{0}{};
  \tileAr{300}{-1}{1}{};
  \tileA{180}{0}{-1}{};
  \ftileA{0}{0}{0}{};
  \tileAr{0}{0}{1}{};
\end{tikzpicture}%
\end{center}%
\end{minipage}%
} \\ \subfloat[Partial patch 29 (extends to partial patches 49--50)]{%
\begin{minipage}[b]{6cm}
\begin{center}
\begin{tikzpicture}[x=3mm,y=3mm]
  \ffkite{60}{1}{0};
  \tileA{300}{-1}{1}{};
  \tileA{180}{0}{-1}{};
  \ftileA{0}{0}{0}{};
  \tileA{120}{1}{-2}{};
  \tileA{120}{2}{-2}{};
\end{tikzpicture}%
\end{center}%
\end{minipage}%
} \qquad \subfloat[Partial patch 30 (extends to partial patch 51)]{%
\begin{minipage}[b]{6cm}
\begin{center}
\begin{tikzpicture}[x=3mm,y=3mm]
  \ffkite{60}{1}{0};
  \tileA{300}{-1}{1}{};
  \tileA{180}{0}{-1}{};
  \ftileA{0}{0}{0}{};
  \tileA{120}{1}{-2}{};
  \tileA{240}{2}{-1}{};
\end{tikzpicture}%
\end{center}%
\end{minipage}%
} \\ \subfloat[Partial patch 31 (extends to partial patch 52)]{%
\begin{minipage}[b]{6cm}
\begin{center}
\begin{tikzpicture}[x=3mm,y=3mm]
  \ffkite{300}{2}{-1};
  \tileA{300}{-1}{1}{};
  \tileA{180}{0}{-1}{};
  \ftileA{0}{0}{0}{};
  \tileA{300}{1}{-1}{};
  \tileA{0}{2}{-1}{};
\end{tikzpicture}%
\end{center}%
\end{minipage}%
} \qquad \subfloat[Partial patch 32 (extends to partial patches 53--54)]{%
\begin{minipage}[b]{6cm}
\begin{center}
\begin{tikzpicture}[x=3mm,y=3mm]
  \ffkite{60}{1}{0};
  \tileA{300}{-1}{1}{};
  \tileA{180}{0}{-1}{};
  \ftileA{0}{0}{0}{};
  \tileA{300}{1}{-1}{};
  \tileA{300}{2}{-1}{};
\end{tikzpicture}%
\end{center}%
\end{minipage}%
} \\ \subfloat[Partial patch 33 (extends to partial patch 55)]{%
\begin{minipage}[b]{6cm}
\begin{center}
\begin{tikzpicture}[x=3mm,y=3mm]
  \ffkite{300}{2}{-1};
  \tileA{300}{-1}{1}{};
  \tileA{180}{0}{-1}{};
  \ftileA{0}{0}{0}{};
  \tileA{300}{1}{-1}{};
  \tileA{180}{3}{-2}{};
\end{tikzpicture}%
\end{center}%
\end{minipage}%
} \qquad \subfloat[Partial patch 34 (extends to partial patches 56--57)]{%
\begin{minipage}[b]{6cm}
\begin{center}
\begin{tikzpicture}[x=3mm,y=3mm]
  \ffkite{180}{1}{0};
  \tileA{240}{0}{-1}{};
  \ftileA{0}{0}{0}{};
  \tileAr{180}{1}{-1}{};
  \tileA{60}{2}{-2}{};
  \tileA{120}{2}{-1}{};
\end{tikzpicture}%
\end{center}%
\end{minipage}%
}%
\end{center}
\caption{Partial patches (part 5)}
\label{fig:partpatch:5}
\end{figure}
\begin{figure}[htp!]
\renewcommand\thesubfigure{\arabic{subfigure}}%
\ContinuedFloat
\captionsetup{margin=0pt,justification=raggedright}%
\begin{center}
\subfloat[Partial patch 35 (extends to partial patches 58--60)]{%
\begin{minipage}[b]{6cm}
\begin{center}
\begin{tikzpicture}[x=3mm,y=3mm]
  \ffkite{120}{-1}{0};
  \tileA{240}{0}{-1}{};
  \ftileA{0}{0}{0}{};
  \tileAr{180}{1}{-1}{};
  \tileAr{300}{1}{0}{};
  \tileAr{240}{2}{-1}{};
\end{tikzpicture}%
\end{center}%
\end{minipage}%
} \qquad \subfloat[Partial patch 36 (extends to partial patches 61--62)]{%
\begin{minipage}[b]{6cm}
\begin{center}
\begin{tikzpicture}[x=3mm,y=3mm]
  \ffkite{120}{0}{1};
  \tileA{180}{-1}{0}{};
  \tileAr{300}{-1}{1}{};
  \tileAr{120}{0}{-1}{};
  \ftileA{0}{0}{0}{};
  \tileAr{180}{1}{-1}{};
\end{tikzpicture}%
\end{center}%
\end{minipage}%
} \\ \subfloat[Partial patch 37 (extends to partial patches 63--64)]{%
\begin{minipage}[b]{6cm}
\begin{center}
\begin{tikzpicture}[x=3mm,y=3mm]
  \ffkite{180}{1}{0};
  \tileAr{60}{-1}{0}{};
  \tileAr{120}{0}{-1}{};
  \ftileA{0}{0}{0}{};
  \tileAr{180}{0}{1}{};
  \tileAr{180}{1}{-1}{};
\end{tikzpicture}%
\end{center}%
\end{minipage}%
} \qquad \subfloat[Partial patch 38 (extends to partial patch 65)]{%
\begin{minipage}[b]{6cm}
\begin{center}
\begin{tikzpicture}[x=3mm,y=3mm]
  \ffkite{180}{1}{0};
  \tileA{60}{-1}{-1}{};
  \tileA{120}{-1}{0}{};
  \tileAr{300}{-1}{1}{};
  \tileA{0}{0}{-1}{};
  \ftileA{0}{0}{0}{};
  \tileA{120}{0}{1}{};
\end{tikzpicture}%
\end{center}%
\end{minipage}%
} \\ \subfloat[Partial patch 39 (extends to partial patch 66)]{%
\begin{minipage}[b]{6cm}
\begin{center}
\begin{tikzpicture}[x=3mm,y=3mm]
  \ffkite{180}{1}{0};
  \tileA{60}{-1}{-1}{};
  \tileA{120}{-1}{0}{};
  \tileAr{300}{-1}{1}{};
  \tileA{0}{0}{-1}{};
  \ftileA{0}{0}{0}{};
  \tileAr{0}{0}{1}{};
\end{tikzpicture}%
\end{center}%
\end{minipage}%
} \qquad \subfloat[Partial patch 40 (extends to partial patches 67--68)]{%
\begin{minipage}[b]{6cm}
\begin{center}
\begin{tikzpicture}[x=3mm,y=3mm]
  \ffkite{60}{0}{1};
  \tileA{60}{-1}{-1}{};
  \tileA{300}{-1}{1}{};
  \tileA{0}{0}{-1}{};
  \ftileA{0}{0}{0}{};
  \tileA{60}{1}{0}{};
  \tileA{120}{2}{-2}{};
\end{tikzpicture}%
\end{center}%
\end{minipage}%
} \\ \subfloat[Partial patch 41 (extends to partial patch 69, $1$-patch 1)]{%
\begin{minipage}[b]{6cm}
\begin{center}
\begin{tikzpicture}[x=3mm,y=3mm]
  \ffkite{180}{1}{0};
  \tileA{60}{-1}{-1}{};
  \tileA{300}{-1}{1}{};
  \tileA{0}{0}{-1}{};
  \ftileA{0}{0}{0}{};
  \tileA{120}{2}{-2}{};
  \tileA{240}{2}{0}{};
\end{tikzpicture}%
\end{center}%
\end{minipage}%
} \qquad \subfloat[Partial patch 42 (extends to partial patches 70--71)]{%
\begin{minipage}[b]{6cm}
\begin{center}
\begin{tikzpicture}[x=3mm,y=3mm]
  \ffkite{60}{0}{1};
  \tileA{60}{-1}{-1}{};
  \tileA{300}{-1}{1}{};
  \tileA{0}{0}{-1}{};
  \ftileA{0}{0}{0}{};
  \tileA{60}{1}{0}{};
  \tileA{240}{2}{-1}{};
\end{tikzpicture}%
\end{center}%
\end{minipage}%
}%
\end{center}
\caption{Partial patches (part 6)}
\label{fig:partpatch:6}
\end{figure}
\begin{figure}[htp!]
\renewcommand\thesubfigure{\arabic{subfigure}}%
\ContinuedFloat
\captionsetup{margin=0pt,justification=raggedright}%
\begin{center}
\subfloat[Partial patch 43 (extends to partial patch 72)]{%
\begin{minipage}[b]{6cm}
\begin{center}
\begin{tikzpicture}[x=3mm,y=3mm]
  \ffkite{180}{1}{0};
  \tileA{0}{-2}{1}{};
  \tileA{300}{-1}{0}{};
  \tileAr{300}{-1}{1}{};
  \tileA{0}{0}{-1}{};
  \ftileA{0}{0}{0}{};
  \tileA{120}{0}{1}{};
\end{tikzpicture}%
\end{center}%
\end{minipage}%
} \qquad \subfloat[Partial patch 44 (extends to partial patch 73)]{%
\begin{minipage}[b]{6cm}
\begin{center}
\begin{tikzpicture}[x=3mm,y=3mm]
  \ffkite{180}{1}{0};
  \tileA{0}{-2}{1}{};
  \tileA{300}{-1}{0}{};
  \tileAr{300}{-1}{1}{};
  \tileA{0}{0}{-1}{};
  \ftileA{0}{0}{0}{};
  \tileAr{0}{0}{1}{};
\end{tikzpicture}%
\end{center}%
\end{minipage}%
} \\ \subfloat[Partial patch 45 (extends to $1$-patch 2)]{%
\begin{minipage}[b]{6cm}
\begin{center}
\begin{tikzpicture}[x=3mm,y=3mm]
  \ffkite{180}{1}{0};
  \tileA{300}{-1}{0}{};
  \tileA{240}{-1}{1}{};
  \tileA{0}{0}{-1}{};
  \ftileA{0}{0}{0}{};
  \tileAr{180}{0}{1}{};
  \tileA{120}{2}{-2}{};
\end{tikzpicture}%
\end{center}%
\end{minipage}%
} \qquad \subfloat[Partial patch 46 (extends to $1$-patch 3)]{%
\begin{minipage}[b]{6cm}
\begin{center}
\begin{tikzpicture}[x=3mm,y=3mm]
  \ffkite{60}{1}{0};
  \tileA{300}{-1}{0}{};
  \tileA{240}{-1}{1}{};
  \tileA{0}{0}{-1}{};
  \ftileA{0}{0}{0}{};
  \tileAr{180}{0}{1}{};
  \tileA{240}{2}{-1}{};
\end{tikzpicture}%
\end{center}%
\end{minipage}%
} \\ \subfloat[Partial patch 47 (extends to partial patches 74--75)]{%
\begin{minipage}[b]{6cm}
\begin{center}
\begin{tikzpicture}[x=3mm,y=3mm]
  \ffkite{0}{1}{-1};
  \tileA{120}{-1}{0}{};
  \tileAr{300}{-1}{1}{};
  \tileA{180}{0}{-1}{};
  \ftileA{0}{0}{0}{};
  \tileA{120}{0}{1}{};
  \tileA{60}{1}{0}{};
\end{tikzpicture}%
\end{center}%
\end{minipage}%
} \qquad \subfloat[Partial patch 48 (extends to partial patch 76)]{%
\begin{minipage}[b]{6cm}
\begin{center}
\begin{tikzpicture}[x=3mm,y=3mm]
  \ffkite{60}{1}{-1};
  \tileA{120}{-1}{0}{};
  \tileAr{300}{-1}{1}{};
  \tileA{180}{0}{-1}{};
  \ftileA{0}{0}{0}{};
  \tileAr{0}{0}{1}{};
  \tileAr{300}{1}{0}{};
\end{tikzpicture}%
\end{center}%
\end{minipage}%
} \\ \subfloat[Partial patch 49 (extends to partial patches 77--78)]{%
\begin{minipage}[b]{6cm}
\begin{center}
\begin{tikzpicture}[x=3mm,y=3mm]
  \ffkite{60}{0}{1};
  \tileA{300}{-1}{1}{};
  \tileA{180}{0}{-1}{};
  \ftileA{0}{0}{0}{};
  \tileA{120}{1}{-2}{};
  \tileA{60}{1}{0}{};
  \tileA{120}{2}{-2}{};
\end{tikzpicture}%
\end{center}%
\end{minipage}%
} \qquad \subfloat[Partial patch 50 (extends to partial patch 79, $1$-patch 4)]{%
\begin{minipage}[b]{6cm}
\begin{center}
\begin{tikzpicture}[x=3mm,y=3mm]
  \ffkite{180}{1}{0};
  \tileA{300}{-1}{1}{};
  \tileA{180}{0}{-1}{};
  \ftileA{0}{0}{0}{};
  \tileA{120}{1}{-2}{};
  \tileA{120}{2}{-2}{};
  \tileA{240}{2}{0}{};
\end{tikzpicture}%
\end{center}%
\end{minipage}%
}%
\end{center}
\caption{Partial patches (part 7)}
\label{fig:partpatch:7}
\end{figure}
\begin{figure}[htp!]
\renewcommand\thesubfigure{\arabic{subfigure}}%
\ContinuedFloat
\captionsetup{margin=0pt,justification=raggedright}%
\begin{center}
\subfloat[Partial patch 51 (extends to partial patches 80--81)]{%
\begin{minipage}[b]{6cm}
\begin{center}
\begin{tikzpicture}[x=3mm,y=3mm]
  \ffkite{60}{0}{1};
  \tileA{300}{-1}{1}{};
  \tileA{180}{0}{-1}{};
  \ftileA{0}{0}{0}{};
  \tileA{120}{1}{-2}{};
  \tileA{60}{1}{0}{};
  \tileA{240}{2}{-1}{};
\end{tikzpicture}%
\end{center}%
\end{minipage}%
} \qquad \subfloat[Partial patch 52 (extends to partial patch 82)]{%
\begin{minipage}[b]{6cm}
\begin{center}
\begin{tikzpicture}[x=3mm,y=3mm]
  \ffkite{60}{0}{1};
  \tileA{300}{-1}{1}{};
  \tileA{180}{0}{-1}{};
  \ftileA{0}{0}{0}{};
  \tileA{300}{1}{-1}{};
  \tileAr{300}{1}{0}{};
  \tileA{0}{2}{-1}{};
\end{tikzpicture}%
\end{center}%
\end{minipage}%
} \\ \subfloat[Partial patch 53 (extends to partial patches 83--84)]{%
\begin{minipage}[b]{6cm}
\begin{center}
\begin{tikzpicture}[x=3mm,y=3mm]
  \ffkite{60}{0}{1};
  \tileA{300}{-1}{1}{};
  \tileA{180}{0}{-1}{};
  \ftileA{0}{0}{0}{};
  \tileA{300}{1}{-1}{};
  \tileA{60}{1}{0}{};
  \tileA{300}{2}{-1}{};
\end{tikzpicture}%
\end{center}%
\end{minipage}%
} \qquad \subfloat[Partial patch 54 (extends to partial patch 85, $1$-patch 5)]{%
\begin{minipage}[b]{6cm}
\begin{center}
\begin{tikzpicture}[x=3mm,y=3mm]
  \ffkite{180}{1}{0};
  \tileA{300}{-1}{1}{};
  \tileA{180}{0}{-1}{};
  \ftileA{0}{0}{0}{};
  \tileA{300}{1}{-1}{};
  \tileA{300}{2}{-1}{};
  \tileA{240}{2}{0}{};
\end{tikzpicture}%
\end{center}%
\end{minipage}%
} \\ \subfloat[Partial patch 55 (extends to partial patch 86, $1$-patch 6)]{%
\begin{minipage}[b]{6cm}
\begin{center}
\begin{tikzpicture}[x=3mm,y=3mm]
  \ffkite{180}{1}{0};
  \tileA{300}{-1}{1}{};
  \tileA{180}{0}{-1}{};
  \ftileA{0}{0}{0}{};
  \tileA{300}{1}{-1}{};
  \tileA{120}{2}{-1}{};
  \tileA{180}{3}{-2}{};
\end{tikzpicture}%
\end{center}%
\end{minipage}%
} \qquad \subfloat[Partial patch 56 (extends to $1$-patches 7--8)]{%
\begin{minipage}[b]{6cm}
\begin{center}
\begin{tikzpicture}[x=3mm,y=3mm]
  \ffkite{240}{0}{0};
  \tileA{240}{0}{-1}{};
  \ftileA{0}{0}{0}{};
  \tileA{0}{0}{1}{};
  \tileAr{180}{1}{-1}{};
  \tileA{60}{2}{-2}{};
  \tileA{120}{2}{-1}{};
\end{tikzpicture}%
\end{center}%
\end{minipage}%
} \\ \subfloat[Partial patch 57 (extends to partial patches 87--88)]{%
\begin{minipage}[b]{6cm}
\begin{center}
\begin{tikzpicture}[x=3mm,y=3mm]
  \ffkite{240}{0}{0};
  \tileA{240}{0}{-1}{};
  \ftileA{0}{0}{0}{};
  \tileAr{180}{1}{-1}{};
  \tileA{240}{1}{1}{};
  \tileA{60}{2}{-2}{};
  \tileA{120}{2}{-1}{};
\end{tikzpicture}%
\end{center}%
\end{minipage}%
} \qquad \subfloat[Partial patch 58 (no extensions)]{%
\begin{minipage}[b]{6cm}
\begin{center}
\begin{tikzpicture}[x=3mm,y=3mm]
  \ffkite{60}{0}{1};
  \tileA{60}{-1}{0}{};
  \tileA{240}{0}{-1}{};
  \ftileA{0}{0}{0}{};
  \tileAr{180}{1}{-1}{};
  \tileAr{300}{1}{0}{};
  \tileAr{240}{2}{-1}{};
\end{tikzpicture}%
\end{center}%
\end{minipage}%
}%
\end{center}
\caption{Partial patches (part 8)}
\label{fig:partpatch:8}
\end{figure}
\begin{figure}[htp!]
\renewcommand\thesubfigure{\arabic{subfigure}}%
\ContinuedFloat
\captionsetup{margin=0pt,justification=raggedright}%
\begin{center}
\subfloat[Partial patch 59 (extends to partial patch 89)]{%
\begin{minipage}[b]{6cm}
\begin{center}
\begin{tikzpicture}[x=3mm,y=3mm]
  \ffkite{240}{0}{0};
  \tileA{120}{-1}{0}{};
  \tileA{240}{0}{-1}{};
  \ftileA{0}{0}{0}{};
  \tileAr{180}{1}{-1}{};
  \tileAr{300}{1}{0}{};
  \tileAr{240}{2}{-1}{};
\end{tikzpicture}%
\end{center}%
\end{minipage}%
} \qquad \subfloat[Partial patch 60 (extends to partial patch 90)]{%
\begin{minipage}[b]{6cm}
\begin{center}
\begin{tikzpicture}[x=3mm,y=3mm]
  \ffkite{60}{0}{1};
  \tileA{300}{-1}{1}{};
  \tileA{240}{0}{-1}{};
  \ftileA{0}{0}{0}{};
  \tileAr{180}{1}{-1}{};
  \tileAr{300}{1}{0}{};
  \tileAr{240}{2}{-1}{};
\end{tikzpicture}%
\end{center}%
\end{minipage}%
} \\ \subfloat[Partial patch 61 (extends to partial patch 91)]{%
\begin{minipage}[b]{6cm}
\begin{center}
\begin{tikzpicture}[x=3mm,y=3mm]
  \ffkite{180}{1}{0};
  \tileA{180}{-1}{0}{};
  \tileAr{300}{-1}{1}{};
  \tileAr{120}{0}{-1}{};
  \ftileA{0}{0}{0}{};
  \tileA{120}{0}{1}{};
  \tileAr{180}{1}{-1}{};
\end{tikzpicture}%
\end{center}%
\end{minipage}%
} \qquad \subfloat[Partial patch 62 (extends to partial patch 92)]{%
\begin{minipage}[b]{6cm}
\begin{center}
\begin{tikzpicture}[x=3mm,y=3mm]
  \ffkite{180}{1}{0};
  \tileA{180}{-1}{0}{};
  \tileAr{300}{-1}{1}{};
  \tileAr{120}{0}{-1}{};
  \ftileA{0}{0}{0}{};
  \tileAr{0}{0}{1}{};
  \tileAr{180}{1}{-1}{};
\end{tikzpicture}%
\end{center}%
\end{minipage}%
} \\ \subfloat[Partial patch 63 (no extensions)]{%
\begin{minipage}[b]{6cm}
\begin{center}
\begin{tikzpicture}[x=3mm,y=3mm]
  \ffkite{300}{2}{-1};
  \tileAr{60}{-1}{0}{};
  \tileAr{120}{0}{-1}{};
  \ftileA{0}{0}{0}{};
  \tileAr{180}{0}{1}{};
  \tileAr{180}{1}{-1}{};
  \tileA{60}{1}{0}{};
\end{tikzpicture}%
\end{center}%
\end{minipage}%
} \qquad \subfloat[Partial patch 64 (extends to $1$-patch 9)]{%
\begin{minipage}[b]{6cm}
\begin{center}
\begin{tikzpicture}[x=3mm,y=3mm]
  \ffkite{0}{2}{-1};
  \tileAr{60}{-1}{0}{};
  \tileAr{120}{0}{-1}{};
  \ftileA{0}{0}{0}{};
  \tileAr{180}{0}{1}{};
  \tileAr{180}{1}{-1}{};
  \tileAr{300}{1}{0}{};
\end{tikzpicture}%
\end{center}%
\end{minipage}%
} \\ \subfloat[Partial patch 65 (extends to $1$-patches 10--11)]{%
\begin{minipage}[b]{6cm}
\begin{center}
\begin{tikzpicture}[x=3mm,y=3mm]
  \ffkite{60}{1}{-1};
  \tileA{60}{-1}{-1}{};
  \tileA{120}{-1}{0}{};
  \tileAr{300}{-1}{1}{};
  \tileA{0}{0}{-1}{};
  \ftileA{0}{0}{0}{};
  \tileA{120}{0}{1}{};
  \tileA{60}{1}{0}{};
\end{tikzpicture}%
\end{center}%
\end{minipage}%
} \qquad \subfloat[Partial patch 66 (no extensions)]{%
\begin{minipage}[b]{6cm}
\begin{center}
\begin{tikzpicture}[x=3mm,y=3mm]
  \ffkite{60}{1}{-1};
  \tileA{60}{-1}{-1}{};
  \tileA{120}{-1}{0}{};
  \tileAr{300}{-1}{1}{};
  \tileA{0}{0}{-1}{};
  \ftileA{0}{0}{0}{};
  \tileAr{0}{0}{1}{};
  \tileAr{300}{1}{0}{};
\end{tikzpicture}%
\end{center}%
\end{minipage}%
}%
\end{center}
\caption{Partial patches (part 9)}
\label{fig:partpatch:9}
\end{figure}
\begin{figure}[htp!]
\renewcommand\thesubfigure{\arabic{subfigure}}%
\ContinuedFloat
\captionsetup{margin=0pt,justification=raggedright}%
\begin{center}
\subfloat[Partial patch 67 (extends to $1$-patch 12)]{%
\begin{minipage}[b]{6cm}
\begin{center}
\begin{tikzpicture}[x=3mm,y=3mm]
  \ffkite{0}{0}{1};
  \tileA{60}{-1}{-1}{};
  \tileA{300}{-1}{1}{};
  \tileA{0}{0}{-1}{};
  \ftileA{0}{0}{0}{};
  \tileA{60}{0}{1}{};
  \tileA{60}{1}{0}{};
  \tileA{120}{2}{-2}{};
\end{tikzpicture}%
\end{center}%
\end{minipage}%
} \qquad \subfloat[Partial patch 68 (extends to $1$-patch 13)]{%
\begin{minipage}[b]{6cm}
\begin{center}
\begin{tikzpicture}[x=3mm,y=3mm]
  \ffkite{120}{0}{1};
  \tileA{60}{-1}{-1}{};
  \tileA{300}{-1}{1}{};
  \tileA{0}{0}{-1}{};
  \ftileA{0}{0}{0}{};
  \tileAr{120}{0}{1}{};
  \tileA{60}{1}{0}{};
  \tileA{120}{2}{-2}{};
\end{tikzpicture}%
\end{center}%
\end{minipage}%
} \\ \subfloat[Partial patch 69 (extends to $1$-patch 14)]{%
\begin{minipage}[b]{6cm}
\begin{center}
\begin{tikzpicture}[x=3mm,y=3mm]
  \ffkite{0}{0}{1};
  \tileA{60}{-1}{-1}{};
  \tileA{300}{-1}{1}{};
  \tileA{0}{0}{-1}{};
  \ftileA{0}{0}{0}{};
  \tileA{240}{1}{1}{};
  \tileA{120}{2}{-2}{};
  \tileA{240}{2}{0}{};
\end{tikzpicture}%
\end{center}%
\end{minipage}%
} \qquad \subfloat[Partial patch 70 (extends to $1$-patch 15)]{%
\begin{minipage}[b]{6cm}
\begin{center}
\begin{tikzpicture}[x=3mm,y=3mm]
  \ffkite{0}{0}{1};
  \tileA{60}{-1}{-1}{};
  \tileA{300}{-1}{1}{};
  \tileA{0}{0}{-1}{};
  \ftileA{0}{0}{0}{};
  \tileA{60}{0}{1}{};
  \tileA{60}{1}{0}{};
  \tileA{240}{2}{-1}{};
\end{tikzpicture}%
\end{center}%
\end{minipage}%
} \\ \subfloat[Partial patch 71 (extends to $1$-patch 16)]{%
\begin{minipage}[b]{6cm}
\begin{center}
\begin{tikzpicture}[x=3mm,y=3mm]
  \ffkite{120}{0}{1};
  \tileA{60}{-1}{-1}{};
  \tileA{300}{-1}{1}{};
  \tileA{0}{0}{-1}{};
  \ftileA{0}{0}{0}{};
  \tileAr{120}{0}{1}{};
  \tileA{60}{1}{0}{};
  \tileA{240}{2}{-1}{};
\end{tikzpicture}%
\end{center}%
\end{minipage}%
} \qquad \subfloat[Partial patch 72 (extends to $1$-patches 17--18)]{%
\begin{minipage}[b]{6cm}
\begin{center}
\begin{tikzpicture}[x=3mm,y=3mm]
  \ffkite{60}{1}{-1};
  \tileA{0}{-2}{1}{};
  \tileA{300}{-1}{0}{};
  \tileAr{300}{-1}{1}{};
  \tileA{0}{0}{-1}{};
  \ftileA{0}{0}{0}{};
  \tileA{120}{0}{1}{};
  \tileA{60}{1}{0}{};
\end{tikzpicture}%
\end{center}%
\end{minipage}%
}%
\end{center}
\caption{Partial patches (part 10)}
\label{fig:partpatch:10}
\end{figure}
\begin{figure}[htp!]
\renewcommand\thesubfigure{\arabic{subfigure}}%
\ContinuedFloat
\captionsetup{margin=0pt,justification=raggedright}%
\begin{center}
\subfloat[Partial patch 73 (no extensions)]{%
\begin{minipage}[b]{6cm}
\begin{center}
\begin{tikzpicture}[x=3mm,y=3mm]
  \ffkite{60}{1}{-1};
  \tileA{0}{-2}{1}{};
  \tileA{300}{-1}{0}{};
  \tileAr{300}{-1}{1}{};
  \tileA{0}{0}{-1}{};
  \ftileA{0}{0}{0}{};
  \tileAr{0}{0}{1}{};
  \tileAr{300}{1}{0}{};
\end{tikzpicture}%
\end{center}%
\end{minipage}%
} \qquad \subfloat[Partial patch 74 (extends to $1$-patches 19--20)]{%
\begin{minipage}[b]{6cm}
\begin{center}
\begin{tikzpicture}[x=3mm,y=3mm]
  \ffkite{60}{1}{-1};
  \tileA{120}{-1}{0}{};
  \tileAr{300}{-1}{1}{};
  \tileA{180}{0}{-1}{};
  \ftileA{0}{0}{0}{};
  \tileA{120}{0}{1}{};
  \tileA{120}{1}{-2}{};
  \tileA{60}{1}{0}{};
\end{tikzpicture}%
\end{center}%
\end{minipage}%
} \\ \subfloat[Partial patch 75 (extends to $1$-patch 21)]{%
\begin{minipage}[b]{6cm}
\begin{center}
\begin{tikzpicture}[x=3mm,y=3mm]
  \ffkite{300}{2}{-1};
  \tileA{120}{-1}{0}{};
  \tileAr{300}{-1}{1}{};
  \tileA{180}{0}{-1}{};
  \ftileA{0}{0}{0}{};
  \tileA{120}{0}{1}{};
  \tileA{300}{1}{-1}{};
  \tileA{60}{1}{0}{};
\end{tikzpicture}%
\end{center}%
\end{minipage}%
} \qquad \subfloat[Partial patch 76 (extends to $1$-patch 22)]{%
\begin{minipage}[b]{6cm}
\begin{center}
\begin{tikzpicture}[x=3mm,y=3mm]
  \ffkite{0}{2}{-1};
  \tileA{120}{-1}{0}{};
  \tileAr{300}{-1}{1}{};
  \tileA{180}{0}{-1}{};
  \ftileA{0}{0}{0}{};
  \tileAr{0}{0}{1}{};
  \tileA{300}{1}{-1}{};
  \tileAr{300}{1}{0}{};
\end{tikzpicture}%
\end{center}%
\end{minipage}%
} \\ \subfloat[Partial patch 77 (extends to $1$-patch 23)]{%
\begin{minipage}[b]{6cm}
\begin{center}
\begin{tikzpicture}[x=3mm,y=3mm]
  \ffkite{0}{0}{1};
  \tileA{300}{-1}{1}{};
  \tileA{180}{0}{-1}{};
  \ftileA{0}{0}{0}{};
  \tileA{60}{0}{1}{};
  \tileA{120}{1}{-2}{};
  \tileA{60}{1}{0}{};
  \tileA{120}{2}{-2}{};
\end{tikzpicture}%
\end{center}%
\end{minipage}%
} \qquad \subfloat[Partial patch 78 (extends to $1$-patch 24)]{%
\begin{minipage}[b]{6cm}
\begin{center}
\begin{tikzpicture}[x=3mm,y=3mm]
  \ffkite{120}{0}{1};
  \tileA{300}{-1}{1}{};
  \tileA{180}{0}{-1}{};
  \ftileA{0}{0}{0}{};
  \tileAr{120}{0}{1}{};
  \tileA{120}{1}{-2}{};
  \tileA{60}{1}{0}{};
  \tileA{120}{2}{-2}{};
\end{tikzpicture}%
\end{center}%
\end{minipage}%
}%
\end{center}
\caption{Partial patches (part 11)}
\label{fig:partpatch:11}
\end{figure}
\begin{figure}[htp!]
\renewcommand\thesubfigure{\arabic{subfigure}}%
\ContinuedFloat
\captionsetup{margin=0pt,justification=raggedright}%
\begin{center}
\subfloat[Partial patch 79 (extends to $1$-patch 25)]{%
\begin{minipage}[b]{6cm}
\begin{center}
\begin{tikzpicture}[x=3mm,y=3mm]
  \ffkite{0}{0}{1};
  \tileA{300}{-1}{1}{};
  \tileA{180}{0}{-1}{};
  \ftileA{0}{0}{0}{};
  \tileA{120}{1}{-2}{};
  \tileA{240}{1}{1}{};
  \tileA{120}{2}{-2}{};
  \tileA{240}{2}{0}{};
\end{tikzpicture}%
\end{center}%
\end{minipage}%
} \qquad \subfloat[Partial patch 80 (extends to $1$-patch 26)]{%
\begin{minipage}[b]{6cm}
\begin{center}
\begin{tikzpicture}[x=3mm,y=3mm]
  \ffkite{0}{0}{1};
  \tileA{300}{-1}{1}{};
  \tileA{180}{0}{-1}{};
  \ftileA{0}{0}{0}{};
  \tileA{60}{0}{1}{};
  \tileA{120}{1}{-2}{};
  \tileA{60}{1}{0}{};
  \tileA{240}{2}{-1}{};
\end{tikzpicture}%
\end{center}%
\end{minipage}%
} \\ \subfloat[Partial patch 81 (extends to $1$-patch 27)]{%
\begin{minipage}[b]{6cm}
\begin{center}
\begin{tikzpicture}[x=3mm,y=3mm]
  \ffkite{120}{0}{1};
  \tileA{300}{-1}{1}{};
  \tileA{180}{0}{-1}{};
  \ftileA{0}{0}{0}{};
  \tileAr{120}{0}{1}{};
  \tileA{120}{1}{-2}{};
  \tileA{60}{1}{0}{};
  \tileA{240}{2}{-1}{};
\end{tikzpicture}%
\end{center}%
\end{minipage}%
} \qquad \subfloat[Partial patch 82 (extends to $1$-patch 28)]{%
\begin{minipage}[b]{6cm}
\begin{center}
\begin{tikzpicture}[x=3mm,y=3mm]
  \ffkite{120}{0}{1};
  \tileA{300}{-1}{1}{};
  \tileA{180}{0}{-1}{};
  \ftileA{0}{0}{0}{};
  \tileAr{120}{0}{1}{};
  \tileA{300}{1}{-1}{};
  \tileAr{300}{1}{0}{};
  \tileA{0}{2}{-1}{};
\end{tikzpicture}%
\end{center}%
\end{minipage}%
} \\ \subfloat[Partial patch 83 (extends to $1$-patch 29)]{%
\begin{minipage}[b]{6cm}
\begin{center}
\begin{tikzpicture}[x=3mm,y=3mm]
  \ffkite{0}{0}{1};
  \tileA{300}{-1}{1}{};
  \tileA{180}{0}{-1}{};
  \ftileA{0}{0}{0}{};
  \tileA{60}{0}{1}{};
  \tileA{300}{1}{-1}{};
  \tileA{60}{1}{0}{};
  \tileA{300}{2}{-1}{};
\end{tikzpicture}%
\end{center}%
\end{minipage}%
} \qquad \subfloat[Partial patch 84 (extends to $1$-patch 30)]{%
\begin{minipage}[b]{6cm}
\begin{center}
\begin{tikzpicture}[x=3mm,y=3mm]
  \ffkite{120}{0}{1};
  \tileA{300}{-1}{1}{};
  \tileA{180}{0}{-1}{};
  \ftileA{0}{0}{0}{};
  \tileAr{120}{0}{1}{};
  \tileA{300}{1}{-1}{};
  \tileA{60}{1}{0}{};
  \tileA{300}{2}{-1}{};
\end{tikzpicture}%
\end{center}%
\end{minipage}%
} \\ \subfloat[Partial patch 85 (extends to $1$-patch 31)]{%
\begin{minipage}[b]{6cm}
\begin{center}
\begin{tikzpicture}[x=3mm,y=3mm]
  \ffkite{0}{0}{1};
  \tileA{300}{-1}{1}{};
  \tileA{180}{0}{-1}{};
  \ftileA{0}{0}{0}{};
  \tileA{300}{1}{-1}{};
  \tileA{240}{1}{1}{};
  \tileA{300}{2}{-1}{};
  \tileA{240}{2}{0}{};
\end{tikzpicture}%
\end{center}%
\end{minipage}%
} \qquad \subfloat[Partial patch 86 (extends to $1$-patch 32)]{%
\begin{minipage}[b]{6cm}
\begin{center}
\begin{tikzpicture}[x=3mm,y=3mm]
  \ffkite{0}{0}{1};
  \tileA{300}{-1}{1}{};
  \tileA{180}{0}{-1}{};
  \ftileA{0}{0}{0}{};
  \tileA{300}{1}{-1}{};
  \tileA{240}{1}{1}{};
  \tileA{120}{2}{-1}{};
  \tileA{180}{3}{-2}{};
\end{tikzpicture}%
\end{center}%
\end{minipage}%
}%
\end{center}
\caption{Partial patches (part 12)}
\label{fig:partpatch:12}
\end{figure}
\begin{figure}[htp!]
\renewcommand\thesubfigure{\arabic{subfigure}}%
\ContinuedFloat
\captionsetup{margin=0pt,justification=raggedright}%
\begin{center}
\subfloat[Partial patch 87 (extends to $1$-patch 33)]{%
\begin{minipage}[b]{6cm}
\begin{center}
\begin{tikzpicture}[x=3mm,y=3mm]
  \ffkite{0}{0}{1};
  \tileA{60}{-1}{0}{};
  \tileA{240}{0}{-1}{};
  \ftileA{0}{0}{0}{};
  \tileAr{180}{1}{-1}{};
  \tileA{240}{1}{1}{};
  \tileA{60}{2}{-2}{};
  \tileA{120}{2}{-1}{};
\end{tikzpicture}%
\end{center}%
\end{minipage}%
} \qquad \subfloat[Partial patch 88 (extends to $1$-patch 34)]{%
\begin{minipage}[b]{6cm}
\begin{center}
\begin{tikzpicture}[x=3mm,y=3mm]
  \ffkite{0}{0}{1};
  \tileA{300}{-1}{1}{};
  \tileA{240}{0}{-1}{};
  \ftileA{0}{0}{0}{};
  \tileAr{180}{1}{-1}{};
  \tileA{240}{1}{1}{};
  \tileA{60}{2}{-2}{};
  \tileA{120}{2}{-1}{};
\end{tikzpicture}%
\end{center}%
\end{minipage}%
} \\ \subfloat[Partial patch 89 (extends to $1$-patch 35)]{%
\begin{minipage}[b]{6cm}
\begin{center}
\begin{tikzpicture}[x=3mm,y=3mm]
  \ffkite{120}{0}{1};
  \tileA{120}{-1}{0}{};
  \tileAr{300}{-1}{1}{};
  \tileA{240}{0}{-1}{};
  \ftileA{0}{0}{0}{};
  \tileAr{180}{1}{-1}{};
  \tileAr{300}{1}{0}{};
  \tileAr{240}{2}{-1}{};
\end{tikzpicture}%
\end{center}%
\end{minipage}%
} \qquad \subfloat[Partial patch 90 (extends to $1$-patch 36)]{%
\begin{minipage}[b]{6cm}
\begin{center}
\begin{tikzpicture}[x=3mm,y=3mm]
  \ffkite{120}{0}{1};
  \tileA{300}{-1}{1}{};
  \tileA{240}{0}{-1}{};
  \ftileA{0}{0}{0}{};
  \tileAr{120}{0}{1}{};
  \tileAr{180}{1}{-1}{};
  \tileAr{300}{1}{0}{};
  \tileAr{240}{2}{-1}{};
\end{tikzpicture}%
\end{center}%
\end{minipage}%
} \\ \subfloat[Partial patch 91 (no extensions)]{%
\begin{minipage}[b]{6cm}
\begin{center}
\begin{tikzpicture}[x=3mm,y=3mm]
  \ffkite{300}{2}{-1};
  \tileA{180}{-1}{0}{};
  \tileAr{300}{-1}{1}{};
  \tileAr{120}{0}{-1}{};
  \ftileA{0}{0}{0}{};
  \tileA{120}{0}{1}{};
  \tileAr{180}{1}{-1}{};
  \tileA{60}{1}{0}{};
\end{tikzpicture}%
\end{center}%
\end{minipage}%
} \qquad \subfloat[Partial patch 92 (extends to $1$-patch 37)]{%
\begin{minipage}[b]{6cm}
\begin{center}
\begin{tikzpicture}[x=3mm,y=3mm]
  \ffkite{0}{2}{-1};
  \tileA{180}{-1}{0}{};
  \tileAr{300}{-1}{1}{};
  \tileAr{120}{0}{-1}{};
  \ftileA{0}{0}{0}{};
  \tileAr{0}{0}{1}{};
  \tileAr{180}{1}{-1}{};
  \tileAr{300}{1}{0}{};
\end{tikzpicture}%
\end{center}%
\end{minipage}%
}%
\end{center}
\caption{Partial patches (part 13)}
\label{fig:partpatch:13}
\end{figure}
\begin{figure}[htp!]
\renewcommand\thesubfigure{\arabic{subfigure}}%
\captionsetup{margin=0pt,justification=raggedright}%
\begin{center}
\subfloat[$1$-patch 1 (central tile class $F_2$)]{%
\begin{minipage}[b]{6cm}
\begin{center}
\begin{tikzpicture}[x=3mm,y=3mm]
  \tileA{60}{-1}{-1}{};
  \tileA{300}{-1}{1}{};
  \tileA{0}{0}{-1}{};
  \tileA{0}{0}{0}{};
  \tileA{0}{0}{1}{};
  \tileA{120}{2}{-2}{};
  \tileA{240}{2}{0}{};
\end{tikzpicture}%
\end{center}%
\end{minipage}%
} \qquad \subfloat[$1$-patch 2 (central tile class $H_3$)]{%
\begin{minipage}[b]{6cm}
\begin{center}
\begin{tikzpicture}[x=3mm,y=3mm]
  \tileA{300}{-1}{0}{};
  \tileA{240}{-1}{1}{};
  \tileA{0}{0}{-1}{};
  \tileA{0}{0}{0}{};
  \tileAr{180}{0}{1}{};
  \tileA{60}{1}{0}{};
  \tileA{120}{2}{-2}{};
\end{tikzpicture}%
\end{center}%
\end{minipage}%
} \\ \subfloat[$1$-patch 3 (central tile class $H_3$)]{%
\begin{minipage}[b]{6cm}
\begin{center}
\begin{tikzpicture}[x=3mm,y=3mm]
  \tileA{300}{-1}{0}{};
  \tileA{240}{-1}{1}{};
  \tileA{0}{0}{-1}{};
  \tileA{0}{0}{0}{};
  \tileAr{180}{0}{1}{};
  \tileA{60}{1}{0}{};
  \tileA{240}{2}{-1}{};
\end{tikzpicture}%
\end{center}%
\end{minipage}%
} \qquad \subfloat[$1$-patch 4 (central tile class $F_2$)]{%
\begin{minipage}[b]{6cm}
\begin{center}
\begin{tikzpicture}[x=3mm,y=3mm]
  \tileA{300}{-1}{1}{};
  \tileA{180}{0}{-1}{};
  \tileA{0}{0}{0}{};
  \tileA{0}{0}{1}{};
  \tileA{120}{1}{-2}{};
  \tileA{120}{2}{-2}{};
  \tileA{240}{2}{0}{};
\end{tikzpicture}%
\end{center}%
\end{minipage}%
} \\ \subfloat[$1$-patch 5 (central tile class $P_2$)]{%
\begin{minipage}[b]{6cm}
\begin{center}
\begin{tikzpicture}[x=3mm,y=3mm]
  \tileA{300}{-1}{1}{};
  \tileA{180}{0}{-1}{};
  \tileA{0}{0}{0}{};
  \tileA{0}{0}{1}{};
  \tileA{300}{1}{-1}{};
  \tileA{300}{2}{-1}{};
  \tileA{240}{2}{0}{};
\end{tikzpicture}%
\end{center}%
\end{minipage}%
} \qquad \subfloat[$1$-patch 6 (eliminated by trying to surround shaded tile)]{%
\begin{minipage}[b]{6cm}
\begin{center}
\begin{tikzpicture}[x=3mm,y=3mm]
  \tileA{300}{-1}{1}{};
  \tileA{180}{0}{-1}{};
  \tileA{0}{0}{0}{};
  \tileA{0}{0}{1}{};
  \tileA{300}{1}{-1}{};
  \fftileA{120}{2}{-1}{};
  \tileA{180}{3}{-2}{};
\end{tikzpicture}%
\end{center}%
\end{minipage}%
} \\ \subfloat[$1$-patch 7 (central tile class $H_2$)]{%
\begin{minipage}[b]{6cm}
\begin{center}
\begin{tikzpicture}[x=3mm,y=3mm]
  \tileA{60}{-1}{0}{};
  \tileA{240}{0}{-1}{};
  \tileA{0}{0}{0}{};
  \tileA{0}{0}{1}{};
  \tileAr{180}{1}{-1}{};
  \tileA{60}{2}{-2}{};
  \tileA{120}{2}{-1}{};
\end{tikzpicture}%
\end{center}%
\end{minipage}%
} \qquad \subfloat[$1$-patch 8 (central tile class $H_2$)]{%
\begin{minipage}[b]{6cm}
\begin{center}
\begin{tikzpicture}[x=3mm,y=3mm]
  \tileA{300}{-1}{1}{};
  \tileA{240}{0}{-1}{};
  \tileA{0}{0}{0}{};
  \tileA{0}{0}{1}{};
  \tileAr{180}{1}{-1}{};
  \tileA{60}{2}{-2}{};
  \tileA{120}{2}{-1}{};
\end{tikzpicture}%
\end{center}%
\end{minipage}%
}%
\end{center}
\caption{$1$-patches (part 1)}
\label{fig:patch}
\end{figure}
\begin{figure}[htp!]
\renewcommand\thesubfigure{\arabic{subfigure}}%
\ContinuedFloat
\captionsetup{margin=0pt,justification=raggedright}%
\begin{center}
\subfloat[$1$-patch 9 (central tile class $H_1$)]{%
\begin{minipage}[b]{6cm}
\begin{center}
\begin{tikzpicture}[x=3mm,y=3mm]
  \tileAr{60}{-1}{0}{};
  \tileAr{120}{0}{-1}{};
  \tileA{0}{0}{0}{};
  \tileAr{180}{0}{1}{};
  \tileAr{180}{1}{-1}{};
  \tileAr{300}{1}{0}{};
  \tileAr{240}{2}{-1}{};
\end{tikzpicture}%
\end{center}%
\end{minipage}%
} \qquad \subfloat[$1$-patch 10 (central tile class $H_4$)]{%
\begin{minipage}[b]{6cm}
\begin{center}
\begin{tikzpicture}[x=3mm,y=3mm]
  \tileA{60}{-1}{-1}{};
  \tileA{120}{-1}{0}{};
  \tileAr{300}{-1}{1}{};
  \tileA{0}{0}{-1}{};
  \tileA{0}{0}{0}{};
  \tileA{120}{0}{1}{};
  \tileA{60}{1}{0}{};
  \tileA{120}{2}{-2}{};
\end{tikzpicture}%
\end{center}%
\end{minipage}%
} \\ \subfloat[$1$-patch 11 (eliminated by trying to surround shaded tile)]{%
\begin{minipage}[b]{6cm}
\begin{center}
\begin{tikzpicture}[x=3mm,y=3mm]
  \tileA{60}{-1}{-1}{};
  \tileA{120}{-1}{0}{};
  \tileAr{300}{-1}{1}{};
  \fftileA{0}{0}{-1}{};
  \tileA{0}{0}{0}{};
  \tileA{120}{0}{1}{};
  \tileA{60}{1}{0}{};
  \tileA{240}{2}{-1}{};
\end{tikzpicture}%
\end{center}%
\end{minipage}%
} \qquad \subfloat[$1$-patch 12 (central tile class $FP_1$)]{%
\begin{minipage}[b]{6cm}
\begin{center}
\begin{tikzpicture}[x=3mm,y=3mm]
  \tileA{60}{-1}{-1}{};
  \tileA{300}{-1}{1}{};
  \tileAr{240}{-1}{2}{};
  \tileA{0}{0}{-1}{};
  \tileA{0}{0}{0}{};
  \tileA{60}{0}{1}{};
  \tileA{60}{1}{0}{};
  \tileA{120}{2}{-2}{};
\end{tikzpicture}%
\end{center}%
\end{minipage}%
} \\ \subfloat[$1$-patch 13 (central tile class $FP_1$)]{%
\begin{minipage}[b]{6cm}
\begin{center}
\begin{tikzpicture}[x=3mm,y=3mm]
  \tileA{60}{-1}{-1}{};
  \tileA{300}{-1}{1}{};
  \tileA{0}{0}{-1}{};
  \tileA{0}{0}{0}{};
  \tileAr{120}{0}{1}{};
  \tileA{300}{0}{2}{};
  \tileA{60}{1}{0}{};
  \tileA{120}{2}{-2}{};
\end{tikzpicture}%
\end{center}%
\end{minipage}%
} \qquad \subfloat[$1$-patch 14 (central tile class $F_2$)]{%
\begin{minipage}[b]{6cm}
\begin{center}
\begin{tikzpicture}[x=3mm,y=3mm]
  \tileA{60}{-1}{-1}{};
  \tileA{300}{-1}{1}{};
  \tileA{0}{0}{-1}{};
  \tileA{0}{0}{0}{};
  \tileAr{60}{0}{1}{};
  \tileA{240}{1}{1}{};
  \tileA{120}{2}{-2}{};
  \tileA{240}{2}{0}{};
\end{tikzpicture}%
\end{center}%
\end{minipage}%
}%
\end{center}
\caption{$1$-patches (part 2)}
\label{fig:patch:2}
\end{figure}
\begin{figure}[htp!]
\renewcommand\thesubfigure{\arabic{subfigure}}%
\ContinuedFloat
\captionsetup{margin=0pt,justification=raggedright}%
\begin{center}
\subfloat[$1$-patch 15 (eliminated by trying to surround shaded tile)]{%
\begin{minipage}[b]{6cm}
\begin{center}
\begin{tikzpicture}[x=3mm,y=3mm]
  \tileA{60}{-1}{-1}{};
  \tileA{300}{-1}{1}{};
  \tileAr{240}{-1}{2}{};
  \fftileA{0}{0}{-1}{};
  \tileA{0}{0}{0}{};
  \tileA{60}{0}{1}{};
  \tileA{60}{1}{0}{};
  \tileA{240}{2}{-1}{};
\end{tikzpicture}%
\end{center}%
\end{minipage}%
} \qquad \subfloat[$1$-patch 16 (eliminated by trying to surround shaded tile)]{%
\begin{minipage}[b]{6cm}
\begin{center}
\begin{tikzpicture}[x=3mm,y=3mm]
  \tileA{60}{-1}{-1}{};
  \tileA{300}{-1}{1}{};
  \fftileA{0}{0}{-1}{};
  \tileA{0}{0}{0}{};
  \tileAr{120}{0}{1}{};
  \tileA{300}{0}{2}{};
  \tileA{60}{1}{0}{};
  \tileA{240}{2}{-1}{};
\end{tikzpicture}%
\end{center}%
\end{minipage}%
} \\ \subfloat[$1$-patch 17 (eliminated by trying to surround shaded tile)]{%
\begin{minipage}[b]{6cm}
\begin{center}
\begin{tikzpicture}[x=3mm,y=3mm]
  \tileA{0}{-2}{1}{};
  \fftileA{300}{-1}{0}{};
  \tileAr{300}{-1}{1}{};
  \tileA{0}{0}{-1}{};
  \tileA{0}{0}{0}{};
  \tileA{120}{0}{1}{};
  \tileA{60}{1}{0}{};
  \tileA{120}{2}{-2}{};
\end{tikzpicture}%
\end{center}%
\end{minipage}%
} \qquad \subfloat[$1$-patch 18 (eliminated by trying to surround shaded tile)]{%
\begin{minipage}[b]{6cm}
\begin{center}
\begin{tikzpicture}[x=3mm,y=3mm]
  \tileA{0}{-2}{1}{};
  \fftileA{300}{-1}{0}{};
  \tileAr{300}{-1}{1}{};
  \tileA{0}{0}{-1}{};
  \tileA{0}{0}{0}{};
  \tileA{120}{0}{1}{};
  \tileA{60}{1}{0}{};
  \tileA{240}{2}{-1}{};
\end{tikzpicture}%
\end{center}%
\end{minipage}%
} \\ \subfloat[$1$-patch 19 (central tile class $H_4$)]{%
\begin{minipage}[b]{6cm}
\begin{center}
\begin{tikzpicture}[x=3mm,y=3mm]
  \tileA{120}{-1}{0}{};
  \tileAr{300}{-1}{1}{};
  \tileA{180}{0}{-1}{};
  \tileA{0}{0}{0}{};
  \tileA{120}{0}{1}{};
  \tileA{120}{1}{-2}{};
  \tileA{60}{1}{0}{};
  \tileA{120}{2}{-2}{};
\end{tikzpicture}%
\end{center}%
\end{minipage}%
} \qquad \subfloat[$1$-patch 20 (central tile class $H_4$)]{%
\begin{minipage}[b]{6cm}
\begin{center}
\begin{tikzpicture}[x=3mm,y=3mm]
  \tileA{120}{-1}{0}{};
  \tileAr{300}{-1}{1}{};
  \tileA{180}{0}{-1}{};
  \tileA{0}{0}{0}{};
  \tileA{120}{0}{1}{};
  \tileA{120}{1}{-2}{};
  \tileA{60}{1}{0}{};
  \tileA{240}{2}{-1}{};
\end{tikzpicture}%
\end{center}%
\end{minipage}%
}%
\end{center}
\caption{$1$-patches (part 3)}
\label{fig:patch:3}
\end{figure}
\begin{figure}[htp!]
\renewcommand\thesubfigure{\arabic{subfigure}}%
\ContinuedFloat
\captionsetup{margin=0pt,justification=raggedright}%
\begin{center}
\subfloat[$1$-patch 21 (eliminated by trying to surround shaded tile)]{%
\begin{minipage}[b]{6cm}
\begin{center}
\begin{tikzpicture}[x=3mm,y=3mm]
  \tileA{120}{-1}{0}{};
  \tileAr{300}{-1}{1}{};
  \tileA{180}{0}{-1}{};
  \tileA{0}{0}{0}{};
  \tileA{120}{0}{1}{};
  \tileA{300}{1}{-1}{};
  \fftileA{60}{1}{0}{};
  \tileA{300}{2}{-1}{};
\end{tikzpicture}%
\end{center}%
\end{minipage}%
} \qquad \subfloat[$1$-patch 22 (eliminated by trying to surround shaded tile)]{%
\begin{minipage}[b]{6cm}
\begin{center}
\begin{tikzpicture}[x=3mm,y=3mm]
  \tileA{120}{-1}{0}{};
  \fftileAr{300}{-1}{1}{};
  \tileA{180}{0}{-1}{};
  \tileA{0}{0}{0}{};
  \tileAr{0}{0}{1}{};
  \tileA{300}{1}{-1}{};
  \tileAr{300}{1}{0}{};
  \tileA{0}{2}{-1}{};
\end{tikzpicture}%
\end{center}%
\end{minipage}%
} \\ \subfloat[$1$-patch 23 (central tile class $FP_1$)]{%
\begin{minipage}[b]{6cm}
\begin{center}
\begin{tikzpicture}[x=3mm,y=3mm]
  \tileA{300}{-1}{1}{};
  \tileAr{240}{-1}{2}{};
  \tileA{180}{0}{-1}{};
  \tileA{0}{0}{0}{};
  \tileA{60}{0}{1}{};
  \tileA{120}{1}{-2}{};
  \tileA{60}{1}{0}{};
  \tileA{120}{2}{-2}{};
\end{tikzpicture}%
\end{center}%
\end{minipage}%
} \qquad \subfloat[$1$-patch 24 (central tile class $FP_1$)]{%
\begin{minipage}[b]{6cm}
\begin{center}
\begin{tikzpicture}[x=3mm,y=3mm]
  \tileA{300}{-1}{1}{};
  \tileA{180}{0}{-1}{};
  \tileA{0}{0}{0}{};
  \tileAr{120}{0}{1}{};
  \tileA{300}{0}{2}{};
  \tileA{120}{1}{-2}{};
  \tileA{60}{1}{0}{};
  \tileA{120}{2}{-2}{};
\end{tikzpicture}%
\end{center}%
\end{minipage}%
} \\ \subfloat[$1$-patch 25 (central tile class $F_2$)]{%
\begin{minipage}[b]{6cm}
\begin{center}
\begin{tikzpicture}[x=3mm,y=3mm]
  \tileA{300}{-1}{1}{};
  \tileA{180}{0}{-1}{};
  \tileA{0}{0}{0}{};
  \tileAr{60}{0}{1}{};
  \tileA{120}{1}{-2}{};
  \tileA{240}{1}{1}{};
  \tileA{120}{2}{-2}{};
  \tileA{240}{2}{0}{};
\end{tikzpicture}%
\end{center}%
\end{minipage}%
} \qquad \subfloat[$1$-patch 26 (central tile class $FP_1$)]{%
\begin{minipage}[b]{6cm}
\begin{center}
\begin{tikzpicture}[x=3mm,y=3mm]
  \tileA{300}{-1}{1}{};
  \tileAr{240}{-1}{2}{};
  \tileA{180}{0}{-1}{};
  \tileA{0}{0}{0}{};
  \tileA{60}{0}{1}{};
  \tileA{120}{1}{-2}{};
  \tileA{60}{1}{0}{};
  \tileA{240}{2}{-1}{};
\end{tikzpicture}%
\end{center}%
\end{minipage}%
}%
\end{center}
\caption{$1$-patches (part 4)}
\label{fig:patch:4}
\end{figure}
\begin{figure}[htp!]
\renewcommand\thesubfigure{\arabic{subfigure}}%
\ContinuedFloat
\captionsetup{margin=0pt,justification=raggedright}%
\begin{center}
\subfloat[$1$-patch 27 (central tile class $FP_1$)]{%
\begin{minipage}[b]{6cm}
\begin{center}
\begin{tikzpicture}[x=3mm,y=3mm]
  \tileA{300}{-1}{1}{};
  \tileA{180}{0}{-1}{};
  \tileA{0}{0}{0}{};
  \tileAr{120}{0}{1}{};
  \tileA{300}{0}{2}{};
  \tileA{120}{1}{-2}{};
  \tileA{60}{1}{0}{};
  \tileA{240}{2}{-1}{};
\end{tikzpicture}%
\end{center}%
\end{minipage}%
} \qquad \subfloat[$1$-patch 28 (eliminated by trying to surround shaded tile)]{%
\begin{minipage}[b]{6cm}
\begin{center}
\begin{tikzpicture}[x=3mm,y=3mm]
  \tileA{300}{-1}{1}{};
  \tileA{180}{0}{-1}{};
  \tileA{0}{0}{0}{};
  \tileAr{120}{0}{1}{};
  \fftileA{300}{1}{-1}{};
  \tileAr{300}{1}{0}{};
  \tileA{180}{1}{1}{};
  \tileA{0}{2}{-1}{};
\end{tikzpicture}%
\end{center}%
\end{minipage}%
} \\ \subfloat[$1$-patch 29 (central tile class $T_1$)]{%
\begin{minipage}[b]{6cm}
\begin{center}
\begin{tikzpicture}[x=3mm,y=3mm]
  \tileA{300}{-1}{1}{};
  \tileAr{240}{-1}{2}{};
  \tileA{180}{0}{-1}{};
  \tileA{0}{0}{0}{};
  \tileA{60}{0}{1}{};
  \tileA{300}{1}{-1}{};
  \tileA{60}{1}{0}{};
  \tileA{300}{2}{-1}{};
\end{tikzpicture}%
\end{center}%
\end{minipage}%
} \qquad \subfloat[$1$-patch 30 (central tile class $T_1$)]{%
\begin{minipage}[b]{6cm}
\begin{center}
\begin{tikzpicture}[x=3mm,y=3mm]
  \tileA{300}{-1}{1}{};
  \tileA{180}{0}{-1}{};
  \tileA{0}{0}{0}{};
  \tileAr{120}{0}{1}{};
  \tileA{300}{0}{2}{};
  \tileA{300}{1}{-1}{};
  \tileA{60}{1}{0}{};
  \tileA{300}{2}{-1}{};
\end{tikzpicture}%
\end{center}%
\end{minipage}%
} \\ \subfloat[$1$-patch 31 (central tile class $P_2$)]{%
\begin{minipage}[b]{6cm}
\begin{center}
\begin{tikzpicture}[x=3mm,y=3mm]
  \tileA{300}{-1}{1}{};
  \tileA{180}{0}{-1}{};
  \tileA{0}{0}{0}{};
  \tileAr{60}{0}{1}{};
  \tileA{300}{1}{-1}{};
  \tileA{240}{1}{1}{};
  \tileA{300}{2}{-1}{};
  \tileA{240}{2}{0}{};
\end{tikzpicture}%
\end{center}%
\end{minipage}%
} \qquad \subfloat[$1$-patch 32 (central tile class $P_2$)]{%
\begin{minipage}[b]{6cm}
\begin{center}
\begin{tikzpicture}[x=3mm,y=3mm]
  \tileA{300}{-1}{1}{};
  \tileA{180}{0}{-1}{};
  \tileA{0}{0}{0}{};
  \tileAr{60}{0}{1}{};
  \tileA{300}{1}{-1}{};
  \tileA{240}{1}{1}{};
  \tileA{120}{2}{-1}{};
  \tileA{180}{3}{-2}{};
\end{tikzpicture}%
\end{center}%
\end{minipage}%
}%
\end{center}
\caption{$1$-patches (part 5)}
\label{fig:patch:5}
\end{figure}
\begin{figure}[htp!]
\renewcommand\thesubfigure{\arabic{subfigure}}%
\ContinuedFloat
\captionsetup{margin=0pt,justification=raggedright}%
\begin{center}
\subfloat[$1$-patch 33 (central tile class $H_2$)]{%
\begin{minipage}[b]{6cm}
\begin{center}
\begin{tikzpicture}[x=3mm,y=3mm]
  \tileA{60}{-1}{0}{};
  \tileA{240}{0}{-1}{};
  \tileA{0}{0}{0}{};
  \tileA{180}{0}{1}{};
  \tileAr{180}{1}{-1}{};
  \tileA{240}{1}{1}{};
  \tileA{60}{2}{-2}{};
  \tileA{120}{2}{-1}{};
\end{tikzpicture}%
\end{center}%
\end{minipage}%
} \qquad \subfloat[$1$-patch 34 (central tile class $H_2$)]{%
\begin{minipage}[b]{6cm}
\begin{center}
\begin{tikzpicture}[x=3mm,y=3mm]
  \tileA{300}{-1}{1}{};
  \tileA{240}{0}{-1}{};
  \tileA{0}{0}{0}{};
  \tileAr{60}{0}{1}{};
  \tileAr{180}{1}{-1}{};
  \tileA{240}{1}{1}{};
  \tileA{60}{2}{-2}{};
  \tileA{120}{2}{-1}{};
\end{tikzpicture}%
\end{center}%
\end{minipage}%
} \\ \subfloat[$1$-patch 35 (eliminated by trying to surround shaded tile)]{%
\begin{minipage}[b]{6cm}
\begin{center}
\begin{tikzpicture}[x=3mm,y=3mm]
  \tileA{120}{-1}{0}{};
  \fftileAr{300}{-1}{1}{};
  \tileA{240}{0}{-1}{};
  \tileA{0}{0}{0}{};
  \tileAr{0}{0}{1}{};
  \tileAr{180}{1}{-1}{};
  \tileAr{300}{1}{0}{};
  \tileAr{240}{2}{-1}{};
\end{tikzpicture}%
\end{center}%
\end{minipage}%
} \qquad \subfloat[$1$-patch 36 (eliminated by trying to surround shaded tile)]{%
\begin{minipage}[b]{6cm}
\begin{center}
\begin{tikzpicture}[x=3mm,y=3mm]
  \fftileA{300}{-1}{1}{};
  \tileA{240}{0}{-1}{};
  \tileA{0}{0}{0}{};
  \tileAr{120}{0}{1}{};
  \tileAr{180}{1}{-1}{};
  \tileAr{300}{1}{0}{};
  \tileA{180}{1}{1}{};
  \tileAr{240}{2}{-1}{};
\end{tikzpicture}%
\end{center}%
\end{minipage}%
} \\ \subfloat[$1$-patch 37 (eliminated by trying to surround shaded tile)]{%
\begin{minipage}[b]{6cm}
\begin{center}
\begin{tikzpicture}[x=3mm,y=3mm]
  \tileA{180}{-1}{0}{};
  \fftileAr{300}{-1}{1}{};
  \tileAr{120}{0}{-1}{};
  \tileA{0}{0}{0}{};
  \tileAr{0}{0}{1}{};
  \tileAr{180}{1}{-1}{};
  \tileAr{300}{1}{0}{};
  \tileAr{240}{2}{-1}{};
\end{tikzpicture}%
\end{center}%
\end{minipage}%
}%
\end{center}
\caption{$1$-patches (part 6)}
\label{fig:patch:6}
\end{figure}

\FloatBarrier

\subsection{Classification of outer tiles}

For each of the possible neighbours that actually occurs in some of
the remaining $1$-patches, we can now list the possible
classifications of a central tile that has such a neighbour; see
Table~\ref{table:nbrclass}.

For each of the outer tiles in a $1$-patch, we have some but not all
of its neighbours, and can take the intersection of the sets from
Table~\ref{table:nbrclass} to produce a set of possible classes for
that outer tile.  Although this is not a single class, it can still be
used for the within-cluster and between-cluster checks.  In each case,
it turns out that the set of possible classes for a neighbour
appearing in one of those checks is a subset of the classes permitted
by that check, and so we have a complete proof of the within-cluster
and between-cluster matching properties that depends only on the
enumeration of $1$-patches presented here and not on a larger
enumeration of $2$-patches; the lists of checks and corresponding sets
of classes appear below.

\begin{table}[htp!]
\begin{center}
\begin{tabular}{|c|c|}
\textbf{Possible neighbour} & \textbf{Possible classes for central tile}\\
$2$ & $\{FP_1, F_2, H_4\}$\\
$3$ & $\{H_2\}$\\
$4$ & $\{H_4\}$\\
$6$ & $\{H_3\}$\\
$7$ & $\{H_1\}$\\
$8$ & $\{H_3\}$\\
$9$ & $\{FP_1, F_2, H_2, P_2, T_1\}$\\
$11$ & $\{H_4\}$\\
$12$ & $\{FP_1, T_1\}$\\
$13$ & $\{FP_1, F_2, H_3, H_4\}$\\
$14$ & $\{FP_1, F_2, H_4, P_2, T_1\}$\\
$15$ & $\{H_2\}$\\
$16$ & $\{H_1\}$\\
$17$ & $\{F_2, H_2, P_2\}$\\
$18$ & $\{FP_1, T_1\}$\\
$19$ & $\{H_4\}$\\
$20$ & $\{H_2\}$\\
$22$ & $\{F_2, H_2, P_2\}$\\
$23$ & $\{FP_1, T_1\}$\\
$24$ & $\{H_1, H_3\}$\\
$25$ & $\{FP_1, T_1\}$\\
$26$ & $\{FP_1, F_2, H_4\}$\\
$27$ & $\{P_2, T_1\}$\\
$28$ & $\{H_1, H_2\}$\\
$29$ & $\{FP_1, H_3, H_4, T_1\}$\\
$30$ & $\{H_1\}$\\
$32$ & $\{F_2, H_2, P_2\}$\\
$33$ & $\{H_2\}$\\
$34$ & $\{FP_1, F_2, H_3, H_4\}$\\
$36$ & $\{H_2, P_2\}$\\
$37$ & $\{FP_1, H_3, H_4\}$\\
$38$ & $\{P_2, T_1\}$\\
$39$ & $\{H_1\}$\\
$40$ & $\{F_2, P_2\}$\\
$41$ & $\{P_2\}$\\
\end{tabular}
\caption{}
\label{table:nbrclass}
\end{center}
\end{table}
\begin{itemize}
\item $1$-patch 1 (class $F_2$)
\begin{itemize}
\item $P_2$ or $F_2$ neighbour $FP_1$ OK: $\{FP_1\} \subseteq \{FP_1\}$
\item $F$ edge $F^+$ OK: $\{F_2\} \subseteq \{F_2\}$
\item $F$ edge $F^-$ OK: $\{F_2\} \subseteq \{F_2\}$
\item $X^+$ edge at top of polykite OK: $\{H_3\} \subseteq \{F_2, FP_1, H_3, H_4\}$
\item $X^-$ edge at bottom of polykite OK: $\{H_2\} \subseteq \{F_2, H_2, P_2\}$
\item $L$ edge at bottom of polykite OK: $\{P_2\} \subseteq \{P_2\}$
\end{itemize}
\item $1$-patch 2 (class $H_3$)
\begin{itemize}
\item $H_3$ neighbour $H_1$ OK: $\{H_1\} \subseteq \{H_1\}$
\item $H$ lower edge $B^-$ OK: $\{FP_1, T_1\} \subseteq \{FP_1, T_1\}$
\item $X^+$ edge at right of polykite OK: $\{P_2\} \subseteq \{H_2, P_2\}$
\item $X^-$ edge at bottom of polykite OK: $\{H_2, P_2\} \subseteq \{F_2, H_2, P_2\}$
\end{itemize}
\item $1$-patch 3 (class $H_3$)
\begin{itemize}
\item $H_3$ neighbour $H_1$ OK: $\{H_1\} \subseteq \{H_1\}$
\item $H$ lower edge $B^-$ OK: $\{FP_1, T_1\} \subseteq \{FP_1, T_1\}$
\item $X^+$ edge at right of polykite OK: $\{F_2, FP_1, H_4\} \subseteq \{F_2, FP_1, H_3, H_4\}$
\item $X^-$ edge at bottom of polykite OK: $\{F_2, P_2\} \subseteq \{F_2, H_2, P_2\}$
\end{itemize}
\item $1$-patch 4 (class $F_2$)
\begin{itemize}
\item $P_2$ or $F_2$ neighbour $FP_1$ OK: $\{FP_1\} \subseteq \{FP_1\}$
\item $F$ edge $F^+$ OK: $\{F_2\} \subseteq \{F_2\}$
\item $F$ edge $F^-$ OK: $\{F_2\} \subseteq \{F_2\}$
\item $X^+$ edge at top of polykite OK: $\{H_3\} \subseteq \{F_2, FP_1, H_3, H_4\}$
\item $X^-$ edge at bottom of polykite OK: $\{FP_1, H_3, H_4\} \subseteq \{FP_1, H_3, H_4\}$
\item $L$ edge at bottom of polykite OK: $\{F_2, FP_1\} \subseteq \{F_2, FP_1\}$
\end{itemize}
\item $1$-patch 5 (class $P_2$)
\begin{itemize}
\item $P_2$ or $F_2$ neighbour $FP_1$ OK: $\{FP_1\} \subseteq \{FP_1\}$
\item $T$ or $P$ lower edge $A^-$ OK: $\{H_2\} \subseteq \{H_2\}$
\item $X^+$ edge at top of polykite OK: $\{H_3\} \subseteq \{F_2, FP_1, H_3, H_4\}$
\item $X^-$ edge at right of polykite OK: $\{FP_1, H_4\} \subseteq \{FP_1, H_3, H_4\}$
\item $L$ edge at right of polykite OK: $\{F_2, FP_1\} \subseteq \{F_2, FP_1\}$
\end{itemize}
\item $1$-patch 7 (class $H_2$)
\begin{itemize}
\item $H_2$ neighbour $H_1$ OK: $\{H_1\} \subseteq \{H_1\}$
\item $H$ edge $A^+$ OK: $\{P_2, T_1\} \subseteq \{P_2, T_1\}$
\item $X^+$ edge at top of polykite OK: $\{F_2, FP_1, H_4\} \subseteq \{F_2, FP_1, H_3, H_4\}$
\item $X^-$ edge at right of polykite OK: $\{F_2, P_2\} \subseteq \{F_2, H_2, P_2\}$
\end{itemize}
\item $1$-patch 8 (class $H_2$)
\begin{itemize}
\item $H_2$ neighbour $H_1$ OK: $\{H_1\} \subseteq \{H_1\}$
\item $H$ edge $A^+$ OK: $\{T_1\} \subseteq \{T_1\}$
\item $X^+$ edge at top of polykite OK: $\{H_3\} \subseteq \{F_2, FP_1, H_3, H_4\}$
\item $X^-$ edge at right of polykite OK: $\{F_2, P_2\} \subseteq \{F_2, H_2, P_2\}$
\end{itemize}
\item $1$-patch 9 (class $H_1$)
\begin{itemize}
\item $H_1$ neighbour $H_2$ OK: $\{H_2\} \subseteq \{H_2\}$
\item $H_1$ neighbour $H_3$ OK: $\{H_3\} \subseteq \{H_3\}$
\item $H_1$ neighbour $H_4$ OK: $\{H_4\} \subseteq \{H_4\}$
\item $H$ upper edge $B^-$ OK: $\{FP_1, T_1\} \subseteq \{FP_1, T_1\}$
\end{itemize}
\item $1$-patch 10 (class $H_4$)
\begin{itemize}
\item $H_4$ neighbour $H_1$ OK: $\{H_1\} \subseteq \{H_1\}$
\item $X^+$ edge at right of polykite OK: $\{P_2\} \subseteq \{H_2, P_2\}$
\item $X^-$ edge at bottom of polykite OK: $\{H_2\} \subseteq \{F_2, H_2, P_2\}$
\end{itemize}
\item $1$-patch 12 (class $FP_1$)
\begin{itemize}
\item $FP_1$ neighbour $P_2$ or $F_2$ OK: $\{F_2\} \subseteq \{F_2, P_2\}$
\item $T$, $P$ or $F$ edge $B^+$ OK: $\{H_4\} \subseteq \{H_3, H_4\}$
\item $X^+$ edge at right of polykite OK: $\{P_2\} \subseteq \{H_2, P_2\}$
\item $X^-$ edge at bottom of polykite OK: $\{H_2\} \subseteq \{F_2, H_2, P_2\}$
\item $L$ edge at bottom of polykite OK: $\{P_2\} \subseteq \{P_2\}$
\end{itemize}
\item $1$-patch 13 (class $FP_1$)
\begin{itemize}
\item $FP_1$ neighbour $P_2$ or $F_2$ OK: $\{F_2\} \subseteq \{F_2, P_2\}$
\item $T$, $P$ or $F$ edge $B^+$ OK: $\{H_3\} \subseteq \{H_3, H_4\}$
\item $X^+$ edge at right of polykite OK: $\{P_2\} \subseteq \{H_2, P_2\}$
\item $X^-$ edge at bottom of polykite OK: $\{H_2\} \subseteq \{F_2, H_2, P_2\}$
\item $L$ edge at bottom of polykite OK: $\{P_2\} \subseteq \{P_2\}$
\end{itemize}
\item $1$-patch 14 (class $F_2$)
\begin{itemize}
\item $P_2$ or $F_2$ neighbour $FP_1$ OK: $\{FP_1\} \subseteq \{FP_1\}$
\item $F$ edge $F^+$ OK: $\{F_2\} \subseteq \{F_2\}$
\item $F$ edge $F^-$ OK: $\{F_2\} \subseteq \{F_2\}$
\item $X^+$ edge at top of polykite OK: $\{H_2\} \subseteq \{H_2, P_2\}$
\item $X^-$ edge at bottom of polykite OK: $\{H_2\} \subseteq \{F_2, H_2, P_2\}$
\item $L$ edge at bottom of polykite OK: $\{P_2\} \subseteq \{P_2\}$
\end{itemize}
\item $1$-patch 19 (class $H_4$)
\begin{itemize}
\item $H_4$ neighbour $H_1$ OK: $\{H_1\} \subseteq \{H_1\}$
\item $X^+$ edge at right of polykite OK: $\{P_2\} \subseteq \{H_2, P_2\}$
\item $X^-$ edge at bottom of polykite OK: $\{FP_1, H_3, H_4\} \subseteq \{FP_1, H_3, H_4\}$
\end{itemize}
\item $1$-patch 20 (class $H_4$)
\begin{itemize}
\item $H_4$ neighbour $H_1$ OK: $\{H_1\} \subseteq \{H_1\}$
\item $X^+$ edge at right of polykite OK: $\{FP_1, H_4\} \subseteq \{F_2, FP_1, H_3, H_4\}$
\item $X^-$ edge at bottom of polykite OK: $\{FP_1, H_4\} \subseteq \{FP_1, H_3, H_4\}$
\end{itemize}
\item $1$-patch 23 (class $FP_1$)
\begin{itemize}
\item $FP_1$ neighbour $P_2$ or $F_2$ OK: $\{F_2\} \subseteq \{F_2, P_2\}$
\item $T$, $P$ or $F$ edge $B^+$ OK: $\{H_4\} \subseteq \{H_3, H_4\}$
\item $X^+$ edge at right of polykite OK: $\{P_2\} \subseteq \{H_2, P_2\}$
\item $X^-$ edge at bottom of polykite OK: $\{FP_1, H_3, H_4\} \subseteq \{FP_1, H_3, H_4\}$
\item $L$ edge at bottom of polykite OK: $\{F_2, FP_1\} \subseteq \{F_2, FP_1\}$
\end{itemize}
\item $1$-patch 24 (class $FP_1$)
\begin{itemize}
\item $FP_1$ neighbour $P_2$ or $F_2$ OK: $\{F_2\} \subseteq \{F_2, P_2\}$
\item $T$, $P$ or $F$ edge $B^+$ OK: $\{H_3\} \subseteq \{H_3, H_4\}$
\item $X^+$ edge at right of polykite OK: $\{P_2\} \subseteq \{H_2, P_2\}$
\item $X^-$ edge at bottom of polykite OK: $\{FP_1, H_3, H_4\} \subseteq \{FP_1, H_3, H_4\}$
\item $L$ edge at bottom of polykite OK: $\{F_2, FP_1\} \subseteq \{F_2, FP_1\}$
\end{itemize}
\item $1$-patch 25 (class $F_2$)
\begin{itemize}
\item $P_2$ or $F_2$ neighbour $FP_1$ OK: $\{FP_1\} \subseteq \{FP_1\}$
\item $F$ edge $F^+$ OK: $\{F_2\} \subseteq \{F_2\}$
\item $F$ edge $F^-$ OK: $\{F_2\} \subseteq \{F_2\}$
\item $X^+$ edge at top of polykite OK: $\{H_2\} \subseteq \{H_2, P_2\}$
\item $X^-$ edge at bottom of polykite OK: $\{FP_1, H_3, H_4\} \subseteq \{FP_1, H_3, H_4\}$
\item $L$ edge at bottom of polykite OK: $\{F_2, FP_1\} \subseteq \{F_2, FP_1\}$
\end{itemize}
\item $1$-patch 26 (class $FP_1$)
\begin{itemize}
\item $FP_1$ neighbour $P_2$ or $F_2$ OK: $\{F_2, P_2\} \subseteq \{F_2, P_2\}$
\item $T$, $P$ or $F$ edge $B^+$ OK: $\{H_4\} \subseteq \{H_3, H_4\}$
\item $X^+$ edge at right of polykite OK: $\{FP_1, H_4\} \subseteq \{F_2, FP_1, H_3, H_4\}$
\item $X^-$ edge at bottom of polykite OK: $\{FP_1, H_4\} \subseteq \{FP_1, H_3, H_4\}$
\item $L$ edge at bottom of polykite OK: $\{F_2, FP_1\} \subseteq \{F_2, FP_1\}$
\end{itemize}
\item $1$-patch 27 (class $FP_1$)
\begin{itemize}
\item $FP_1$ neighbour $P_2$ or $F_2$ OK: $\{F_2, P_2\} \subseteq \{F_2, P_2\}$
\item $T$, $P$ or $F$ edge $B^+$ OK: $\{H_3\} \subseteq \{H_3, H_4\}$
\item $X^+$ edge at right of polykite OK: $\{FP_1, H_4\} \subseteq \{F_2, FP_1, H_3, H_4\}$
\item $X^-$ edge at bottom of polykite OK: $\{FP_1, H_4\} \subseteq \{FP_1, H_3, H_4\}$
\item $L$ edge at bottom of polykite OK: $\{F_2, FP_1\} \subseteq \{F_2, FP_1\}$
\end{itemize}
\item $1$-patch 29 (class $T_1$)
\begin{itemize}
\item $T$ upper edge $A^-$ OK: $\{H_2\} \subseteq \{H_2\}$
\item $T$ or $P$ lower edge $A^-$ OK: $\{H_2\} \subseteq \{H_2\}$
\item $T$, $P$ or $F$ edge $B^+$ OK: $\{H_4\} \subseteq \{H_3, H_4\}$
\end{itemize}
\item $1$-patch 30 (class $T_1$)
\begin{itemize}
\item $T$ upper edge $A^-$ OK: $\{H_2\} \subseteq \{H_2\}$
\item $T$ or $P$ lower edge $A^-$ OK: $\{H_2\} \subseteq \{H_2\}$
\item $T$, $P$ or $F$ edge $B^+$ OK: $\{H_3\} \subseteq \{H_3, H_4\}$
\end{itemize}
\item $1$-patch 31 (class $P_2$)
\begin{itemize}
\item $P_2$ or $F_2$ neighbour $FP_1$ OK: $\{FP_1\} \subseteq \{FP_1\}$
\item $T$ or $P$ lower edge $A^-$ OK: $\{H_2\} \subseteq \{H_2\}$
\item $X^+$ edge at top of polykite OK: $\{H_2\} \subseteq \{H_2, P_2\}$
\item $X^-$ edge at right of polykite OK: $\{FP_1, H_3, H_4\} \subseteq \{FP_1, H_3, H_4\}$
\item $L$ edge at right of polykite OK: $\{F_2, FP_1\} \subseteq \{F_2, FP_1\}$
\end{itemize}
\item $1$-patch 32 (class $P_2$)
\begin{itemize}
\item $P_2$ or $F_2$ neighbour $FP_1$ OK: $\{FP_1\} \subseteq \{FP_1\}$
\item $T$ or $P$ lower edge $A^-$ OK: $\{H_2\} \subseteq \{H_2\}$
\item $X^+$ edge at top of polykite OK: $\{H_2\} \subseteq \{H_2, P_2\}$
\item $X^-$ edge at right of polykite OK: $\{H_2\} \subseteq \{F_2, H_2, P_2\}$
\item $L$ edge at right of polykite OK: $\{P_2\} \subseteq \{P_2\}$
\end{itemize}
\item $1$-patch 33 (class $H_2$)
\begin{itemize}
\item $H_2$ neighbour $H_1$ OK: $\{H_1\} \subseteq \{H_1\}$
\item $H$ edge $A^+$ OK: $\{P_2\} \subseteq \{P_2, T_1\}$
\item $X^+$ edge at top of polykite OK: $\{P_2\} \subseteq \{H_2, P_2\}$
\item $X^-$ edge at right of polykite OK: $\{H_2, P_2\} \subseteq \{F_2, H_2, P_2\}$
\end{itemize}
\item $1$-patch 34 (class $H_2$)
\begin{itemize}
\item $H_2$ neighbour $H_1$ OK: $\{H_1\} \subseteq \{H_1\}$
\item $H$ edge $A^+$ OK: $\{T_1\} \subseteq \{T_1\}$
\item $X^+$ edge at top of polykite OK: $\{H_2\} \subseteq \{H_2, P_2\}$
\item $X^-$ edge at right of polykite OK: $\{H_2, P_2\} \subseteq \{F_2, H_2, P_2\}$
\end{itemize}
\end{itemize}

\FloatBarrier


%
%

\bibliographystyle{alphaurl}
\bibliography{tilings}

\end{document}